\documentclass[a4paper, 10pt]{elsarticle}%
\usepackage[a4paper, headheight = 2.0cm, tmargin=2.0cm, bmargin = 2.0cm, hmargin={1.5cm,1.5cm} ]{geometry}
\usepackage{amsmath}
\usepackage{amsfonts}
\usepackage{amssymb, amsthm}
\usepackage{graphicx}
\usepackage[unicode]{hyperref}
\usepackage{color}%
\providecommand{\U}[1]{\protect\rule{.1in}{.1in}}
\newtheorem{theorem}{Theorem}[section]

\newtheorem{corollary}{Corollary}[section]

\newtheorem{definition}{Definition}[section]
\newtheorem{example}{Example}[section]

\newtheorem{lemma}{Lemma}[section]

\newtheorem{question}{Question}[section]
\newtheorem{proposition}{Proposition}[section]
\newtheorem{remark}{Remark}[section]

\hypersetup{
colorlinks = true,citecolor = blue,
filecolor=black,linkcolor = red,anchorcolor = red,
pagecolor = red,
urlcolor= red
}
\allowdisplaybreaks

\newcommand{\Red}[1]{{\color[rgb]{1.0, 0.0, 0.0}{#1}}}
\newcommand{\Green}[1]{{\color[rgb]{0.0, 0.4, 0.0}{#1}}}
\newcommand{\Blue}[1]{{\color[rgb]{0.0, 0.0, 0.7}{#1}}}
\newcommand{\Orange}[1]{{\color[rgb]{1.0, 0.6, 0.0}{#1}}}

\makeatletter
\def\ps@pprintTitle{%
 \let\@oddhead\@empty
 \let\@evenhead\@empty
 \def\@oddfoot{}%
 \let\@evenfoot\@oddfoot}
\makeatother

\begin{document}

\setcounter{tocdepth}{3}

\begin{frontmatter}

\title{\textbf{A constructive approach to triangular trigonometric patches}}

\author[ra]{\'Agoston R\'oth\corref{cor1}}
\ead{agoston\_roth@yahoo.com}

\author[ji]{Imre Juh\'asz}
\author[ka]{Alexandru Krist\'aly}
\address[ra]{Department of Mathematics and Computer Science, Babe\c{s}--Bolyai University, RO--400084 Cluj-Napoca, Romania}
\address[ji]{Department of Descriptive Geometry, University of Miskolc, H-3515 Miskolc-Egyetemv\'aros, Hungary}
\address[ka]{Department of Economics, Babe\c{s}--Bolyai University, RO--400591 Cluj-Napoca, Romania}
\cortext[cor1]{Corresponding author.}

%

\begin{abstract}
We construct a constrained trivariate extension of the univariate
normalized B-basis of the vector space of trigonometric polynomials of
arbitrary (finite) order $n\in%
\mathbb{N}
$ defined on any compact interval $\left[  0,\alpha\right]  $, where
$\alpha\in\left(  0,\pi\right)  $. Our triangular extension is a normalized
linearly independent constrained trivariate trigonometric function system of
dimension $\delta_{n}=3n\left(  n+1\right)  +1$ that spans the same vector
space of functions as the constrained trivariate extension of the canonical
basis of truncated Fourier series of order $n$ over $\left[  0,\alpha\right]
$. Although the explicit general basis transformation is yet unknown, the
coincidence of these vector spaces is proved by means of an appropriate
equivalence relation. As a possible application of our triangular extension, we introduce the notion of (rational) triangular trigonometric patches of order $n$ and of singularity free parametrization that could be used as control point based modeling tools in CAGD.
\end{abstract}



%
%
%
%

\end{frontmatter}

%

\makeatletter\def\overfrown#1{\mathop{\vbox{\m@th\ialign{##\crcr\noalign
{\kern3\p@}
\downfrownfill\crcr\noalign{\kern3\p@\nointerlineskip}
$\hfil\displaystyle{#1}\hfil$\crcr}}}\limits}
\def\downfrownfill{$\m@th\setbox\z@\hbox{$\braceld$}  \braceld\leaders\vrule
height\ht\z@ depth\z@\hfill\bracerd$}
\makeatother

\section{Introduction}

Triangular (rational) polynomial (spline) surfaces like the Coons \cite{Barnhill1973}, the Bernstein-B\'ezier \cite{deCasteljau1963, Barnhill1985, Farin1986} and the B-spline \cite{Dahmen1992} triangular patches (and their rational variants) form an important aspect of CAGD.
However, in order to provide control point based exact description of triangular parts of certain well-known surfaces, the rational form of these triangular patches should be used, which implies an undesired complexity (consider for example the evaluation of higher order (partial) derivatives, or the usage of non-negative special weights associated with control points that may be unknown for the designer) and in most cases provides a parametrization that does not reflect the variation of certain inner geometric properties (like curvature distribution) along the surface. One possibility to overcome some of these shortcomings is to consider the normalized B-basis of other (non-polynomial) vector spaces of functions that include the desired surfaces without the need of rational forms.

Non-polynomial surfaces like the first and second order
triangular trigonometric patches and fourth order algebraic trigonometric ones
were initiated by W.-Q. Shen, G.-Z. Wang and Y.-W. Wei in recent papers
\cite{Wang2010a,Wang2010b} and \cite{Wei2011}, respectively. These special
cases of triangular patches were obtained by certain constrained trivariate extensions of
univariate normalized B-bases of the first and second order trigonometric and of the
fourth order algebraic trigonometric vector spaces%
\begin{align*}
\mathcal{F}_{2}^{\alpha} &  =%
\operatorname{span}%
\left\{  1,\cos\left(  t\right)  ,\sin\left(  t\right)  :t\in\left[
0,\alpha\right]  \right\}  ,\\
\mathcal{F}_{4}^{\alpha} &  =%
\operatorname{span}%
\left\{  1,\cos\left(  t\right)  ,\sin\left(  t\right)  ,\cos\left(
2t\right)  ,\sin\left(  2t\right)  :t\in\left[  0,\alpha\right]  \right\}
\end{align*}
and%
\[
\mathcal{M}_{4}^{\alpha}=%
\operatorname{span}%
\left\{  1,t,\cos\left(  t\right)  ,\sin\left(  t\right)  :t\in\left[
0,\alpha\right]  \right\}  ,
\]
respectively, where $\alpha\in\left(  0,\pi\right)  $ is an arbitrarily fixed shape (or
design) parameter. The authors of the cited papers referred to their
results as triangular Bernstein or B\'{e}zier-like (algebraic) trigonometric
extensions and patches.

We restrict our
attention only to the normalized B-basis of the vector space%
\begin{equation}
\mathcal{F}_{2n}^{\alpha}=%
\operatorname{span}%
\left\{  \cos\left(  it\right)  ,\sin\left(  it\right)  :t\in\left[
0,\alpha\right]  \right\}  _{i=0}^{n} \label{truncated_Fourier_vector_space}%
\end{equation}
of truncated Fourier series (i.e., trigonometric polynomials of order at most
$n\in%
\mathbb{N}
$).

It was shown in \cite{Pena1997} that vector space
(\ref{truncated_Fourier_vector_space}) has no normalized totally positive
bases when $\alpha\geq\pi$, i.e., in this case it does not provide shape
preserving representations using control polygons. Thus, it is crucial for the
shape parameter $\alpha$ to determine an interval of length strictly less than
$\pi$.

The normalized B-basis of the vector space (\ref{truncated_Fourier_vector_space}) was introduced by J. S\'{a}nchez-Reyes in \cite{SReyes1998}. A linear reparametrization of his function system can be written in the form%
\begin{equation}
\left\{  A_{2n,i}^{\alpha}\left(  t\right)  :t\in\left[  0,\alpha\right]
\right\}  _{i=0}^{2n}=\left\{  c_{2n,i}^{\alpha}\sin^{2n-i}
\frac{\alpha-t}{2}  \sin^{i}\frac{t}{2}:t\in\left[  0,\alpha\right]
\right\}  _{i=0}^{2n}, \label{univariate_basis}%
\end{equation}
where the normalizing non-negative coefficients%
\[
c_{2n,i}^{\alpha}=\frac{1}{\sin^{2n}\frac{\alpha}{2}}\sum_{r=0}^{\left\lfloor
\frac{i}{2}\right\rfloor }\binom{n}{i-r}\binom{i-r}{r}\left(  2\cos
\frac{\alpha}{2}\right)  ^{i-2r},~i=0,1,\ldots,2n
\]
fulfill the symmetry%
\begin{equation}
c_{2n,i}^{\alpha}=c_{2n,2n-i}^{\alpha},~i=0,1,\ldots,n.
\label{symmetry_of_Sanchez_Reyes_constants}%
\end{equation}

Our main objective is to construct the constrained trivariate counterpart of
basis functions (\ref{univariate_basis}) over the triangular domain%
\[
\Omega^{\alpha}=\left\{  \left(  u,v,w\right)  \in\left[  0,\alpha\right]\times\left[  0,\alpha\right]\times\left[  0,\alpha\right]
:u+v+w=\alpha\right\}  ,
\]
i.e., our intention is to propose a non-negative normalized basis for the
constrained trivariate extension%
\begin{equation}
\mathcal{V}_{n}^{\alpha}=%
\operatorname{span}%
V_{n}^{\alpha}%
\label{constrained_trivariate_extension_of_truncated_Fourier_series}%
\end{equation}
of the vector space (\ref{truncated_Fourier_vector_space}), where the function
system $V_{n}^{\alpha}$ consists of the largest linearly independent subset of
the function system%
\[
\left\{  \cos\left(  ru+gv+bw\right)  ,\sin\left(  ru+gv+bw\right)  :\left(
u,v,w\right)  \in\Omega^{\alpha}\right\}  _{r=0,g=0,b=0}^{n,n,n}.
\]

\begin{remark}
Note, that we only provide a special constrained trivariate extension of the univariate B-basis (\ref{univariate_basis}) that can be used to describe triangular (rational) trigonometric patches with boundary curves defined by linear combinations of control points and functions (\ref{univariate_basis}). We do not suggest that the proposed extension is also a B-basis. To the best of our knowledge, the notion of multivariate normalized B-basis (or something similar) does not even exist in the literature. The present paper describes some aspects of the proposed extension, but its global nature (e.g., whether its total behavior is similar to that of the multivariate Bernstein polynomials, or whether the multivariate Bernstein polynomials form the normalized B-basis of the vector space of multivariate polynomials of finite degree) needs further studies. 
\end{remark}

In CAGD, the expression constrained trivariate function system
refers in fact to a bivariate one. Due to the constraint $u+v+w=\alpha$ each
variable can be written as the linear combination of the remaining two
independent ones. However, we do not fix which is pair assumed to be
independent, since in most cases we will work with different parametrizations
of the triangular domain $\Omega^{\alpha}$.

As we already mentioned, the literature details only the special cases $n=1$
and $n=2$ in recent articles \cite{Wang2010a} and \cite{Wang2010b},
respectively. In order to develop the general framework of the constrained trivariate
extension of the univariate basis functions (\ref{univariate_basis}), we split
our paper into eight sections that are outlined below.

Section \ref{sec:RGB} defines construction rules of a multiplicatively
weighted oriented\ graph of $n\geq1$ levels (numbered from $0$ to $n-1$) of
nodes that store three groups of non-negative constrained trivariate
trigonometric function systems (denoted by $R_{2n}^{\alpha}$, $G_{2n}^{\alpha
}$ and $B_{2n}^{\alpha}$) of order $n$ over $\Omega^{\alpha}$ that fulfill
six cyclic symmetry properties in their variables. The union $T_{2n}^{\alpha}$
of these function systems will form the basis of the constrained trivariate
extension of univariate basis functions (\ref{univariate_basis}).

The linear independence of $T_{2n}^{\alpha}$ will be proved in Section
\ref{sec:linear_independence} by exploiting the symmetry properties of the
oriented graph. More precisely, using three periodically rotating
parametrizations of $\Omega^{\alpha}$, we apply a technique based on a special
form of mathematical induction on the order of partial derivatives\ of a
vanishing linear combination of constrained trivariate functions $R_{2n}^{\alpha}$,
$G_{2n}^{\alpha}$ and $B_{2n}^{\alpha}$.

Section \ref{sec:coincidence} introduces an equivalence relation by means of
which one can recursively construct the linearly independent function system
$V_{n}^{\alpha}$ that spans the constrained trivariate extension
$\mathcal{V}_{n}^{\alpha}$ of the univariate vector space $\mathcal{F}%
_{2n}^{\alpha}$. As expected, vector spaces $\mathcal{T}_{2n}^{\alpha}=%
\operatorname{span}%
T_{2n}^{\alpha}$ and $\mathcal{V}_{n}^{\alpha}=%
\operatorname{span}%
V_{n}^{\alpha}$ coincide. Indeed, using equivalence classes we also prove
the latter statement along with the determination of the common dimension
$\delta_{n}=3n\left(  n+1\right)  +1$ of these vector spaces.

Section \ref{sec:partition_of_unity} offers a procedure to obtain the
normalized form $\overline{T}_{2n}^{\alpha}$ of the function system
$T_{2n}^{\alpha}$. Due to the complexity of this problem, closed formulas of
corresponding non-negative, symmetric and unique normalizing coefficients are
given only in case of levels $0$ and $1$ for arbitrary order, and for all
levels just for orders $n=1$, $2$ and $3$.

Section \ref{sec:applications} lists possible applications of our constrained triangular extension by providing control point based surface modeling tools that may be used in CAGD. Subsection
\ref{sec:triangular_trigonometric_patches} introduces the notion of triangular
trigonometric patches of order $n$ and presents some of their geometric
properties. Using non-negative weights of rank $1$ and quotient basis functions, the
rational counterpart of $\overline{T}_{2n}^{\alpha}$ is formulated in Subsection
\ref{sec:triangular_rational_trigonometric_patches} that defines triangular
rational trigonometric patches.

Compared to the classical constrained trivariate Bernstein polynomials on
triangular domains, in this non-polyno\-mi\-al case, theoretical questions are
significantly harder to answer even for special values of the order $n$.
Since we cannot answer some theoretical questions in their full generality for
the present, Section \ref{sec:open_problems} formulates several open problems
like the general basis transformation between vector spaces $\mathcal{T}%
_{2n}^{\alpha}$ and $\mathcal{V}_{n}^{\alpha}$, the non-negativity, symmetry
properties and closed/recursive formulas of normalizing coefficients of arbitrary
order, general order elevation, and convergence of (rational) triangular trigonometric patches to (rational) B\'ezier triangles when $\alpha \to 0$. 

\begin{remark}[Technical report]
In order to reduce the length of the paper, technical details of some proofs and reformulations are left out, that can be found in the technical report \cite{tech_report}.
\end{remark}

\section{\label{sec:RGB}Constrained trivariate function systems $R_{2n}%
^{\alpha}$, $G_{2n}^{\alpha}$ and $B_{2n}^{\alpha}$: a graph-based approach}

Consider the oriented graph of order $n\geq1$ illustrated in Fig.\ \ref{fig:oriented_graph}(\emph{a}).%
\begin{figure}
[!htb]
\begin{center}
\includegraphics[
height=11.3126cm,
width=15.7059cm
]%
{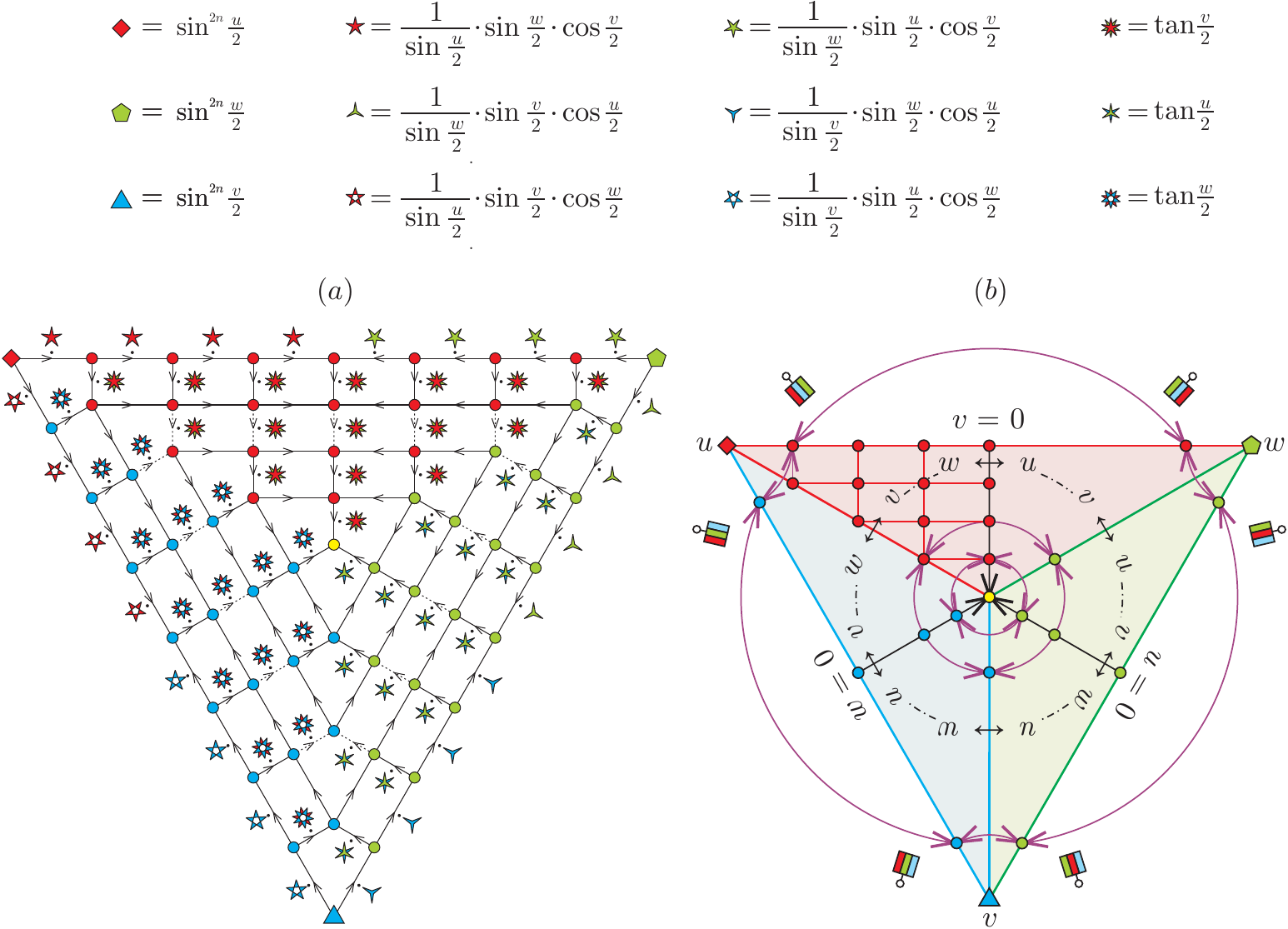}%
\caption{(\emph{a}) Construction rules of constrained trivariate function
systems (\ref{R_system}), (\ref{G_system}) and (\ref{B_system}). (\emph{b})
The red, green and blue domains correspond to systems (\ref{R_system}),
(\ref{G_system}) and (\ref{B_system}), respectively. Functions associated
with the red, green and blue shaded domains are symmetric in variables
$\left(  w,u\right)  $, $\left(  v,w\right)  $ and $\left(  u,v\right)  $,
respectively. Functions which correspond to the adjacent domains red--green,
green--blue and blue--red are symmetric in variables $\left(  v,u\right)  $,
$\left(  u,w\right)  $ and $\left(  w,v\right)  $, respectively. (For
interpretation of the references to color in this figure legend, the reader is
referred to the web version of this paper.)}%
\label{fig:oriented_graph}%
\end{center}
\end{figure}
The three outermost nodes (i.e., the vertices of the outermost triangle) store
the functions $\sin^{2n}\frac{u}{2}$, $\sin^{2n}\frac{v}{2}$ and $\sin
^{2n}\frac{w}{2}$. Each directed edge has a weight function that defines a
multiplication factor when one follows a path from a given node to another
one. Tracking a given path and multiplying by weight functions along the
edges, determines the function stored in an inner node of the graph. Observe
that the layout of these multiplicative weight functions ensures that all
paths starting from an outermost node and terminating at the innermost one (i.e.,
the common centroid of triangles) generate the same constrained trivariate function
$\sin^{n}\frac{u}{2}\sin^{n}\frac{v}{2}\sin^{n}\frac{w}{2}$. Based on these
multiplication rules, one can introduce the constrained trivariate function
systems%
\begin{align}
R_{2n}^{\alpha} &  =\left\{  R_{2n,2n-i,j}^{\alpha}\left(  u,v,w\right)
:\left(  u,v,w\right)  \in\Omega^{\alpha}\right\}  _{j=0,i=j}^{n,2n-j}%
,\label{R_system}\\
G_{2n}^{\alpha} &  =\left\{  G_{2n,2n-i,j}^{\alpha}\left(  u,v,w\right)
:\left(  u,v,w\right)  \in\Omega^{\alpha}\right\}  _{j=0,i=j}^{n,2n-j}%
\label{G_system}%
\end{align}
and%
\begin{equation}
B_{2n}^{\alpha}=\left\{  B_{2n,2n-i,j}^{\alpha}\left(  u,v,w\right)  :\left(
u,v,w\right)  \in\Omega^{\alpha}\right\}  _{j=0,i=j}^{n,2n-j},\label{B_system}%
\end{equation}
where%
\begin{align}
\left\{  R_{2n,2n-i,j}^{\alpha}\left(  u,v,w\right)  \right\}  _{j=0,i=j}%
^{n,n} &  =\left\{  \sin^{2n-i}\frac{u}{2}\sin^{i}\frac{w}{2}\cos^{i-j}%
\frac{v}{2}\sin^{j}\frac{v}{2}\right\}  _{j=0,i=j}^{n,n}%
,\label{R_system_leq_n}\\
& \nonumber\\
\left\{  R_{2n,2n-i,j}^{\alpha}\left(  u,v,w\right)  \right\}  _{j=0,i=n+1}%
^{n-1,2n-j} &  =\left\{  R_{2n,i,j}^{\alpha}\left(  w,v,u\right)  \right\}
_{j=0,i=n+1}^{n-1,2n-j}\label{R_system_geq_n}\\
&  =\left\{  \sin^{i}\frac{w}{2}\sin^{2n-i}\frac{u}{2}\cos^{2n-i-j}\frac{v}%
{2}\sin^{j}\frac{v}{2}\right\}  _{j=0,i=n+1}^{n-1,2n-j},\nonumber
\end{align}
and due to the symmetry%
\begin{align}
\left\{  G_{2n,2n-i,j}^{\alpha}\left(  u,v,w\right)  \right\}  _{j=0,i=j}%
^{n,n} &  =\left\{  R_{2n,2n-i,j}^{\alpha}\left(  w,u,v\right)  \right\}
_{j=0,i=j}^{n,n}\\
&  =\left\{  \sin^{2n-i}\frac{w}{2}\sin^{i}\frac{v}{2}\cos^{i-j}\frac{u}%
{2}\sin^{j}\frac{u}{2}\right\}  _{j=0,i=j}^{n,n},\nonumber\\
& \nonumber\\
\left\{  G_{2n,2n-i,j}^{\alpha}\left(  u,v,w\right)  \right\}  _{j=0,i=n+1}%
^{n-1,2n-j} &  =\left\{  G_{2n,i,j}^{\alpha}\left(  u,w,v\right)  \right\}
_{j=0,i=n+1}^{n-1,2n-j}\\
&  =\left\{  \sin^{i}\frac{v}{2}\sin^{2n-i}\frac{w}{2}\cos^{2n-i-j}\frac{u}%
{2}\sin^{j}\frac{u}{2}\right\}  _{j=0,i=n+1}^{n-1,2n-j},\nonumber
\end{align}%
\begin{align}
\left\{  B_{2n,2n-i,j}^{\alpha}\left(  u,v,w\right)  \right\}  _{j=0,i=j}%
^{n,n} &  =\left\{  R_{2n,2n-i,j}^{\alpha}\left(  v,w,u\right)  \right\}
_{j=0,i=j}^{n,n}\\
&  =\left\{  \sin^{2n-i}\frac{v}{2}\sin^{i}\frac{u}{2}\cos^{i-j}\frac{w}%
{2}\sin^{j}\frac{w}{2}\right\}  _{j=0,i=j}^{n,n},\nonumber\\
& \nonumber\\
\left\{  B_{2n,2n-i,j}^{\alpha}\left(  u,v,w\right)  \right\}  _{j=0,i=n+1}%
^{n-1,2n-j} &  =\left\{  B_{2n,i,j}^{\alpha}\left(  v,u,w\right)  \right\}
_{j=0,i=n+1}^{n-1,2n-j}\\
&  =\left\{  \sin^{i}\frac{u}{2}\sin^{2n-i}\frac{v}{2}\cos^{2n-i-j}\frac{w}%
{2}\sin^{j}\frac{w}{2}\right\}  _{j=0,i=n+1}^{n-1,2n-j}.\nonumber
\end{align}

Observe that function systems (\ref{R_system}), (\ref{G_system}) and
(\ref{B_system}) consist of non-negative functions which fulfill the six
symmetry properties illustrated in Fig.\ \ref{fig:oriented_graph}(\emph{b}).

In what follows, we study the properties of the constrained trivariate
non-negative joint function system%
\begin{equation}%
\begin{array}
[c]{ccl}%
T_{2n}^{\alpha} & = & \left\{  R_{2n,2n-i,j}^{\alpha}\left(  u,v,w\right)
,G_{2n,2n-i,j}^{\alpha}\left(  u,v,w\right)  ,B_{2n,2n-i,j}^{\alpha}\left(
u,v,w\right)  :\left(  u,v,w\right)  \in\Omega^{\alpha}\right\}
_{j=0,i=j}^{n-1,2n-1-j}\\
&  & \\
&  & \cup\left\{  R_{2n,n,n}^{\alpha}\left(  u,v,w\right)  =G_{2n,n,n}%
^{\alpha}\left(  u,v,w\right)  =B_{2n,n,n}^{\alpha}\left(  u,v,w\right)
:\left(  u,v,w\right)  \in\Omega^{\alpha}\right\}
\end{array}
\label{united_system}%
\end{equation}
of order $n\geq1$. We refer to the index $j$ of these functions as levels
since they correspond to the nested triangles shrinking from the boundary to
their common centroid depicted in Fig.\ \ref{fig:oriented_graph}(\emph{a}).
Fig.\ \ref{fig:united_system_n_3} illustrates the layout of the constrained
trivariate function system $T_{6}^{\alpha}$ of order $3$.%

\begin{figure}
[!htb]
\begin{center}
\includegraphics[
height=2.3281in,
width=2.8859in
]%
{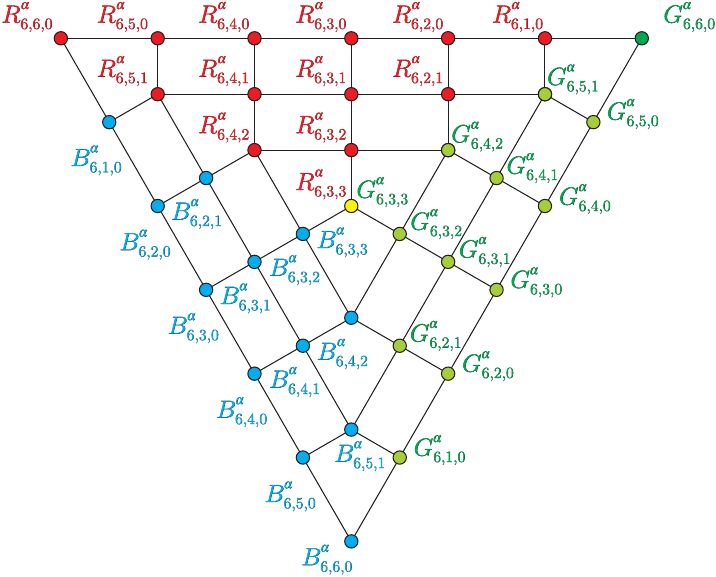}%
\caption{The layout of the constrained trivariate joint function system
(\ref{united_system}) of order $n=3$.}%
\label{fig:united_system_n_3}%
\end{center}
\end{figure}

Note that the function system (\ref{united_system}) is reduced to the univariate
B-basis (\ref{univariate_basis}) whenever one of the three variables $u$, $v$
and $w$ vanishes, i.e., the system fulfills the boundary properties%
\begin{align}
\left.  T_{2n}^{\alpha}\right\vert _{v=0} &  =\left\{  \frac{1}{c_{2n,2n-i}%
^{\alpha}}A_{2n,2n-i}^{\alpha}\left(  u\right)  :u\in\left[  0,\alpha\right]
\right\}  _{i=0}^{2n},\label{boundary_property_1}\\
\left.  T_{2n}^{\alpha}\right\vert _{u=0} &  =\left\{  \frac{1}{c_{2n,2n-i}%
^{\alpha}}A_{2n,2n-i}^{\alpha}\left(  w\right)  :w\in\left[  0,\alpha\right]
\right\}  _{i=0}^{2n},\label{boundary_property_2}\\
\left.  T_{2n}^{\alpha}\right\vert _{w=0} &  =\left\{  \frac{1}{c_{2n,2n-i}%
^{\alpha}}A_{2n,2n-i}^{\alpha}\left(  v\right)  :v\in\left[  0,\alpha\right]
\right\}  _{i=0}^{2n}.\label{boundary_property_3}%
\end{align}

\begin{proposition}
[\textbf{Bernstein polynomials of degree $2n$ as special case}]\label{Bernstein}If one uses the linear reparametrization $t\left(  s\right)  =\alpha s,~s\in\left[
0,1\right]  $, then basis functions (\ref{univariate_basis}) of order $n$ converge to the
Bernstein polynomials of degree $2n$ when $\alpha\rightarrow0$, i.e.,%
\begin{equation}
\lim_{\alpha\rightarrow0}A_{2n,i}^{\alpha}\left( \alpha s\right)  =B_{i}^{2n}\left(
s\right)  ,\, \forall s \in \left[0,1\right],~i=0,1,\ldots,2n\text{.} \label{limit_Bernstein}%
\end{equation}

\end{proposition}

\begin{proof}
In Reference \cite{SReyes1998} functions (\ref{univariate_basis}) were
obtained by raising the identity%
\[
A_{2n,0}^{\alpha}\left(  \alpha s\right)  +A_{2n,1}^{\alpha}\left( \alpha s\right)
+A_{2n,2}^{\alpha}\left(  \alpha s\right)  =1,~\forall s\in\left[  0,1\right]
\]
to the power $n$, i.e.,
\[
\sum_{i=0}^{2n}A_{2n,i}^{\alpha}\left(  \alpha s\right)  =\left(  A_{2n,0}^{\alpha
}\left(  s\right)  +A_{2n,1}^{\alpha}\left(  s\right)  +A_{2n,2}^{\alpha
}\left(  s\right)  \right)  ^{n}=1,~\forall s\in\left[  0,1\right]
\]
which in the limiting case $\alpha\rightarrow0$ takes the form%
\begin{align*}
&  \lim_{\alpha\rightarrow0}\left(  \frac{1}{\sin^{2}\left(  \frac{\alpha}%
{2}\right)  }\sin^{2}\left(  \frac{\alpha-\alpha s}{2}\right)  +\frac
{2\cos\left(  \frac{\alpha}{2}\right)  }{\sin^{2}\left(  \frac{\alpha}%
{2}\right)  }\sin\left(  \frac{\alpha-\alpha s}{2}\right)  \sin\left(
\frac{\alpha s}{2}\right)  +\frac{1}{\sin^{2}\left(  \frac{\alpha}{2}\right)
}\sin^{2}\left(  \frac{\alpha s}{2}\right)  \right)  ^{n}\\
=  &  \left(  \left(  1-s\right)  ^{2}+2\left(  1-s\right)  s+s^{2}\right)
^{n}\\
=  &  \left(  \left(  1-s\right)  +s\right)  ^{2n}\\
=  &  \sum_{i=0}^{2n}B_{i}^{2n}\left(  s\right)
\end{align*}
for all values of $s\in\left[  0,1\right]  $,
since $
\lim\limits_{\alpha\rightarrow0}\frac{\sin\left(  \frac{\alpha s}{2}\right)  }%
{\sin\left(  \frac{\alpha}{2}\right)  }  =s$ and $\lim\limits_{\alpha\rightarrow0}\cos\left(  \frac{\alpha s}{2}\right)   =1.$
\end{proof}

\section{\label{sec:linear_independence}Linear independence}

The linear independence of the joint function system (\ref{united_system}) will be proved by using higher order mixed partial derivatives.
In order to evaluate these derivates, when one of the constrained variables equals $0$, we
have to study the behavior of functions%
\begin{align*}
C_{i,\lambda}^{p}\left(  t\right)   &  =\frac{\text{d}^{p}}{\text{d}t^{p}}%
\cos^{i}\left(  \lambda t\right)  ,~t\in%
\mathbb{R}
,\\
S_{j,\mu}^{p}\left(  t\right)   &  =\frac{\text{d}^{p}}{\text{d}t^{p}}\sin
^{j}\left(  \mu t\right)  ,~t\in%
\mathbb{R}%
\end{align*}
and%
\[
M_{i,\lambda,j,\mu}^{z}\left(  t\right)  =\frac{\text{d}^{z}}{\text{d}t^{z}%
}\left(  \cos^{i}\left(  \lambda t\right)  \sin^{j}\left(  \mu t\right)
\right)
\]
at $t=0$, where exponents $i$, $j$ and orders $p$, $z$ are natural numbers, while
angular velocities $\lambda,\mu>0$ are real parameters. The proofs of the following Lemma \ref{lem:sign_cos_power_derivative} and Proposition \ref{derivative_cosine_sine} can be found in \cite{tech_report}.

\begin{lemma}
\label{lem:sign_cos_power_derivative}
Let $i$, $j$ and $p$ be natural numbers greater than or equal to $1$. If
$\lambda,\mu>0$, then signs of values $\left\{  C_{i,\lambda}^{p}\left(
0\right)  \right\}  _{p\geq0}$ and $\left\{  S_{j,\mu}^{p}\left(  0\right)
\right\}  _{p\geq0}$ are%
\begin{equation}%
\operatorname{sign}%
C_{i,\lambda}^{p}\left(  0\right)  =\left\{
\begin{array}
[c]{rl}%
-1, & ~p-i=2r,~r\equiv1\left(  \operatorname{mod}2\right)  ,\\
0, & ~p\equiv1\left(  \operatorname{mod}2\right)  ,\\
+1, & ~p-i=2r,~r\equiv0\left(  \operatorname{mod}2\right)
\end{array}
\right.  \label{sign_cos_power_derivative}%
\end{equation}
and%
\begin{equation}%
\operatorname{sign}%
S_{j,\mu}^{p}\left(  0\right)  =\left\{
\begin{array}
[c]{rl}%
-1, & ~p\geq j,~p-j=2r,~r\equiv1\left(  \operatorname{mod}2\right)  ,\\
0, & ~\left(  p<j\right)  \ \text{or }\left(  p-j\equiv1\left(
\operatorname{mod}2\right)  \text{ and }p>j\right)  ,\\
+1, & ~p\geq j,~p-j=2r,~r\equiv0\left(  \operatorname{mod}2\right)  ,
\end{array}
\right.  \label{sign_sin_power_derivative}%
\end{equation}
respectively. (In particular, $S_{j,\mu}^{j}\left(  0\right)  =j!\mu^{j}$).
\end{lemma}


\begin{proposition}
\label{derivative_cosine_sine}Independently of $i\in%
\mathbb{N}
$ and $\lambda,\mu>0$, for all values of $j,z\in%
\mathbb{N}
$ such that $i+j\neq0$ and $z\geq1$ we have the equality%
\begin{equation}%
\operatorname{sign}%
M_{i,\lambda,j,\mu}^{z}\left(  0\right)  =\left\{
\begin{array}
[c]{rl}%
-1, & ~z-j\equiv2\left(  \operatorname{mod}4\right)  \text{ and }z\geq j,\\
0, & ~\left(  z<j\right)  \text{ or }\left(  z-j\equiv\pm1\left(
\operatorname{mod}4\right)  \text{ and }z>j\right)  ,\\
+1, & ~z-j\equiv0\left(  \operatorname{mod}4\right)  \text{ and }z\geq j.
\end{array}
\right.  \label{derivative_cosine_sine_at_0}%
\end{equation}

\end{proposition}

Under the parametrization%
\begin{equation}
\left\{
\begin{array}
[c]{rcl}%
u\left(  x,y\right)  & = & \alpha-x,~x\in\left[  0,\alpha\right]  ,\\
v\left(  x,y\right)  & = & y,~y\in\left[  0,\alpha-x\right]  ,\\
w\left(  x,y\right)  & = & x-y,
\end{array}
\right.  \label{parametrization_for_partition_of_unity}%
\end{equation}
of $\Omega^{\alpha}$, one can easily show that the $r$th order ($r\geq0$)
partial derivative of any smooth function $L:\Omega^{\alpha}\rightarrow%
\mathbb{R}
$ with respect to the variable $y$ is%
\begin{equation}
\frac{\partial^{r}}{\partial y^{r}}L\left(  u\left(  x,y\right)  ,v\left(
x,y\right)  ,w\left(  x,y\right)  \right)  =\sum_{k=0}^{r}\dbinom{r}{k}\left(
-1\right)  ^{k}\frac{\partial^{r}}{\partial v^{r-k}\partial w^{k}}L\left(
u\left(  x,y\right)  ,v\left(  x,y\right)  ,w\left(  x,y\right)  \right)  .
\label{chain_rule}%
\end{equation}
If the function $L$ is defined as the linear combination%
\begin{equation}%
\begin{array}
[c]{ccl}%
L\left(  u,v,w\right)  & = & \rho_{2n,n,n}R_{2n,n,n}^{\alpha}\left(
u,v,w\right)  +%
{\displaystyle\sum\limits_{j=0}^{n-1}}
{\displaystyle\sum\limits_{i=j}^{2n-1-j}}
\rho_{2n,2n-i,j}R_{2n,2n-i,j}^{\alpha}\left(  u,v,w\right) \\
&  & +%
{\displaystyle\sum\limits_{j=0}^{n-1}}
{\displaystyle\sum\limits_{i=j}^{2n-1-j}}
\gamma_{2n,2n-i,j}G_{2n,2n-i,j}^{\alpha}\left(  u,v,w\right)  +%
{\displaystyle\sum\limits_{j=0}^{n-1}}
{\displaystyle\sum\limits_{i=j}^{2n-1-j}}
\beta_{2n,2n-i,j}B_{2n,2n-i,j}^{\alpha}\left(  u,v,w\right)
\end{array}
\label{linear_combination_union}%
\end{equation}
where $\rho_{2n,n,n}$, $\left\{  \rho_{2n,2n-i,j}\right\}  _{j=0,i=j}%
^{n-1,2n-1-j}$, $\left\{  \gamma_{2n,2n-i,j}\right\}  _{j=0,i=j}^{n-1,2n-1-j}$
and $\left\{  \beta_{2n,2n-i,j}\right\}  _{j=0,i=j}^{n-1,2n-1-j}$ are real
scalars, then, by the help of Proposition \ref{derivative_cosine_sine}, one
can easily determine those functions of system (\ref{united_system}) the
higher order mixed partial derivatives ($\frac{\partial^{r}}{\partial
v^{r-k}\partial w^{k}}\left(  \cdot\right)  $, $k=0,1,\ldots,r$) of which
vanish when one evaluates the terms of (\ref{chain_rule}) at $y=0$. In the
case of $n=4$, we have provided an example in Fig.\ \ref{fig:derivative_scheme_n_4}.%

\begin{figure}
[!htb]
\begin{center}
\includegraphics[
height=5.4388in,
width=6.333in
]%
{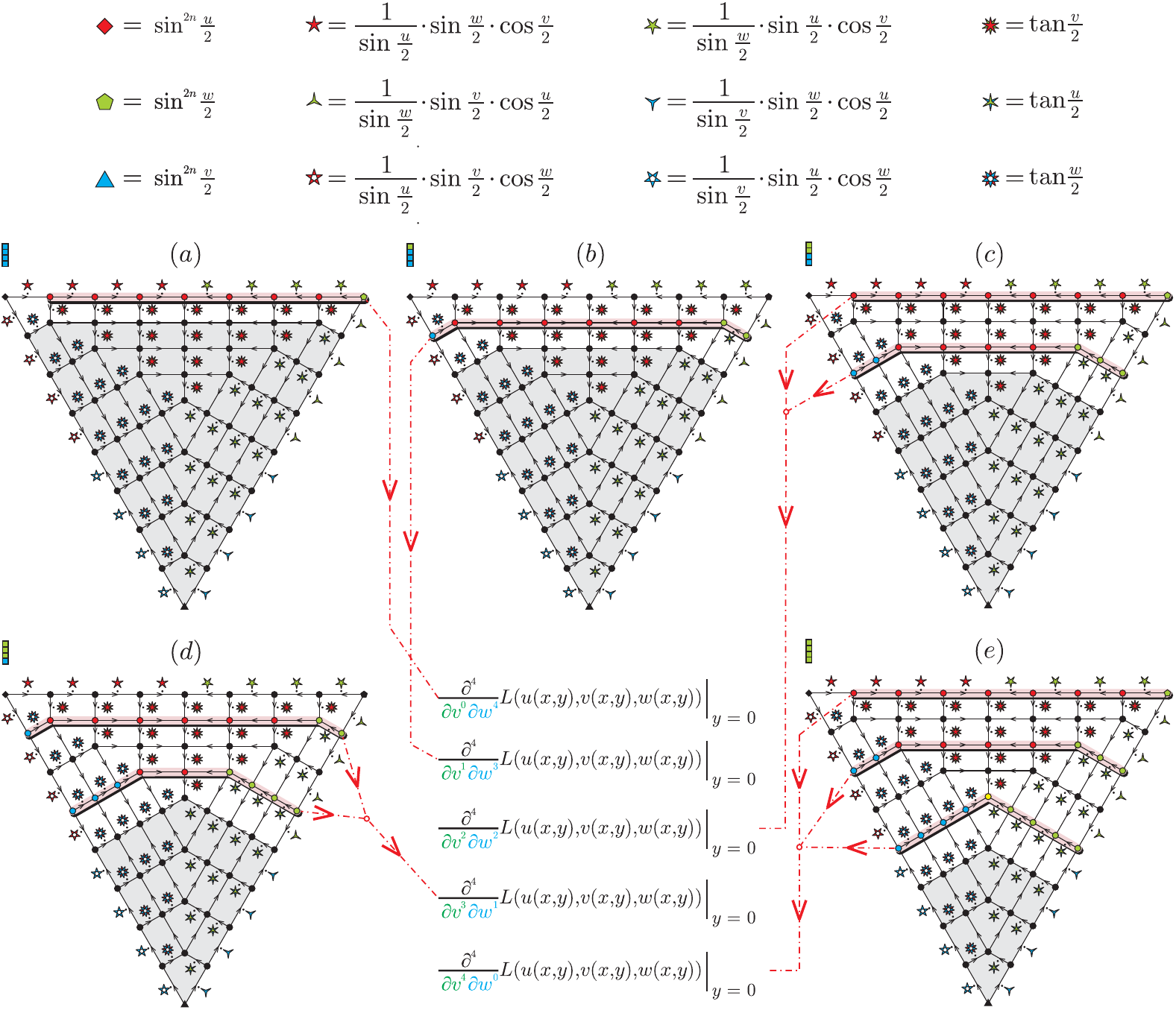}%
\caption{(\emph{a})--(\emph{e}) Using the sign function
(\ref{derivative_cosine_sine_at_0}), one can easily verify that black and
non-black dots represent those functions of system (\ref{united_system}) the
$4$th order mixed partial derivatives of which with respect to $v\left(
x,y\right)  =y$ and $w\left(  x,y\right)  =x-y$ are zero and non-zero,
respectively, under parametrization
(\ref{parametrization_for_partition_of_unity}) at $y=0$. Observe, that the
gray shaded areas correspond to functions which comprise the factor $\sin
\frac{v\left(  x,y\right)  }{2}=\sin\frac{y}{2}$ raised to a power greater
than the corresponding order $k=0,1,\ldots,4$ of the partial derivative with
respect to $v\left(  x,y\right)  =y$. Thus, these partial derivatives vanish
at $y=0$ due to the first condition of the zero branch of the sign function
(\ref{derivative_cosine_sine_at_0}). Black dots that fall outside of the gray
shaded areas correspond to constrained trivariate functions that comprise the factor
$\sin\frac{v\left(  x,y\right)  }{2}=\sin\frac{y}{2}$ raised to a power which
-- together with the order $k=0,1,\ldots,4$ of the partial derivative with
respect to $v\left(  x,y\right)  =y$, at $y=0$ -- fulfills the second
condition of the zero branch of the sign function
(\ref{derivative_cosine_sine_at_0}). (For interpretation of the references to
color in this figure legend, the reader is referred to the web version of this
paper.)}%
\label{fig:derivative_scheme_n_4}%
\end{center}
\end{figure}

\begin{theorem}
[\textbf{Linear independence of the joint systems}]%
\label{joint_systems_linear_independence}The function system
(\ref{united_system}) is linearly independent and the dimension of the vector
space $\mathcal{T}_{2n}^{\alpha}=~$\emph{span}$\emph{~}T_{2n}^{\alpha}$ is
\[
\delta_{n}=3n\left(  n+1\right)  +1.
\]

\end{theorem}

\begin{proof}
It is easy to verify that the number of functions in system
(\ref{united_system}) is exactly%
\[
\delta_{n}=1+3\sum_{j=0}^{n-1}\sum_{i=j}^{2n-1-j}1=1+3n\left(  n+1\right)  .
\]
Consider the linear combination (\ref{linear_combination_union}) and assume
that the equality%
\begin{equation}
L\left(  u,v,w\right)  =0 \label{linear_combination_union_null_equality}%
\end{equation}
holds for all $\left(  u,v,w\right)  \in\Omega^{\alpha}$.

If equality (\ref{linear_combination_union_null_equality}) holds for all
$\left(  u,v,w\right)  \in\Omega^{\alpha}$ then it is also valid under all
possible parametrizations of the definition domain $\Omega^{\alpha}$. In what
follows, we will work with three different parametrizations of $\Omega
^{\alpha}$, namely with%
\begin{equation}
\left\{
\begin{array}
[c]{rcl}%
u_{1}\left(  x,y\right)  & = & x,~x\in\left[  0,\alpha\right]  ,\\
v_{1}\left(  x,y\right)  & = & y,~y\in\left[  0,\alpha-x\right]  ,\\
w_{1}\left(  x,y\right)  & = & \alpha-x-y,
\end{array}
\right.  \label{parametrization_1_y_0}%
\end{equation}%
\begin{equation}
\left\{
\begin{array}
[c]{rcl}%
u_{2}\left(  x,y\right)  & = & x,~x\in\left[  0,\alpha\right]  ,\\
v_{2}\left(  x,y\right)  & = & \alpha-x-y,\\
w_{2}\left(  x,y\right)  & = & y,~y\in\left[  0,\alpha-x\right]  ,
\end{array}
\right.  \label{parametrization_2_x_0}%
\end{equation}
and%
\begin{equation}
\left\{
\begin{array}
[c]{rcl}%
u_{3}\left(  x,y\right)  & = & \alpha-x-y,\\
v_{3}\left(  x,y\right)  & = & x,~x\in\left[  0,\alpha\right]  ,\\
w_{3}\left(  x,y\right)  & = & y,~y\in\left[  0,\alpha-x\right]  .
\end{array}
\right.  \label{parametrization_3_y_0}%
\end{equation}
Straight calculations show that%
\begin{align}
\frac{\partial^{r}}{\partial y^{r}}L\left(  u_{1}\left(  x,y\right)
,v_{1}\left(  x,y\right)  ,w_{1}\left(  x,y\right)  \right)   &  =\sum
_{k=0}^{r}\dbinom{r}{k}\left(  -1\right)  ^{k}\frac{\partial^{r}}{\partial
v_{1}^{r-k}\partial w_{1}^{k}}L\left(  u_{1}\left(  x,y\right)  ,v_{1}\left(
x,y\right)  ,w_{1}\left(  x,y\right)  \right)
,\label{derivative_formula_parametrization_1}\\
& \nonumber\\
\frac{\partial^{r}}{\partial x^{r}}L\left(  u_{2}\left(  x,y\right)
,v_{2}\left(  x,y\right)  ,w_{2}\left(  x,y\right)  \right)   &  =\sum
_{k=0}^{r}\dbinom{r}{k}\left(  -1\right)  ^{k}\frac{\partial^{r}}{\partial
u_{2}^{r-k}\partial v_{2}^{k}}L\left(  u_{2}\left(  x,y\right)  ,v_{2}\left(
x,y\right)  ,w_{2}\left(  x,y\right)  \right)
,\label{derivative_formula_parametrization_2}\\
& \nonumber\\
\frac{\partial^{r}}{\partial y^{r}}L\left(  u_{3}\left(  x,y\right)
,v_{3}\left(  x,y\right)  ,w_{3}\left(  x,y\right)  \right)   &  =\sum
_{k=0}^{r}\dbinom{r}{k}\left(  -1\right)  ^{k}\frac{\partial^{r}}{\partial
w_{3}^{r-k}\partial u_{3}^{k}}L\left(  u_{3}\left(  x,y\right)  ,v_{3}\left(
x,y\right)  ,w_{3}\left(  x,y\right)  \right)
\label{derivative_formula_parametrization_3}%
\end{align}
for all $x\in\left[  0,\alpha\right]  $,$\ y\in\left[  0,\alpha-x\right]  $
and order $r\in%
\mathbb{N}
$.

In order to prove the statement, we argue by mathematical induction on the
derivation order $r=0,1,\ldots,n$.

In the case of $r=0$ and parametrization (\ref{parametrization_1_y_0}), observe
that at $y=0$ equality (\ref{linear_combination_union_null_equality}) becomes%
\begin{align*}
0 &  =\left.  L\left(  u_{1}\left(  x,y\right)  ,v_{1}\left(  x,y\right)
,w_{1}\left(  x,y\right)  \right)  \right\vert _{y=0}\\
&  =\sum_{i=0}^{2n-1}\rho_{2n,2n-i,0}R_{2n,2n-i,0}^{\alpha}\left(
x,0,\alpha-x\right)  +\gamma_{2n,2n,0}G_{2n,2n,0}^{\alpha}\left(
x,0,\alpha-x\right)  \\
&  =\sum_{i=0}^{2n-1}\rho_{2n,2n-i,0}\sin^{2n-i}\frac{x}{2}\sin^{i}%
\frac{\alpha-x}{2}+\gamma_{2n,2n-i,0}\sin^{2n}\frac{\alpha-x}{2}\\
&  =\sum_{i=0}^{2n}\rho_{2n,2n-i,0}^{\prime}c_{2n,2n-i}^{\alpha}\sin
^{2n-i}\frac{x}{2}\sin^{i}\frac{\alpha-x}{2}\\
&  =\sum_{i=0}^{2n}\rho_{2n,2n-i,0}^{\prime}A_{2n,2n-i}^{\alpha}\left(
x\right)  ,~\forall x\in\left[  0,\alpha\right]  ,
\end{align*}
where we have used the notation%
\[
\rho_{2n,2n-i,0}^{\prime}=\left\{
\begin{array}
[c]{rl}%
\dfrac{\rho_{2n,2n-i,0}}{c_{2n,2n-i}^{\alpha}}, & i=0,1,\ldots,2n-1,\\
& \\
\dfrac{\gamma_{2n,2n,0}}{c_{2n,0}^{\alpha}}, & i=2n
\end{array}
\right.
\]
and the fact that all omitted functions include the factor $\left.  \sin
\frac{v_{1}\left(  x,y\right)  }{2}\right\vert _{y=0}=\left.  \sin\frac{y}%
{2}\right\vert _{y=0}=0$ raised to a power at least one. Since the function
system $\left\{  A_{2n,2n-i}^{\alpha}\left(  x\right)  :x\in\left[
0,\alpha\right]  \right\}  _{i=0}^{2n}$ is linearly independent, it follows
that%
\[
\rho_{2n,2n-i,0}^{\prime}=0,~i=0,1,\ldots,2n.
\]
Thus, equality (\ref{linear_combination_union_null_equality}) can be reduced
to%
\begin{equation}%
\begin{array}
[c]{ccl}%
0 & = & \widetilde{L}\left(  u,v,w\right)  \\
& = & \rho_{2n,n,n}R_{2n,n,n}^{\alpha}\left(  u,v,w\right)  +%
{\displaystyle\sum\limits_{j=1}^{n-1}}
{\displaystyle\sum\limits_{i=j}^{2n-1-j}}
\rho_{2n,2n-i,j}R_{2n,2n-i,j}^{\alpha}\left(  u,v,w\right)  \\
&  & +%
{\displaystyle\sum\limits_{i=1}^{2n-1}}
\gamma_{2n,2n-i,0}G_{2n,2n-i,0}^{\alpha}\left(  u,v,w\right)  +%
{\displaystyle\sum\limits_{j=1}^{n-1}}
{\displaystyle\sum\limits_{i=j}^{2n-1-j}}
\gamma_{2n,2n-i,j}G_{2n,2n-i,j}^{\alpha}\left(  u,v,w\right)  \\
&  & +%
{\displaystyle\sum\limits_{j=0}^{n-1}}
{\displaystyle\sum\limits_{i=j}^{2n-1-j}}
\beta_{2n,2n-i,j}B_{2n,2n-i,j}^{\alpha}\left(  u,v,w\right)  ,~\forall\left(
u,v,w\right)  \in\Omega^{\alpha}.
\end{array}
\label{linear_combination_union_reduced_null_equality_1}%
\end{equation}

Now, considering parametrization (\ref{parametrization_2_x_0}) and the partial
derivative of order $r=0$ of equality
(\ref{linear_combination_union_reduced_null_equality_1}) with respect to the
variable $x$, at $x=0$, we obtain that%
\begin{align*}
0 &  =\left.  \widetilde{L}\left(  u_{2}\left(  x,y\right)  ,v_{2}\left(
x,y\right)  ,w_{2}\left(  x,y\right)  \right)  \right\vert _{x=0}\\
&  =\sum_{i=1}^{2n-1}\gamma_{2n,2n-i,0}G_{2n,2n-i,0}^{\alpha}\left(
0,\alpha-y,y\right)  +\beta_{2n,2n,0}B_{2n,2n,0}^{\alpha}\left(
0,\alpha-y,y\right)  \\
&  =\sum_{i=1}^{2n}\gamma_{2n,2n-i,0}^{\prime}A_{2n,2n-i}^{\alpha}\left(
y\right)  ,~\forall y\in\left[  0,\alpha\right]  ,
\end{align*}
where we have used the notation%
\[
\gamma_{2n,2n-i,0}^{\prime}=\left\{
\begin{array}
[c]{rl}%
\dfrac{\gamma_{2n,2n-i,0}}{c_{2n,2n-i}^{\alpha}}, & i=1,2,\ldots,2n-1,\\
& \\
\dfrac{\beta_{2n,2n,0}}{c_{2n,0}^{\alpha}}, & i=2n
\end{array}
\right.
\]
and the fact that discarded functions include the factor $\left.  \sin
\frac{u_{2}\left(  x,y\right)  }{2}\right\vert _{x=0}=\left.  \sin\frac{x}%
{2}\right\vert _{x=0}=0$ raised to a power greater than or equal to one. Due
to the linear independence of the univariate function system $\left\{
A_{2n,2n-i}^{\alpha}\left(  y\right)  :y\in\left[  0,\alpha\right]  \right\}
_{i=1}^{2n}$, one has that%
\[
\gamma_{2n,2n-i,0}^{\prime}=0,~i=1,2,\ldots,2n.
\]

Therefore, equality (\ref{linear_combination_union_reduced_null_equality_1})
can be simplified to the form%
\begin{equation}%
\begin{array}
[c]{rll}%
0 & = & \widehat{L}\left(  u,v,w\right) \\
& = & \rho_{2n,n,n}R_{2n,n,n}^{\alpha}\left(  u,v,w\right)  +%
{\displaystyle\sum\limits_{j=1}^{n-1}}
{\displaystyle\sum\limits_{i=j}^{2n-1-j}}
\rho_{2n,2n-i,j}R_{2n,2n-i,j}^{\alpha}\left(  u,v,w\right) \\
&  & +%
{\displaystyle\sum\limits_{j=1}^{n-1}}
{\displaystyle\sum\limits_{i=j}^{2n-1-j}}
\gamma_{2n,2n-i,j}G_{2n,2n-i,j}^{\alpha}\left(  u,v,w\right) \\
&  & +%
{\displaystyle\sum\limits_{i=1}^{2n-1}}
\beta_{2n,2n-i,0}B_{2n,2n-i,0}^{\alpha}\left(  u,v,w\right)  +%
{\displaystyle\sum\limits_{j=1}^{n-1}}
{\displaystyle\sum\limits_{i=j}^{2n-1-j}}
\beta_{2n,2n-i,j}B_{2n,2n-i,j}^{\alpha}\left(  u,v,w\right)  ,~\forall\left(
u,v,w\right)  \in\Omega^{\alpha}.
\end{array}
\label{linear_combination_union_reduced_null_equality_2}%
\end{equation}
Performing similar calculations as above, but using parametrization
(\ref{parametrization_3_y_0}) and the zeroth order partial derivative of
equality (\ref{linear_combination_union_reduced_null_equality_2}) with respect
to $y$, at $y=0$, we obtain that%
\begin{align*}
0  &  =\left.  \widehat{L}\left(  u_{3}\left(  x,y\right)  ,v_{3}\left(
x,y\right)  ,w_{3}\left(  x,y\right)  \right)  \right\vert _{y=0}\\
&  =\sum_{i=1}^{2n-1}\beta_{2n,2n-i,0}B_{2n,2n-i,0}^{\alpha}\left(
\alpha-x,x,0\right) \\
&  =\sum_{i=1}^{2n-1}\beta_{2n,2n-i,0}^{\prime}A_{2n,2n-i}^{\alpha}\left(
x\right)  ,~\forall x\in\left[  0,\alpha\right]  ,
\end{align*}
where%
\[
\beta_{2n,2n-i,0}^{\prime}=\frac{\beta_{2n,2n-i,0}}{c_{2n,2n-i}^{\alpha}%
},~i=1,2,\ldots,2n-1,
\]
and, due to the linear independence of the function system $\left\{
A_{2n,2n-i}^{\alpha}\left(  x\right)  :x\in\left[  0\,\alpha\right]  \right\}
_{i=1}^{2n-1}$, we have that%
\[
\beta_{2n,2n-i,0}^{\prime}=0,~i=1,2,\ldots,2n-1.
\]
Hence, equality (\ref{linear_combination_union_reduced_null_equality_2}) can
be reduced to%
\begin{align}
0=  &  L_{0}\left(  u,v,w\right)
\label{linear_combination_union_reduced_null_equality_3}\\
=  &  \rho_{2n,n,n}R_{2n,n,n}^{\alpha}\left(  u,v,w\right)  +\sum_{j=1}%
^{n-1}\sum_{i=j}^{2n-1-j}\rho_{2n,2n-i,j}R_{2n,2n-i,j}^{\alpha}\left(
u,v,w\right) \nonumber\\
&  +\sum_{j=1}^{n-1}\sum_{i=j}^{2n-1-j}\gamma_{2n,2n-i,j}G_{2n,2n-i,j}%
^{\alpha}\left(  u,v,w\right)  +\sum_{j=1}^{n-1}\sum_{i=j}^{2n-1-j}%
\beta_{2n,2n-i,j}B_{2n,2n-i,j}^{\alpha}\left(  u,v,w\right)  ,~\forall\left(
u,v,w\right)  \in\Omega^{\alpha}.\nonumber
\end{align}

Cases (\emph{a})--(\emph{c}) of Fig.\ \ref{fig:linear_independence_j_0}
represent calculations that correspond to the evaluation of the zeroth order
partial derivatives of equalities
(\ref{linear_combination_union_null_equality}),
(\ref{linear_combination_union_reduced_null_equality_1}) and
(\ref{linear_combination_union_reduced_null_equality_2}) detailed above, under
successive parametrizations (\ref{parametrization_1_y_0}%
)--(\ref{parametrization_3_y_0}).%

\begin{figure}
[!htb]
\begin{center}
\includegraphics[
height=3.7516in,
width=6.3607in
]%
{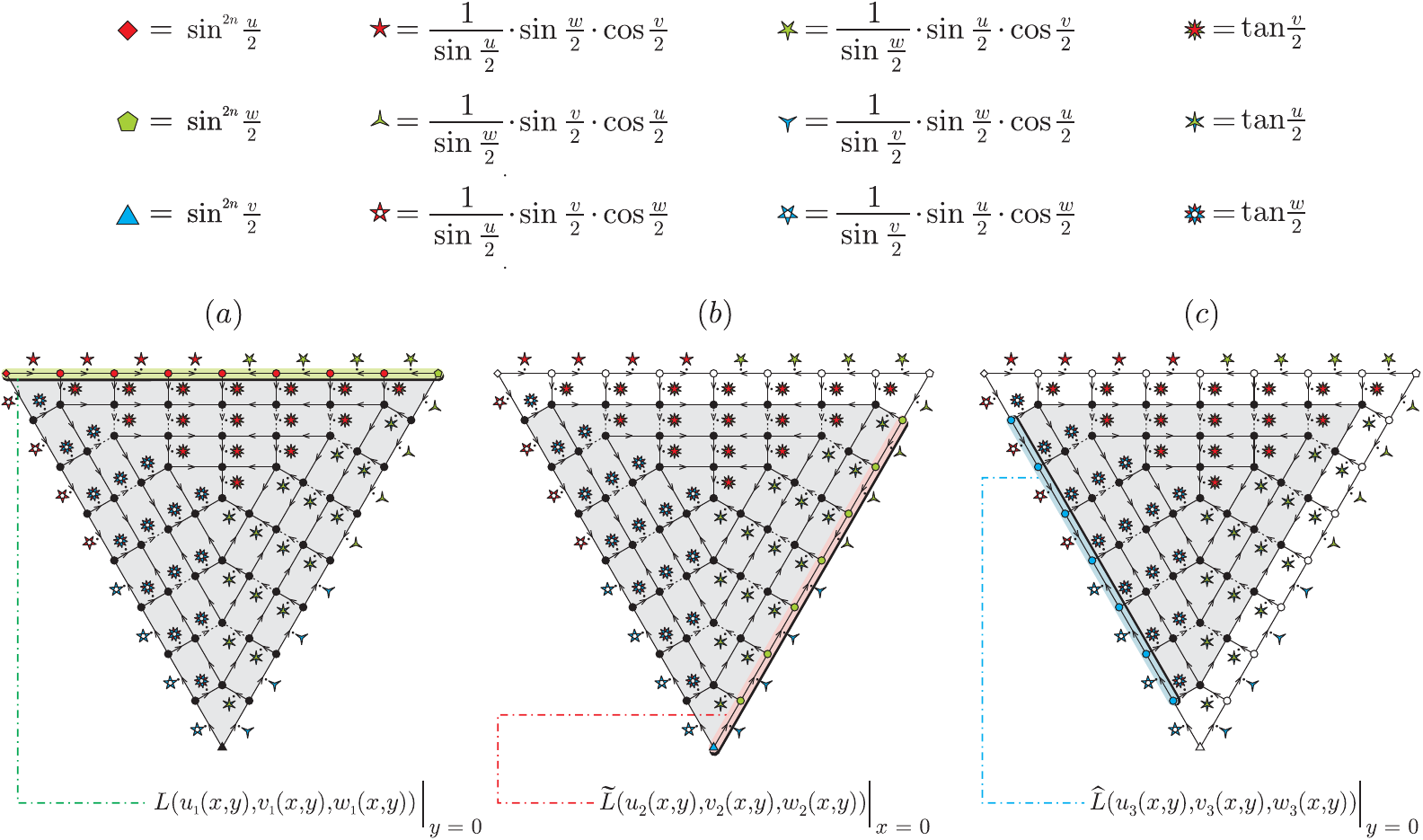}%
\caption{Cases (\emph{a}), (\emph{b}) and (\emph{c}) illustrate three
consecutive steps performed in order to gradually simplify equality
(\ref{linear_combination_union_null_equality}) to its final form
(\ref{linear_combination_union_reduced_null_equality_3}) via intermediate
equalities (\ref{linear_combination_union_reduced_null_equality_1}) and
(\ref{linear_combination_union_reduced_null_equality_2}). In each step, black
dots represent functions whose zeroth order partial derivatives vanish when
one evaluates the zeroth order partial derivatives of equalities
(\ref{linear_combination_union_null_equality}),
(\ref{linear_combination_union_reduced_null_equality_1}) and
(\ref{linear_combination_union_reduced_null_equality_2}) under
parametrizations (\ref{parametrization_1_y_0}), (\ref{parametrization_2_x_0})
and (\ref{parametrization_3_y_0}) with respect to variables $y$, $x$ and $y$,
at $y=0$, $x=0$ and $y=0$, respectively. Red, green and blue colored dots
represent functions the zeroth order partial derivatives of which do not
vanish under corresponding parametrizations and substitutions with $0$. White
dots correspond to functions the combining constants of which proved to be $0$
in the previous steps. The gray shaded shrinking areas show all those
functions which are present in the corresponding equalities. (For
interpretation of the references to color in this figure legend, the reader is
referred to the web version of this paper.)}%
\label{fig:linear_independence_j_0}%
\end{center}
\end{figure}

Now, we formulate and prove an induction hypothesis with respect to the $r$th
order ($r=0,1,\ldots,n-1$) partial derivative of the initial equality
(\ref{linear_combination_union_null_equality}) under successive application of
parametrizations (\ref{parametrization_1_y_0}), (\ref{parametrization_2_x_0})
and (\ref{parametrization_3_y_0}) with respect to variables $y$, $x$ and $y$,
at $y=0$, $x=0$ and $y=0$, respectively. Namely, we assume that the evaluation
of these partial derivatives gradually reduces the initial equality
(\ref{linear_combination_union_null_equality}) to the simpler form%
\begin{equation}%
\begin{array}
[c]{ccl}%
0 & = & L_{r}\left(  u,v,w\right)  \\
& = & \rho_{2n,n,n}R_{2n,n,n}^{\alpha}\left(  u,v,w\right)  +%
{\displaystyle\sum\limits_{j=r+1}^{n-1}}
{\displaystyle\sum\limits_{i=j}^{2n-1-j}}
\rho_{2n,2n-i,j}R_{2n,2n-i,j}^{\alpha}\left(  u,v,w\right)  \\
&  & +%
{\displaystyle\sum\limits_{j=r+1}^{n-1}}
{\displaystyle\sum\limits_{i=j}^{2n-1-j}}
\gamma_{2n,2n-i,j}G_{2n,2n-i,j}^{\alpha}\left(  u,v,w\right)  \\
&  & +%
{\displaystyle\sum\limits_{j=r+1}^{n-1}}
{\displaystyle\sum\limits_{i=j}^{2n-1-j}}
\beta_{2n,2n-i,j}B_{2n,2n-i,j}^{\alpha}\left(  u,v,w\right)  ,\forall\left(
u,v,w\right)  \in\Omega^{\alpha}%
\end{array}
\label{induction_hypothesis_for_linear_independence}%
\end{equation}
for all orders $r=0,1,\ldots,n-1$, where the coefficients of all discarded
functions are equal to $0$.

Hereafter, we prove our hypothesis from $r$ to $r+1$ for all orders
$r=0,1,\ldots,n-2$, i.e., we show that equality
(\ref{induction_hypothesis_for_linear_independence}) can be reduced to%
\begin{equation}%
\begin{array}
[c]{ccl}%
0 & = & L_{r+1}\left(  u,v,w\right) \\
& = & \rho_{2n,n,n}R_{2n,n,n}^{\alpha}\left(  u,v,w\right)  +%
{\displaystyle\sum\limits_{j=r+2}^{n-1}}
{\displaystyle\sum\limits_{i=j}^{2n-1-j}}
\rho_{2n,2n-i,j}R_{2n,2n-i,j}^{\alpha}\left(  u,v,w\right) \\
&  & +%
{\displaystyle\sum\limits_{j=r+2}^{n-1}}
{\displaystyle\sum\limits_{i=j}^{2n-1-j}}
\gamma_{2n,2n-i,j}G_{2n,2n-i,j}^{\alpha}\left(  u,v,w\right) \\
&  & +%
{\displaystyle\sum\limits_{j=r+2}^{n-1}}
{\displaystyle\sum\limits_{i=j}^{2n-1-j}}
\beta_{2n,2n-i,j}B_{2n,2n-i,j}^{\alpha}\left(  u,v,w\right)  ,\forall\left(
u,v,w\right)  \in\Omega^{\alpha},
\end{array}
\label{induction_hypothesis_nex_step}%
\end{equation}
by evaluating its $\left(  r+1\right)  $th order partial derivatives under
parametrizations (\ref{parametrization_1_y_0})--(\ref{parametrization_3_y_0})
with proper substitutions of $0$, where the combining constants $\left\{
\rho_{2n,2n-i,r+1},\gamma_{2n,2n-i,r+1},\beta_{2n,2n-i,r+1}\right\}
_{i=r+1}^{2n-1-\left(  r+1\right)  }$ of all the omitted functions are equal
to $0$.

Using parametrization (\ref{parametrization_1_y_0}) and the partial derivative
formula (\ref{derivative_formula_parametrization_1}) with respect to $y$, at
$y=0$ one has that%
\begin{equation}%
\begin{array}
[c]{ccl}%
0 & = & \left.  \dfrac{\partial^{r+1}}{\partial y^{r+1}}L_{r}\left(
u_{1}\left(  x,y\right)  ,v_{1}\left(  x,y\right)  ,w_{1}\left(  x,y\right)
\right)  \right\vert _{y=0}\\
& = &
{\displaystyle\sum\limits_{k=0}^{r+1}}
\dbinom{r+1}{k}\left(  -1\right)  ^{k}\left.  \dfrac{\partial^{r+1}}{\partial
v_{1}^{r+1-k}\partial w_{1}^{k}}L_{r}\left(  u_{1}\left(  x,y\right)
,v_{1}\left(  x,y\right)  ,w_{1}\left(  x,y\right)  \right)  \right\vert
_{y=0},~\forall x\in\left[  0,\alpha\right]  .
\end{array}
\label{y_derivative_L_r}%
\end{equation}
Observe that in case of the $k$th ($k=0,1,\ldots,r+1$) term of the summation
appearing in equality (\ref{y_derivative_L_r}) we can successively write\footnote{Technical details can be found in \cite{tech_report}. From hereon, a number in parenthesis above the equality sign indicates that we apply the corresponding trigonometric identity.}%
\begin{align*}
&  \left.  \frac{\partial^{r+1}}{\partial v_{1}^{r+1-k}\partial w_{1}^{k}%
}L_{r}\left(  u_{1}\left(  x,y\right)  ,v_{1}\left(  x,y\right)  ,w_{1}\left(
x,y\right)  \right)  \right\vert _{y=0}\\
& \\
& \\
= &  \rho_{2n,n,n}\left.  \sin^{n}\frac{u_{1}\left(  x,y\right)  }%
{2}\right\vert _{y=0}\cdot S_{n,\frac{1}{2}}^{r+1-k}\left(  0\right)
\cdot\left.  \frac{\text{d}^{k}}{\text{d}w_{1}^{k}}\sin^{n}\frac{w_{1}\left(
x,y\right)  }{2}\right\vert _{y=0}\\
& \\
&  +\sum_{j=r+1}^{n-1}\sum_{i=j}^{n}\rho_{2n,2n-i,j}\left.  \sin^{2n-i}%
\frac{u_{1}\left(  x,y\right)  }{2}\right\vert _{y=0}\cdot\left.
\frac{\text{d}^{k}}{\text{d}w_{1}^{k}}\sin^{i}\frac{w_{1}\left(  x,y\right)
}{2}\right\vert _{y=0}\cdot M_{i-j,\frac{1}{2},j,\frac{1}{2}}^{r+1-k}\left(
0\right)  \\
&  +\sum_{j=r+1}^{n-1}\sum_{i=n+1}^{2n-1-j}\rho_{2n,2n-i,j}\left.
\frac{\text{d}^{k}}{\text{d}w_{1}^{k}}\sin^{i}\frac{w_{1}\left(  x,y\right)
}{2}\right\vert _{y=0}\cdot\left.  \sin^{2n-i}\frac{u_{1}\left(  x,y\right)
}{2}\right\vert _{y=0}\cdot M_{2n-i-j,\frac{1}{2},j,\frac{1}{2}}%
^{r+1-k}\left(  0\right)  \\
& \\
&  +\sum_{j=r+1}^{n-1}\sum_{i=j}^{n}\gamma_{2n,2n-i,j}\left.  \frac
{\text{d}^{k}}{\text{d}w_{1}^{k}}\sin^{2n-i}\frac{w_{1}\left(  x,y\right)
}{2}\right\vert _{y=0}\cdot S_{i,\frac{1}{2}}^{r+1-k}\left(  0\right)
\cdot\left.  \left(  \cos^{i-j}\frac{u_{1}\left(  x,y\right)  }{2}\sin
^{j}\frac{u_{1}\left(  x,y\right)  }{2}\right)  \right\vert _{y=0}\\
&  +\sum_{j=r+1}^{n-1}\sum_{i=n+1}^{2n-1-j}\gamma_{2n,2n-i,j}S_{i,\frac{1}{2}%
}^{r+1-k}\left(  0\right)  \cdot\left.  \frac{\text{d}^{k}}{\text{d}w_{1}^{k}%
}\sin^{2n-i}\frac{w_{1}\left(  x,y\right)  }{2}\right\vert _{y=0}\cdot\left.
\left(  \cos^{2n-i-j}\frac{u_{1}\left(  x,y\right)  }{2}\sin^{j}\frac
{u_{1}\left(  x,y\right)  }{2}\right)  \right\vert _{y=0}\\
& \\
&  +\sum_{j=r+1}^{n-1}\sum_{i=j}^{n}\beta_{2n,2n-i,j}S_{2n-i,\frac{1}{2}%
}^{r+1-k}\left(  0\right)  \cdot\left.  \sin^{i}\frac{u_{1}\left(  x,y\right)
}{2}\right\vert _{y=0}\cdot\left.  \frac{\text{d}^{k}}{\text{d}w_{1}^{k}%
}\left(  \cos^{i-j}\frac{w_{1}\left(  x,y\right)  }{2}\sin^{j}\frac
{w_{1}\left(  x,y\right)  }{2}\right)  \right\vert _{y=0}\\
&  +\sum_{j=r+1}^{n-1}\sum_{i=n+1}^{2n-1-j}\beta_{2n,2n-i,j}\left.  \sin
^{i}\frac{u_{1}\left(  x,y\right)  }{2}\right\vert _{y=0}\cdot S_{2n-i,\frac
{1}{2}}^{r+1-k}\left(  0\right)  \cdot\left.  \frac{\text{d}^{k}}%
{\text{d}w_{1}^{k}}\left(  \cos^{2n-i-j}\frac{w_{1}\left(  x,y\right)  }%
{2}\sin^{j}\frac{w_{1}\left(  x,y\right)  }{2}\right)  \right\vert _{y=0}\\
& \\
\underset{\text{(\ref{derivative_cosine_sine_at_0})}}{\overset
{\text{(\ref{sign_sin_power_derivative})}}{=}} &  \left\{
\begin{array}
[c]{rl}%
S_{r+1,\frac{1}{2}}^{r+1}\left(  0\right)  \cdot\left(
{\displaystyle\sum\limits_{i=r+1}^{n}}
\rho_{2n,2n-i,r+1}\left.  \sin^{2n-i}\dfrac{u_{1}\left(  x,y\right)  }%
{2}\right\vert _{y=0}\cdot\left.  \sin^{i}\dfrac{w_{1}\left(  x,y\right)  }%
{2}\right\vert _{y=0}\right.   & \\
& \\
+%
{\displaystyle\sum\limits_{i=n+1}^{2n-1-\left(  r+1\right)  }}
\rho_{2n,2n-i,r+1}\left.  \sin^{i}\dfrac{w_{1}\left(  x,y\right)  }%
{2}\right\vert _{y=0}\cdot\left.  \sin^{2n-i}\dfrac{u_{1}\left(  x,y\right)
}{2}\right\vert _{y=0} & \\
& \\
\left.  +\gamma_{2n,2n-\left(  r+1\right)  ,r+1}\left.  \sin^{2n-\left(
r+1\right)  }\dfrac{w_{1}\left(  x,y\right)  }{2}\right\vert _{y=0}%
\cdot\left.  \sin^{r+1}\dfrac{u_{1}\left(  x,y\right)  }{2}\right\vert
_{y=0}\right)  , & k=0\\
& \\
0, & 1\leq k\leq r+1
\end{array}
\right.  \\
& \\
= &  \left\{
\begin{array}
[c]{rl}%
S_{r+1,\frac{1}{2}}^{r+1}\left(  0\right)  \cdot\left(
{\displaystyle\sum\limits_{i=r+1}^{n}}
\rho_{2n,2n-i,r+1}\sin^{2n-i}\dfrac{x}{2}\sin^{i}\dfrac{\alpha-x}{2}\right.
& \\
& \\
+%
{\displaystyle\sum\limits_{i=n+1}^{2n-1-\left(  r+1\right)  }}
\rho_{2n,2n-i,r+1}\sin^{i}\dfrac{\alpha-x}{2}\sin^{2n-i}\dfrac{x}{2} & \\
& \\
\left.  +\gamma_{2n,2n-\left(  r+1\right)  ,r+1}\sin^{2n-\left(  r+1\right)
}\dfrac{\alpha-x}{2}\sin^{r+1}\dfrac{x}{2}\right)  , & k=0,\\
& \\
0, & 1\leq k\leq r+1
\end{array}
\right.  \\
& \\
= &  \left\{
\begin{array}
[c]{rl}%
S_{r+1,\frac{1}{2}}^{r+1}\left(  0\right)  \cdot%
{\displaystyle\sum\limits_{i=r+1}^{2n-1-\left(  r+1\right)  }}
\rho_{2n,2n-i,r+1}^{\prime}A_{2n,2n-i}^{\alpha}\left(  x\right)  , & k=0,\\
& \\
0, & 1\leq k\leq r+1
\end{array}
\right.
\end{align*}
for all values of $x\in\left[  0,\alpha\right]  $, where we have used the
notations%
\[
\rho_{2n,2n-i,r+1}^{\prime}=\left\{
\begin{array}
[c]{rl}%
\dfrac{\rho_{2n,2n-i,r+1}}{c_{2n,2n-i}^{\alpha}}, & i=r+1,r+2,\ldots
,2n-1-\left(  r+1\right)  ,\\
& \\
\dfrac{\gamma_{2n,2n-\left(  r+1\right)  ,r+1}}{c_{2n,2n-\left(  r+1\right)
}^{\alpha}}, & i=2n-\left(  r+1\right)  .
\end{array}
\right.
\]
Therefore, equality (\ref{y_derivative_L_r}) can be reduced to the form%
\[
0=S_{r+1,\frac{1}{2}}^{r+1}\left(  0\right)  \cdot%
{\displaystyle\sum\limits_{i=r+1}^{2n-1-\left(  r+1\right)  }}
\rho_{2n,2n-i,r+1}^{\prime}A_{2n,2n-i}^{\alpha}\left(  x\right)  ,~\forall
x\in\left[  0,\alpha\right]  ,
\]
from which, by using the linear independence of the function system $\left\{
A_{2n,2n-i}^{\alpha}\left(  x\right)  :x\in\left[  0,\alpha\right]  \right\}
_{i=r+1}^{2n-1-\left(  r+1\right)  }$, one obtains that%
\[
\rho_{2n,2n-i,n-1}^{\prime}=0,~i=r+1,r+2,\ldots,2n-\left(  r+1\right)  ,
\]
i.e., equality (\ref{induction_hypothesis_for_linear_independence}) can be
simplified to%
\begin{equation}%
\begin{array}
[c]{ccl}%
0 & = & \widetilde{L}_{r}\left(  u,v,w\right)  \\
& = & \rho_{2n,n,n}R_{2n,n,n}^{\alpha}\left(  u,v,w\right)  +%
{\displaystyle\sum\limits_{j=r+2}^{n-1}}
{\displaystyle\sum\limits_{i=j}^{2n-1-j}}
\rho_{2n,2n-i,j}R_{2n,2n-i,j}^{\alpha}\left(  u,v,w\right)  \\
&  & +%
{\displaystyle\sum\limits_{i=r+2}^{2n-1-\left(  r+1\right)  }}
\gamma_{2n,2n-i,r+1}G_{2n,2n-i,r+1}^{\alpha}\left(  u,v,w\right)  +%
{\displaystyle\sum\limits_{j=r+2}^{n-1}}
{\displaystyle\sum\limits_{i=j}^{2n-1-j}}
\gamma_{2n,2n-i,j}G_{2n,2n-i,j}^{\alpha}\left(  u,v,w\right)  \\
&  & +%
{\displaystyle\sum\limits_{j=r+1}^{n-1}}
{\displaystyle\sum\limits_{i=j}^{2n-1-j}}
\beta_{2n,2n-i,j}B_{2n,2n-i,j}^{\alpha}\left(  u,v,w\right)  ,~\forall\left(
u,v,w\right)  \in\Omega^{\alpha}.
\end{array}
\label{reduced_y_derivative_L_r_tilde}%
\end{equation}

Switching to parametrization (\ref{parametrization_2_x_0}) and evaluating the
$\left(  r+1\right)  $th order partial derivative of equality
(\ref{reduced_y_derivative_L_r_tilde}) with respect to $x$, at $x=0$, we have%
\begin{equation}%
\begin{array}
[c]{ccl}%
0 & = & \left.  \dfrac{\partial^{r+1}}{\partial x^{r+1}}\widetilde{L}%
_{r}\left(  u_{2}\left(  x,y\right)  ,v_{2}\left(  x,y\right)  ,w_{2}\left(
x,y\right)  \right)  \right\vert _{x=0}\\
& = &
{\displaystyle\sum\limits_{k=0}^{r+1}}
\dbinom{r+1}{k}\left(  -1\right)  ^{k}\left.  \dfrac{\partial^{r+1}}{\partial
u_{2}^{r+1-k}\partial v_{2}^{k}}\widetilde{L}_{r}\left(  u_{2}\left(
x,y\right)  ,v_{2}\left(  x,y\right)  ,w_{2}\left(  x,y\right)  \right)
\right\vert _{x=0},~\forall y\in\left[  0,\alpha\right]  ,
\end{array}
\label{x_derivative_L_tilde_r}%
\end{equation}
the $k$th term ($k=0,1,\ldots,r+1$) of which can be expressed as\footnote{Technical details can be found in \cite{tech_report}.}%
\begin{align*}
&  \left.  \frac{\partial^{r+1}}{\partial u_{2}^{r+1-k}\partial v_{2}^{k}%
}\widetilde{L}_{r}\left(  u_{2}\left(  x,y\right)  ,v_{2}\left(  x,y\right)
,w_{2}\left(  x,y\right)  \right)  \right\vert _{x=0}\\
& \\
\underset{\text{(\ref{derivative_cosine_sine_at_0})}}{\overset
{\text{(\ref{sign_sin_power_derivative})}}{=}}  &  \left\{
\begin{array}
[c]{rl}%
S_{r+1,\frac{1}{2}}^{r+1}\left(  0\right)  \cdot\left(
{\displaystyle\sum\limits_{i=r+2}^{n}}
\gamma_{2n,2n-i,r+1}\left.  \sin^{2n-i}\dfrac{w_{2}\left(  x,y\right)  }%
{2}\right\vert _{x=0}\cdot\left.  \sin^{i}\dfrac{v_{2}\left(  x,y\right)  }%
{2}\right\vert _{x=0}\right.  & \\
& \\
+%
{\displaystyle\sum\limits_{i=n+1}^{2n-1-\left(  r+1\right)  }}
\gamma_{2n,2n-i,r+1}\left.  \sin^{i}\dfrac{v_{2}\left(  x,y\right)  }%
{2}\right\vert _{x=0}\cdot\left.  \sin^{2n-i}\dfrac{w_{2}\left(  x,y\right)
}{2}\right\vert _{x=0} & \\
& \\
\left.  +\beta_{2n,2n-\left(  r+1\right)  ,r+1}\left.  \sin^{2n-\left(
r+1\right)  }\dfrac{v_{2}\left(  x,y\right)  }{2}\right\vert _{x=0}%
\cdot\left.  \sin^{r+1}\dfrac{w_{2}\left(  x,y\right)  }{2}\right\vert
_{x=0}\right)  , & k=0,\\
& \\
0, & 1\leq k\leq r+1
\end{array}
\right. \\
& \\
=  &  \left\{
\begin{array}
[c]{rl}%
S_{r+1,\frac{1}{2}}^{r+1}\left(  0\right)  \cdot%
{\displaystyle\sum\limits_{i=r+2}^{2n-\left(  r+1\right)  }}
\gamma_{2n,2n-i,r+1}^{\prime}A_{2n,2n-i}^{\alpha}\left(  y\right)  , & k=0,\\
& \\
0, & 1\leq k\leq r+1
\end{array}
\right.
\end{align*}
for all $y\in\left[  0,\alpha\right]  $, where we have introduced the
notations%
\[
\gamma_{2n,2n-i,r+1}^{\prime}=\left\{
\begin{array}
[c]{rl}%
\dfrac{\gamma_{2n,2n-i,r+1}}{c_{2n,2n-i}^{\alpha}}, & i=r+2,r+3,\ldots
,2n-1-\left(  r+1\right)  ,\\
& \\
\dfrac{\beta_{2n,2n-\left(  r+1\right)  ,r+1}}{c_{2n,2n-\left(  r+1\right)
}^{\alpha}}, & i=2n-\left(  r+1\right)  .
\end{array}
\right.
\]

Thus, equality (\ref{x_derivative_L_tilde_r}) can be simplified to the form%
\[
0=S_{r+1,\frac{1}{2}}^{r+1}\left(  0\right)  \cdot%
{\displaystyle\sum\limits_{i=r+2}^{2n-\left(  r+1\right)  }}
\gamma_{2n,2n-i,r+1}^{\prime}A_{2n,2n-i}^{\alpha}\left(  y\right)  ,~\forall
y\in\left[  0,\alpha\right]  .
\]
Since functions $\left\{  A_{2n,2n-i}^{\alpha}\left(  y\right)  :y\in\left[
0,\alpha\right]  \right\}  _{i=r+2}^{2n-\left(  r+1\right)  }$ are linearly
independent, we obtain that%
\[
\gamma_{2n,2n-i,n-1}^{\prime}=0,~i=r+2,r+3,\ldots,2n-\left(  r+1\right)  ,
\]
therefore equality (\ref{reduced_y_derivative_L_r_tilde}) can be reduced to%
\begin{equation}%
\begin{array}
[c]{ccl}%
0 & = & \widehat{L}_{r}\left(  u,v,w\right)  \\
&  & \\
& = & \rho_{2n,n,n}R_{2n,n,n}^{\alpha}\left(  u,v,w\right)  +%
{\displaystyle\sum\limits_{j=r+2}^{n-1}}
{\displaystyle\sum\limits_{i=j}^{2n-1-j}}
\rho_{2n,2n-i,j}R_{2n,2n-i,j}^{\alpha}\left(  u,v,w\right)  \\
&  & \\
&  & +%
{\displaystyle\sum\limits_{j=r+2}^{n-1}}
{\displaystyle\sum\limits_{i=j}^{2n-1-j}}
\gamma_{2n,2n-i,j}G_{2n,2n-i,j}^{\alpha}\left(  u,v,w\right)  +%
{\displaystyle\sum\limits_{i=r+2}^{2n-1-\left(  r+1\right)  }}
\beta_{2n,2n-i,r+1}B_{2n,2n-i,r+1}^{\alpha}\left(  u,v,w\right)  \\
&  & \\
&  & +%
{\displaystyle\sum\limits_{j=r+2}^{n-1}}
{\displaystyle\sum\limits_{i=j}^{2n-1-j}}
\beta_{2n,2n-i,j}B_{2n,2n-i,j}^{\alpha}\left(  u,v,w\right)  ,~\forall\left(
u,v,w\right)  \in\Omega^{\alpha}.
\end{array}
\label{reduced_x_derivative_L_r_hat}%
\end{equation}
Finally, we apply parametrization (\ref{parametrization_3_y_0}) and evaluate
the $\left(  r+1\right)  $th order partial derivative of equality
(\ref{reduced_x_derivative_L_r_hat}) with respect to $y$, at $y=0$. Using the
derivative formula (\ref{derivative_formula_parametrization_3}), we have that%
\begin{equation}%
\begin{array}
[c]{ccl}%
0 & = & \left.  \dfrac{\partial^{r+1}}{\partial y^{r+1}}\widehat{L}_{r}\left(
u_{3}\left(  x,y\right)  ,v_{3}\left(  x,y\right)  ,w_{3}\left(  x,y\right)
\right)  \right\vert _{y=0}\\
& = &
{\displaystyle\sum\limits_{k=0}^{r+1}}
\dbinom{r+1}{k}\left(  -1\right)  ^{k}\left.  \dfrac{\partial^{r+1}}{\partial
w_{3}^{r+1-k}\partial u_{3}^{k}}\widehat{L}_{r}\left(  u_{3}\left(
x,y\right)  ,v_{3}\left(  x,y\right)  ,w_{3}\left(  x,y\right)  \right)
\right\vert _{y=0},~\forall x\in\left[  0\,,\alpha\right]  ,
\end{array}
\label{y_derivative_L_hat_n_minus_2}%
\end{equation}
where the $k$th term ($k=0,1,\ldots,r+1$) takes the form\footnote{Technical details can be found in \cite{tech_report}.}%
\begin{align*}
&  \left.  \dfrac{\partial^{r+1}}{\partial w_{3}^{r+1-k}\partial u_{3}^{k}%
}\widehat{L}_{r}\left(  u_{3}\left(  x,y\right)  ,v_{3}\left(  x,y\right)
,w_{3}\left(  x,y\right)  \right)  \right\vert _{y=0}\\
& \\
= &  \left\{
\begin{array}
[c]{rl}%
S_{r+1,\frac{1}{2}}^{r+1}\left(  0\right)  \cdot%
{\displaystyle\sum\limits_{i=r+2}^{2n-1-\left(  r+1\right)  }}
\beta_{2n,2n-i,r+1}^{\prime}A_{2n,2n-i}^{\alpha}\left(  x\right)  , & k=0,\\
& \\
0, & 1\leq k\leq r+1
\end{array}
\right.
\end{align*}
for all $x\in\left[  0,\alpha\right]  $, where we have used the notations%
\[
\beta_{2n,2n-i,r+1}^{\prime}=\frac{\beta_{2n,2n-i,r+1}}{c_{2n,2n-i}^{\alpha}%
},~i=r+2,r+3,\ldots,2n-1-\left(  r+1\right)  .
\]
Therefore, equality (\ref{reduced_x_derivative_L_r_hat}) can be reduced to%
\[
0=S_{r+1,\frac{1}{2}}^{r+1}\left(  0\right)  \cdot%
{\displaystyle\sum\limits_{i=r+2}^{2n-1-\left(  r+1\right)  }}
\beta_{2n,2n-i,r+1}^{\prime}A_{2n,2n-i}^{\alpha}\left(  x\right)  ,~\forall
x\in\left[  0,\alpha\right]  ,
\]
which implies that%
\[
\beta_{2n,2n-i,r+1}^{\prime}=0,~i=r+2,r+3,\ldots,2n-1-\left(  r+1\right)  .
\]

Summarizing all calculations and zero constants obtained above, by evaluating
the $\left(  r+1\right)  $th order ($r=0,1,\ldots,n-2$) partial derivatives
under successive parametrizations (\ref{parametrization_1_y_0}%
)--(\ref{parametrization_3_y_0}) of the initial equality
(\ref{induction_hypothesis_for_linear_independence}), we conclude that
equality (\ref{induction_hypothesis_nex_step}) is also valid, i.e., the
induction hypothesis (\ref{induction_hypothesis_for_linear_independence}) is
correct for all orders $r=0,1,\ldots,n-2$.

The special case $r=n-2$ implies that equality%
\begin{equation}%
\begin{array}
[c]{ccl}%
0 & = & L_{n-2}\left(  u,v,w\right) \\
& = & \rho_{2n,n,n}R_{2n,n,n}^{\alpha}\left(  u,v,w\right)  +%
{\displaystyle\sum\limits_{i=n-1}^{n}}
\rho_{2n,2n-i,n-1}R_{2n,2n-i,n-1}^{\alpha}\left(  u,v,w\right) \\
&  & +%
{\displaystyle\sum\limits_{i=n-1}^{n}}
\gamma_{2n,2n-i,n-1}G_{2n,2n-i,n-1}^{\alpha}\left(  u,v,w\right) \\
&  & +%
{\displaystyle\sum\limits_{i=n-1}^{n}}
\beta_{2n,2n-i,n-1}B_{2n,2n-i,n-1}^{\alpha}\left(  u,v,w\right)
,~\forall\left(  u,v,w\right)  \in\Omega^{\alpha}%
\end{array}
\label{induction_hypothesis_for_linear_independence_n_minus_2}%
\end{equation}
can be simplified to the form%
\begin{equation}
0=L_{n-1}\left(  u,v,w\right)  =\rho_{2n,n,n}R_{2n,n,n}^{\alpha}\left(
u,v,w\right)  ,~\forall\left(  u,v,w\right)  \in\Omega^{\alpha},
\label{induction_hypothesis_for_linear_independence_n_minus_1}%
\end{equation}
where the combining constants of all discarded functions proved to be $0$.
Moreover, equality
(\ref{induction_hypothesis_for_linear_independence_n_minus_1}) holds if and
only if%
\[
\rho_{2n,n,n}=0,
\]
which means that all combining constants appearing in the initial vanishing
linear combination (\ref{linear_combination_union_null_equality}) have to be
zero, i.e., the joint constrained trivariate function system
(\ref{united_system}) is linearly independent.%

\begin{figure}
[!htb]
\begin{center}
\includegraphics[
height=6.4135in,
width=6.4437in
]%
{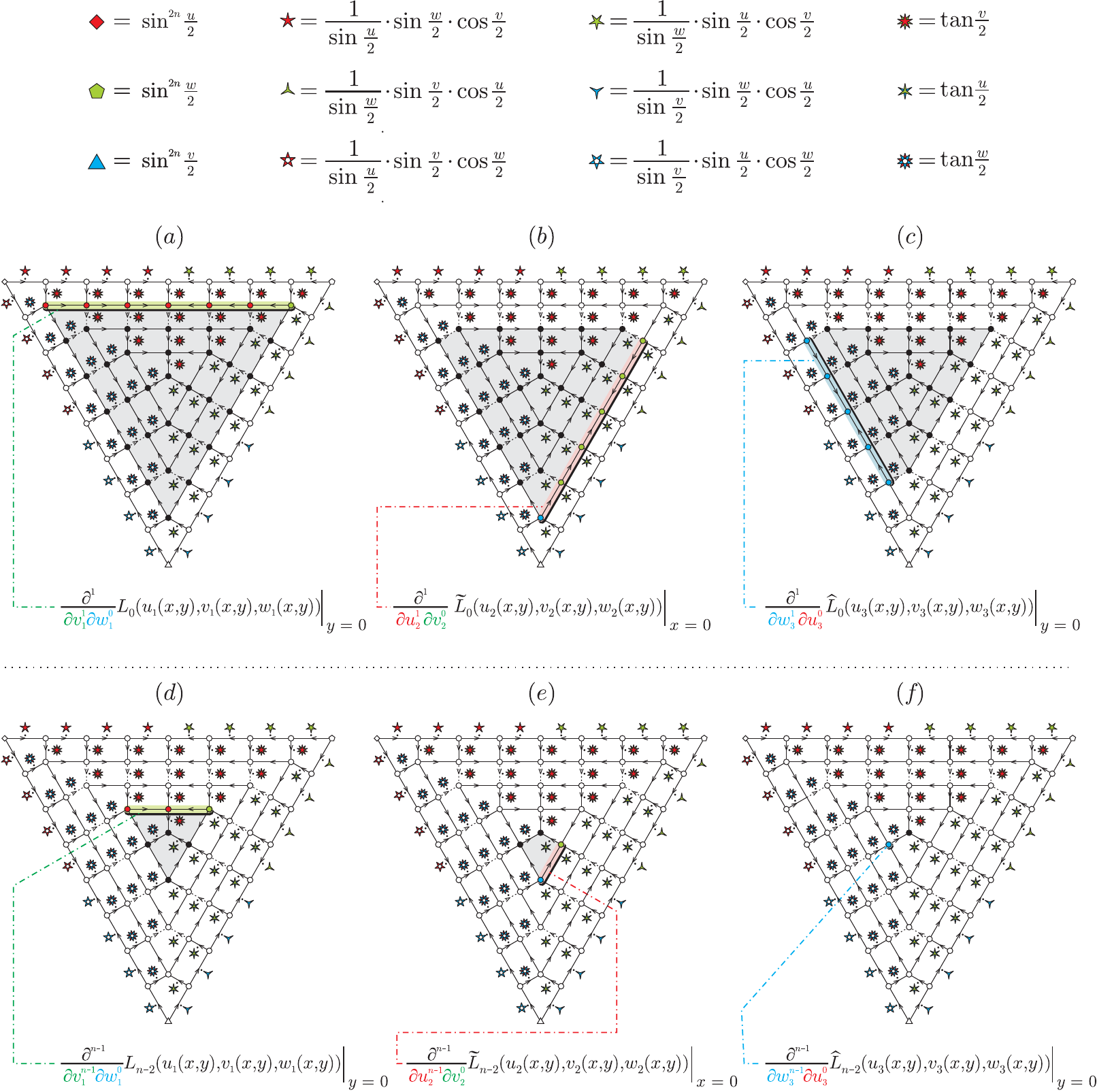}%
\caption{Cases (\emph{a})--(\emph{c}) and (\emph{d})--(\emph{f}) illustrate
three consecutive steps performed in order to gradually simplify equality
(\ref{induction_hypothesis_for_linear_independence}) to its final form
(\ref{induction_hypothesis_nex_step}) via intermediate equalities
(\ref{reduced_y_derivative_L_r_tilde}) and (\ref{reduced_x_derivative_L_r_hat}%
) for $r=0$ and $r=n-2$, respectively. In each step of cases (\emph{a}%
)--(\emph{c}) (resp. (\emph{d})--(\emph{f})), black dots represent functions
whose first (resp. $\left(  n-1\right)  $th) order partial derivatives vanish
when one evaluates the first (resp. $\left(  n-1\right)  $th) order partial
derivatives of equalities (\ref{induction_hypothesis_for_linear_independence}%
), (\ref{reduced_y_derivative_L_r_tilde}) and
(\ref{reduced_x_derivative_L_r_hat}) under parametrizations
(\ref{parametrization_1_y_0}), (\ref{parametrization_2_x_0}) and
(\ref{parametrization_3_y_0}) with respect to variables $y$, $x$ and $y$, at
$y=0$, $x=0$ and $y=0$, respectively. Red, green and blue colored dots
represent functions the first (resp. $\left(  n-1\right)  $th) order partial
derivatives of which do not vanish under corresponding parametrizations and
substitutions with $0$. White dots correspond to functions the coefficients of
which proved to be $0$ in the previous steps. The gradually shrinking gray
shaded areas correspond to functions which appear in the first (resp. $\left(
n-1\right)  $th) order partial derivatives of equalities
(\ref{induction_hypothesis_for_linear_independence}),
(\ref{reduced_y_derivative_L_r_tilde}) and (\ref{reduced_x_derivative_L_r_hat}%
). (For interpretation of the references to color in this figure legend, the
reader is referred to the web version of this paper.)}%
\label{fig:linear_independence_j_1__j_n_minus_1}%
\end{center}
\end{figure}
Cases (\emph{a})--(\emph{c}) and (\emph{d})--(\emph{f}) of Fig.\ \ref{fig:linear_independence_j_1__j_n_minus_1} represent calculations that
correspond to the evaluation of the first and $\left(  n-1\right)  $th order
partial derivatives of equalities
\[
0=L_{0}\left(  u,v,w\right)  ,~\forall\left(  u,v,w\right)  \in\Omega^{\alpha}%
\]
and%
\[
0=L_{n-2}\left(  u,v,w\right)  ,~\forall\left(  u,v,w\right)  \in
\Omega^{\alpha}%
\]
under successive parametrizations (\ref{parametrization_1_y_0}%
)--(\ref{parametrization_3_y_0}), respectively.\qedhere
\end{proof}

Next we introduce an equivalence relation by means of which one can
recursively construct the linearly independent function system $V_{n}^{\alpha
}$ that generates the constrained trivariate extension $\mathcal{V}%
_{n}^{\alpha}$ ($n\in%
\mathbb{N}
$) of the vector space (\ref{truncated_Fourier_vector_space}).

\section{\label{sec:coincidence}Coincidence of vector spaces $\mathcal{T}%
_{2n}^{\alpha}$ and $\mathcal{V}_{n}^{\alpha}$}

Usually, the coincidence of some vector spaces is proved by means of (inverse) basis transformations (i.e., non-singular matrices) between the underlying vector spaces. In our case this approach proved to be very complicated and for the present we could describe these basis transformations only for orders one and two. In order to treat this problem in general, we used an alternative mathematical tool based on equivalence classes generated by the following equivalence relation.

\begin{definition}
[\textbf{Equivalence relation and equivalence classes}]%
\label{equivalence_relation_definition}Let $n\geq1$ and $\left(  u,v,w\right)
\in\Omega^{\alpha}$ be fixed parameters and consider the set%
\[
E_{n}^{\alpha}=\left\{  ru+gv+bw:r,g,b\in%
\mathbb{N}
,~0\leq r,g,b\leq n\right\}  .
\]
We say that combinations $r_{1}u+g_{1}v+b_{1}w\in E_{n}^{\alpha}$ and
$r_{2}u+g_{2}v+b_{2}w\in E_{n}^{\alpha}$ are equivalent, i.e.,
\[
r_{1}u+g_{1}v+b_{1}w\sim r_{2}u+g_{2}v+b_{2}w,
\]
if and only if there exists an integer $z\in\left\{  -n,-\left(  n-1\right)
,\ldots,0,1,\ldots,2n\right\}  $ for which one of the conditions%
\begin{equation}
\left\{
\begin{array}
[c]{rcl}%
r_{2}+r_{1} & = & z,\\
g_{2}+g_{1} & = & z,\\
b_{2}+b_{1} & = & z
\end{array}
\right.  \label{equivalence_conditions_1}%
\end{equation}
and%
\begin{equation}
\left\{
\begin{array}
[c]{rcl}%
r_{2}-r_{1} & = & z,\\
g_{2}-g_{1} & = & z,\\
b_{2}-b_{1} & = & z
\end{array}
\right.  \label{equivalence_conditions_2}%
\end{equation}
is fulfilled. Furthermore, let us denote by%
\[
E_{n}^{\alpha}/_{\sim}=\left\{  \left[  ru+gv+bw\right]  :ru+gv+bw\in
E_{n}^{\alpha}\right\}
\]
the set of all possible equivalence classes of $E_{n}^{\alpha}$ by $\sim$.
\end{definition}

Definition \ref{equivalence_relation_definition} is motivated by the following
reason. If, for example, one wants to check the relationship between constrained
trivariate functions $\cos\left(  r_{1}u+g_{1}v+b_{1}w\right)  $, $\sin\left(
r_{1}u+g_{1}v+b_{1}w\right)  $ and $\cos\left(  r_{2}u+g_{2}v+b_{2}w\right)
$, where $\left(  u,v,w\right)  \in\Omega^{\alpha}$ and combinations
$r_{1}u+g_{1}v+b_{1}w,r_{2}u+g_{2}v+b_{2}w\in E_{n}^{\alpha}$ are equivalent
by $\sim$, then exists an integer $z\in\left\{  -n,-\left(  n-1\right)
,\ldots,0,1,\ldots,2n\right\}  $ for which one of the equalities%
\begin{align*}
\cos\left(  r_{2}u+g_{2}v+b_{2}w\right)  =  &  \cos\left(  \left(  z\mp
r_{1}\right)  u+\left(  z\mp g_{1}\right)  v+\left(  z\mp b_{1}\right)
w\right) \\
=  &  \cos\left(  z\left(  u+v+w\right)  \mp\left(  r_{1}u+g_{1}%
v+b_{1}w\right)  \right) \\
=  &  \cos\left(  z\alpha\mp\left(  r_{1}u+g_{1}v+b_{1}w\right)  \right) \\
=  &  \cos\left(  za\right)  \cdot\cos\left(  r_{1}u+g_{1}v+b_{1}w\right)
\pm\sin\left(  za\right)  \cdot\sin\left(  r_{1}u+g_{1}v+b_{1}w\right)
\end{align*}
holds for all $\left(  u,v,w\right)  \in\Omega^{\alpha}$, i.e., function
$\cos\left(  r_{2}u+g_{2}v+b_{2}w\right)  $ differs from $\cos\left(
r_{1}u+g_{1}v+b_{1}w\right)  $ only in a phase change and it can be expressed
as a linear combination of $\cos\left(  r_{1}u+g_{1}v+b_{1}w\right)  $ and
$\sin\left(  r_{1}u+g_{1}v+b_{1}w\right)  $, hence functions $\cos\left(
r_{1}u+g_{1}v+b_{1}w\right)  $ and $\cos\left(  r_{2}u+g_{2}v+b_{2}w\right)  $
can be considered equivalent (up to a phase change), formally this means that
both selected combinations belong to the same equivalence class, i.e.,
$r_{2}u+g_{2}v+b_{2}w\in\left[  r_{1}u+g_{1}v+b_{1}w\right]  $. E.g. if $n=3$,
then
\[
0\cdot u+0\cdot v+1\cdot w\sim1\cdot u+1\cdot v+0\cdot w
\]
and
\[
1\cdot u+3\cdot v+2\cdot w\sim2\cdot u+0\cdot v+1\cdot w\sim0\cdot u+2\cdot
v+1\cdot w,
\]
but
\[
1\cdot u+0\cdot v+0\cdot w\nsim0\cdot u+1\cdot v+0\cdot w.
\]
Naturally, the equivalence relation $\sim$ introduced in Definition
\ref{equivalence_relation_definition} is also able to check the linear
dependence of constrained trivariate functions $\cos\left(  r_{1}%
u+g_{1}v+b_{1}w\right)  $, $\sin\left(  r_{1}u+g_{1}v+b_{1}w\right)  $ and
$\sin\left(  r_{2}u+g_{2}v+b_{2}w\right)  $ defined over $\Omega^{\alpha}$. If
$r_{1}u+g_{1}v+b_{1}w\sim r_{2}u+g_{2}v+b_{2}w$, then exists an integer $z$
such that one of the equalities%
\begin{align*}
\sin\left(  r_{2}u+g_{2}v+b_{2}w\right)   &  =\sin\left(  z\alpha\mp\left(
r_{1}u+g_{1}v+b_{1}w\right)  \right) \\
&  =\sin\left(  za\right)  \cdot\cos\left(  r_{1}u+g_{1}v+b_{1}w\right)
\mp\cos\left(  z\alpha\right)  \cdot\sin\left(  r_{1}u+g_{1}v+b_{1}w\right)
\end{align*}
holds for all $\left(  u,v,w\right)  \in\Omega^{\alpha}$, i.e., functions
$\sin\left(  r_{1}u+g_{1}v+b_{1}w\right)  $ and $\sin\left(  r_{2}%
u+g_{2}v+b_{2}w\right)  $ can be considered equivalent (up to a phase change)
over $\Omega^{\alpha}$.

In what follows, we will recursively construct the constrained trivariate
extension $\mathcal{V}_{n}^{\alpha}$ ($n\in%
\mathbb{N}
$) of the vector space (\ref{truncated_Fourier_vector_space}). Let
\[
E_{0}^{\alpha}=\left\{  0\cdot u+0\cdot v+0\cdot w:\left(  u,v,w\right)
\in\Omega^{\alpha}\right\}
\]
and $V_{0}^{\alpha}$ be the system of those non-vanishing constrained
trivariate cosine and sine functions which are determined by the quotient set
$E_{0}^{\alpha}/_{\sim}$ of different equivalence classes, i.e.,
\[
V_{0}^{\alpha}=\left\{  1:\left(  u,v,w\right)  \in\Omega^{\alpha}\right\}  .
\]
Define the vector space $\mathcal{V}_{0}^{\alpha}$ as span~$V_{0}^{\alpha}$.
In order to determine the vector space $\mathcal{V}_{1}^{\alpha}$ complete the
function system $V_{0}^{\alpha}$ with those constrained trivariate cosine and
sine functions defined over $\Omega^{\alpha}$ the arguments of which are
representatives of different equivalence classes%
\[
\left[  ru+gv+bw\right]  \in E_{1}^{\alpha}/_{\sim}\setminus E_{0}^{\alpha
}/_{\sim},
\]
where at least one of the coefficients $r,g,b$ is equal to $1$. Since $1\cdot
u+1\cdot v+1\cdot w\sim0\cdot u+0\cdot v+0\cdot w\in\left[  0\right]  \in
E_{0}^{\alpha}/_{\sim}$, $u+v\sim w$, $u+w\sim v$ and $v+w\sim u$, it is easy
to observe that $\mathcal{V}_{1}^{\alpha}=~$span~$V_{1}^{\alpha}$, where%
\[
V_{1}^{\alpha}=V_{0}^{\alpha}\cup\left\{  \cos\left(  u\right)  ,\sin\left(
u\right)  ,\cos\left(  v\right)  ,\sin\left(  v\right)  ,\cos\left(  w\right)
,\sin\left(  w\right)  :\left(  u,v,w\right)  \in\Omega^{\alpha}\right\}  .
\]

Continuing this process, the vector space $\mathcal{V}_{2}^{\alpha}$ can be
obtained by completing the function system $V_{1}^{\alpha}$ with those
constrained trivariate cosine and sine functions defined over $\Omega^{\alpha
}$ the arguments of which are representatives of different equivalence classes%
\[
\left[  ru+gv+bw\right]  \in E_{2}^{\alpha}/_{\sim}\setminus\left(
E_{0}^{\alpha}/_{\sim}\cup E_{1}^{\alpha}/_{\sim}\right)  ,
\]
where at least one of the coefficients $r,g,b$ is equal to $2$. Assume that
$b=2$. If $r=g=2$, then $\left[  2\cdot u+2\cdot v+2\cdot w\right]  =\left[
0\cdot u+0\cdot v+0\cdot w\right]  =\left[  0\right]  \in E_{0}^{\alpha
}/_{\sim}$. If $0<r\leq g\leq2$, i.e., if $r=1,g\in\left\{  1,2\right\}  $,
then let $z=\min\left\{  r,g\right\}  =1$ and observe that%
\begin{align*}
&  ru+gv+2w\\
=  &  \left(  r-z\right)  u+\left(  g-z\right)  v+\left(  2-z\right)
w+z\left(  u+v+w\right) \\
=  &  \left(  r-z\right)  u+\left(  g-z\right)  v+\left(  2-z\right)
w+z\alpha\\
=  &  \left(  g-z\right)  v+w+z\alpha,
\end{align*}
from which one obtains that
\[
u+gv+2w\in\left[  \left(  g-z\right)  v+w\right]  =\left\{
\begin{array}
[c]{rc}%
\left[  w\right]  \in E_{1}^{\alpha}/_{\sim,} & g=1,\\
\left[  v+w\right]  =\left[  u\right]  \in E_{1}^{\alpha}/_{\sim,} & g=2.
\end{array}
\right.
\]
If $r=0$ and $g\in\left\{  0,1,2\right\}  $, then $\left[  gv+2w\right]  \in
E_{2}^{\alpha}/_{\sim}\setminus\left(  E_{0}^{\alpha}/_{\sim}\cup
E_{1}^{\alpha}/_{\sim}\right)  $. Thus, equality $b=2$ generates only three
acceptable equivalence classes, namely $\left\{  \left[  gv+2w\right]
\right\}  _{g=0}^{2}$, from which $\left[  2v+2w\right]  $ must be ignored,
because this equivalence class will reappear as $\left[  2u\right]  \in
E_{2}^{\alpha}/_{\sim}\setminus\left(  E_{0}^{\alpha}/_{\sim}\cup
E_{1}^{\alpha}/_{\sim}\right)  $, when we start the characterization detailed
above with fixed equality $r=2$. By cyclic symmetry, we can conclude that in
case of $n=2$ only equivalence classes%
\begin{align*}
&  \left\{  \left[  0\cdot u+g\cdot v+2\cdot w\right]  \right\}  _{g=0}^{1}\\
&  \left\{  \left[  2\cdot u+0\cdot v+b\cdot w\right]  \right\}  _{b=0}^{1},\\
&  \left\{  \left[  r\cdot u+2\cdot v+0\cdot w\right]  \right\}  _{r=0}^{1},
\end{align*}
generate new linearly independent constrained trivariate cosine and sine functions over
$\Omega^{\alpha}$, i.e., $\mathcal{V}_{2}^{\alpha}=~$span~$V_{2}^{\alpha}$,
where%
\begin{align*}
V_{2}^{\alpha}=  &  V_{1}^{\alpha}\cup\left\{  \cos\left(  2w\right)
,\sin\left(  2w\right)  ,\cos\left(  v+2w\right)  ,\sin\left(  v+2w\right)
:\left(  u,v,w\right)  \in\Omega^{\alpha}\right\} \\
&  \cup\left\{  \cos\left(  2u\right)  ,\sin\left(  2u\right)  ,\cos\left(
2u+w\right)  ,\sin\left(  2u+w\right)  :\left(  u,v,w\right)  \in
\Omega^{\alpha}\right\} \\
&  \cup\left\{  \cos\left(  2v\right)  ,\sin\left(  2v\right)  ,\cos\left(
u+2v\right)  ,\sin\left(  u+2v\right)  :\left(  u,v,w\right)  \in
\Omega^{\alpha}\right\}  .
\end{align*}

Continuing this recursive method, in a similar way as above, one obtains the
following result.

\begin{proposition}
[\textbf{Recursive construction of the vector space $\mathcal{V}_{n}^{\alpha}$}]
\label{iterative_construction_of_trivariate_extension}The basis
$V_{n}^{\alpha}$ of the constrained trivariate extension $\mathcal{V}%
_{n}^{\alpha}=%
\operatorname{span}%
V_{n}^{\alpha}$ of the vector space (\ref{truncated_Fourier_vector_space}) of
order $n\geq0$ fulfills the recurrence property%
\[
V_{n}^{\alpha}=V_{n-1}^{\alpha}\cup\left\{  \cos\left(  ru+gv+bw\right)
,\sin\left(  ru+gv+bw\right)  :\left(  u,v,w\right)  \in\Omega^{\alpha
},~\left[  ru+gv+bw\right]  \in E_{n}^{\alpha}/_{\sim}\setminus\left(
\cup_{i=1}^{n-1}E_{i}^{\alpha}/_{\sim}\right)  \right\}
\]
for all $n\geq1$, where from each equivalence class we choose a single
representant,%
\[
E_{n}^{\alpha}/_{\sim}\setminus\left(  \cup_{i=1}^{n-1}E_{i}^{\alpha}/_{\sim
}\right)  =\left\{  \left[  0\cdot u+g\cdot v+n\cdot w\right]  \right\}
_{g=0}^{n-1}\cup\left\{  \left[  n\cdot u+0\cdot v+b\cdot w\right]  \right\}
_{b=0}^{n-1}\cup\left\{  \left[  r\cdot u+n\cdot v+0\cdot w\right]  \right\}
_{r=0}^{n-1},
\]
and $V_{0}^{\alpha}=\left\{  1:\left(  u,v,w\right)  \in\Omega^{\alpha
}\right\}  $, i.e., if $n\geq1$ the function system $V_{n}^{\alpha}$ can be
obtained by completing $V_{n-1}^{\alpha}$ with $6n$ new linearly independent
constrained trivariate functions defined over the domain $\Omega^{\alpha}$.
\end{proposition}

\begin{corollary}
[\textbf{Dimension of the vector space $\mathcal{V}_{n}^{\alpha}$}]\label{dimension}%
Based on Proposition \ref{iterative_construction_of_trivariate_extension}, one
obtains the recurrence relation%
\[
\dim\mathcal{V}_{n}^{\alpha}=\dim\mathcal{V}_{n-1}^{\alpha}+6n,~n\geq1
\]
with initial condition $\dim\mathcal{V}_{0}^{\alpha}=1$. Thus,
\[
\dim\mathcal{V}_{n}^{\alpha}=1+6\sum_{i=1}^{n}i=3n\left(  n+1\right)
+1=\delta_{n} = \dim\mathcal{T}_{2n}^{\alpha}.
\]

\end{corollary}

\begin{theorem}
[\textbf{Coincidence of vector spaces $\mathcal{T}_{2n}^{\alpha}$ and $\mathcal{V}%
_{n}^{\alpha}$}]The system $T_{2n}^{\alpha}$ of constrained trivariate
functions forms a basis of the vector space $\mathcal{V}_{n}^{\alpha}$, i.e.,%
\[
\mathcal{T}_{2n}^{\alpha}=%
\operatorname{span}%
T_{2n}^{\alpha}=%
\operatorname{span}%
V_{n}^{\alpha}=\mathcal{V}_{n}^{\alpha}\text{.}%
\]

\end{theorem}

\begin{proof}
First of all, we prove that functions of the system (\ref{united_system}) are
elements of $%
\operatorname{span}%
V_{n}^{\alpha}$. Due to symmetry properties of functions (\ref{R_system}),
(\ref{G_system}) and (\ref{B_system}), it is sufficient to show that for all
indices $j=0,1,\ldots,n$ and $i=j,j+1,\ldots,n$ the function $R_{2n,2n-i,j}%
^{\alpha}\left(  u,v,w\right)  =\sin^{2n-i}\frac{u}{2}\sin^{i}\frac{w}{2}%
\cos^{i-j}\frac{v}{2}\sin^{j}\frac{v}{2}\ $belongs to $%
\operatorname{span}%
V_{n}^{\alpha}$. By means of trigonometric identities%
\begin{align}
\cos^{m}\frac{t}{2} &  =\left\{
\begin{array}
[c]{rl}%
\dfrac{1}{2^{m-1}}%
{\displaystyle\sum\limits_{k=0}^{\frac{m-1}{2}}}
\dbinom{m}{k}\cos\left(  \dfrac{1}{2}\left(  m-2k\right)  t\right)  , &
m\equiv1\left(  \operatorname{mod}2\right)  ,\\
& \\
\dfrac{1}{2^{m}}\dbinom{m}{\frac{m}{2}}+\dfrac{1}{2^{m-1}}%
{\displaystyle\sum\limits_{k=0}^{\frac{m}{2}-1}}
\dbinom{m}{k}\cos\left(  \dfrac{1}{2}\left(  m-2k\right)  t\right)  , &
m\equiv0\left(  \operatorname{mod}2\right)  ,
\end{array}
\right.  \\
& \nonumber\\
\sin^{m}\frac{t}{2} &  =\left\{
\begin{array}
[c]{rl}%
\dfrac{1}{2^{m-1}}%
{\displaystyle\sum\limits_{k=0}^{\frac{m-1}{2}}}
\left(  -1\right)  ^{\frac{m-1}{2}-k}\dbinom{m}{k}\sin\left(  \dfrac{1}%
{2}\left(  m-2k\right)  t\right)  , & m\equiv1\left(  \operatorname{mod}%
2\right)  ,\\
& \\
\dfrac{1}{2^{m}}\dbinom{m}{\frac{m}{2}}+\dfrac{1}{2^{m-1}}%
{\displaystyle\sum\limits_{k=0}^{\frac{m}{2}-1}}
\left(  -1\right)  ^{\frac{m}{2}-k}\dbinom{m}{k}\cos\left(  \dfrac{1}%
{2}\left(  m-2k\right)  t\right)  , & m\equiv0\left(  \operatorname{mod}%
2\right)  ,
\end{array}
\right.
\end{align}
($m\in%
\mathbb{N}
$, $t\in%
\mathbb{R}
$) and by the parity of powers $2n-i$, $i$, $i-j$ and $j$, one can easily see
that function $R_{2n,2n-i,j}^{\alpha}\left(  u,v,w\right)  $ can be written as
the sum of products of the type%
\begin{equation}%
\left\{
\begin{array}
[c]{rl}%
\cos\left(  \frac{2n-i-2k_{1}}{2}u\right)  \cos\left(  \frac{i-2k_{2}}%
{2}w\right)  \cos\left(  \frac{i-j-2k_{3}}{2}v\right)  \cos\left(
\frac{j-2k_{4}}{2}v\right)  , & i\equiv0\left(  \operatorname{mod}2\right)
\text{ and }j\equiv0\left(  \operatorname{mod}2\right)  ,\\
& \\
\sin\left(  \frac{2n-i-2k_{1}}{2}u\right)  \sin\left(  \frac{i-2k_{2}}%
{2}w\right)  \cos\left(  \frac{i-j-2k_{3}}{2}v\right)  \sin\left(
\frac{j-2k_{4}}{2}v\right)  , & i\equiv1\left(  \operatorname{mod}2\right)
\text{ and }j\equiv1\left(  \operatorname{mod}2\right)  ,\\
& \\
\cos\left(  \frac{2n-i-2k_{1}}{2}u\right)  \cos\left(  \frac{i-2k_{2}}%
{2}w\right)  \cos\left(  \frac{i-j-2k_{3}}{2}v\right)  \sin\left(
\frac{j-2k_{4}}{2}v\right)  , & i\equiv0\left(  \operatorname{mod}2\right)
\text{ and }j\equiv1\left(  \operatorname{mod}2\right)  ,\\
& \\
\sin\left(  \frac{2n-i-2k_{1}}{2}u\right)  \sin\left(  \frac{i-2k_{2}}%
{2}w\right)  \cos\left(  \frac{i-j-2k_{3}}{2}v\right)  \cos\left(
\frac{j-2k_{4}}{2}v\right)  , & i\equiv1\left(  \operatorname{mod}2\right)
\text{ and }j\equiv0\left(  \operatorname{mod}2\right)  ,
\end{array}
\right.
\label{possible_products_inclusion}%
\end{equation}
where $\left(  k_{1},k_{2},k_{3},k_{4}\right)  \in K_{i\left(
\operatorname{mod}2\right)  ,j\left(  \operatorname{mod}2\right)  }$ such that%
\begin{align*}
K_{0,0} &  =\left\{  0,1,\ldots,\frac{2n-i}{2}\right\}  \times\left\{
0,1,\ldots,\frac{i}{2}\right\}  \times\left\{  0,1,\ldots,\frac{i-j}%
{2}\right\}  \times\left\{  0,1,\ldots,\frac{j}{2}\right\}  ,\\
K_{1,1} &  =\left\{  0,1,\ldots,\frac{2n-i-1}{2}\right\}  \times\left\{
0,1,\ldots,\frac{i-1}{2}\right\}  \times\left\{  0,1,\ldots,\frac{i-j}%
{2}\right\}  \times\left\{  0,1,\ldots,\frac{j-1}{2}\right\}  ,\\
K_{0,1} &  =\left\{  0,1,\ldots,\frac{2n-i}{2}\right\}  \times\left\{
0,1,\ldots,\frac{i}{2}\right\}  \times\left\{  0,1,\ldots,\frac{i-j-1}%
{2}\right\}  \times\left\{  0,1,\ldots,\frac{j-1}{2}\right\}  ,\\
K_{1,0} &  =\left\{  0,1,\ldots,\frac{2n-i-1}{2}\right\}  \times\left\{
0,1,\ldots,\frac{i-1}{2}\right\}  \times\left\{  0,1,\ldots,\frac{i-j}%
{2}\right\}  \times\left\{  0,1,\ldots,\frac{j}{2}\right\}  .
\end{align*}
Observe that mixed constrained products of the type
(\ref{possible_products_inclusion}) can be expressed by means of the
equivalence relation $\sim$ as linear combinations of functions either of the
type $\cos\left(  r_{1}u+g_{1}v+b_{1}w\right)  $ or $\sin\left(  r_{2}%
u+g_{2}v+b_{2}w\right)  $, where $\left[  r_{1},g_{1},b_{1}\right]  ,\left[
r_{2},g_{2},b_{2}\right]  \in\cup_{i=1}^{n}E_{i}^{\alpha}/_{\sim}$, i.e.,
\[
\cos\left(  r_{1}u+g_{1}v+b_{1}w\right)  ,\sin\left(  r_{2}u+g_{2}%
v+b_{2}w\right)  \in V_{n}^{\alpha}.
\]
For instance, in case of $i\equiv1\left(  \operatorname{mod}2\right)  $ and
$j\equiv1\left(  \operatorname{mod}2\right)  $, by means of trigonometric
identities%
\begin{align}
\cos x\cos y &  =\frac{1}{2}\left(  \cos\left(  x-y\right)  +\cos\left(
x+y\right)  \right)  ,\\
\sin x\sin y &  =\frac{1}{2}\left(  \cos\left(  x-y\right)  -\cos\left(
x+y\right)  \right)  ,\\
\sin x\cos y &  =\frac{1}{2}\left(  \sin\left(  x+y\right)  +\sin\left(
x-y\right)  \right)
\end{align}
($x,y\in%
\mathbb{R}
$), each of the eight terms of the expansion of the product%
\[
\sin\left(  \dfrac{2n-i-2k_{1}}{2}u\right)  \sin\left(  \dfrac{i-2k_{2}}%
{2}w\right)  \cos\left(  \dfrac{i-j-2k_{3}}{2}v\right)  \sin\left(
\dfrac{j-2k_{4}}{2}v\right)
\]
has the form%
\[
\sin\left(  s_{1}\frac{2n-i-2k_{1}}{2}u+\left(  s_{2}\frac{i-j-2k_{3}}%
{2}+s_{3}\frac{j-2k_{4}}{2}\right)  v+s_{4}\frac{i-2k_{2}}{2}w\right)  ,
\]
where $s_{1}=1$ and $s_{2},s_{3},s_{4}\in\left\{  -1,+1\right\}  $, i.e.,
$s_{r}^{2}=1$, $r=1,2,3,4$. Using the constraint $u+v+w=\alpha$ one can
successively write that%
\begin{align*}
&  \sin\left(  s_{1}\frac{2n-i-2k_{1}}{2}u+\left(  s_{2}\frac{i-j-2k_{3}}%
{2}+s_{3}\frac{j-2k_{4}}{2}\right)  v+s_{4}\frac{i-2k_{2}}{2}w\right)  \\
&  =\sin\left(  s_{1}\frac{2n-i-2k_{1}}{2}u+\left(  s_{2}\frac{i-j-2k_{3}}%
{2}+s_{3}\frac{j-2k_{4}}{2}\right)  v+s_{4}\frac{i-2k_{2}}{2}w-\frac{\alpha
}{2}+\frac{\alpha}{2}\right)  \\
&  =\cos\frac{\alpha}{2}\sin\left(  s_{1}\frac{2n-i-2k_{1}}{2}u+\left(
s_{2}\frac{i-j-2k_{3}}{2}+s_{3}\frac{j-2k_{4}}{2}\right)  v+s_{4}%
\frac{i-2k_{2}}{2}w-\frac{\alpha}{2}\right)  \\
&  +\sin\frac{\alpha}{2}\cos\left(  s_{1}\frac{2n-i-2k_{1}}{2}u+\left(
s_{2}\frac{i-j-2k_{3}}{2}+s_{3}\frac{j-2k_{4}}{2}\right)  v+s_{4}%
\frac{i-2k_{2}}{2}w-\frac{\alpha}{2}\right)  \\
&  =\cos\frac{\alpha}{2}\sin\left(  s_{1}\frac{2n-i-1-2k_{1}}{2}u+\left(
s_{2}\frac{i-j-2k_{3}}{2}+s_{3}\frac{j-2k_{4}-s_{3}}{2}\right)  v+s_{4}%
\frac{i-2k_{2}-s_{4}}{2}w\right)  \\
&  +\sin\frac{\alpha}{2}\cos\left(  s_{1}\frac{2n-i-1-2k_{1}}{2}u+\left(
s_{2}\frac{i-j-2k_{3}}{2}+s_{3}\frac{j-2k_{4}-s_{3}}{2}\right)  v+s_{4}%
\frac{i-2k_{2}-s_{4}}{2}w\right)  ,
\end{align*}
where for all $4$-tuples $\left(  k_{1},k_{2},k_{3},k_{4}\right)  \in K_{1,1}$
the coefficients of variables $u$, $v$ and $w$ are integer numbers and both of
the sine and cosine functions belong to the equivalence class%
\[
\left[  \underset{\in%
\mathbb{N}
}{\underbrace{s_{1}\frac{2n-i-1-2k_{1}}{2}-z}},\underset{\in%
\mathbb{N}
}{\underbrace{s_{2}\frac{i-j-2k_{3}}{2}+s_{3}\frac{j-2k_{4}-s_{3}}{2}-z}%
},\underset{\in%
\mathbb{N}
}{\underbrace{s_{4}\frac{i-2k_{2}-s_{4}}{2}-z}}\right]  ,
\]
where%
\[
z=\min\left\{  s_{1}\frac{2n-i-1-2k_{1}}{2},s_{2}\frac{i-j-2k_{3}}{2}%
+s_{3}\frac{j-2k_{4}-s_{3}}{2},s_{4}\frac{i-2k_{2}-s_{4}}{2}\right\}  .
\]
Parameter $z$ can take three possible values. For example, if%
\[
z=s_{2}\frac{i-j-2k_{3}}{2}+s_{3}\frac{j-2k_{4}-s_{3}}{2}%
\]
one has that%
\begin{align*}
&  \max\left\{  s_{1}\frac{2n-i-1-2k_{1}}{2}-s_{2}\frac{i-j-2k_{3}}{2}%
-s_{3}\frac{j-2k_{4}-s_{3}}{2}:\left(  k_{1},k_{2},k_{3},k_{4}\right)  \in
K_{1,1},~s_{1}=1,s_{2},s_{3}\in\left\{  -1,+1\right\}  \right\}  \\
= &  \frac{2n-i-1}{2}+\frac{i-j}{2}+\frac{j+1}{2}\\
= &  n
\end{align*}
and%
\begin{align*}
&  \max\left\{  s_{4}\frac{i-2k_{2}-s_{4}}{2}-s_{2}\frac{i-j-2k_{3}}{2}%
-s_{3}\frac{j-2k_{4}-s_{3}}{2}:\left(  k_{1},k_{2},k_{3},k_{4}\right)  \in
K_{1,1},~s_{1}=1,s_{2},s_{3}\in\left\{  -1,+1\right\}  \right\}  \\
&  =\frac{i-1}{2}+\frac{i-j}{2}+\frac{j+1}{2}\\
&  =i\leq n.
\end{align*}
The remaining two cases of $z$ can be treated analogously.

Performing similar calculations as above, one concludes that all types of
products from (\ref{possible_products_inclusion}) can be expanded into linear
combinations of sine and cosine functions that belong to some equivalence
classes $\left[  r,g,b\right]  \in$ $\cup_{i=1}^{n}E_{i}^{\alpha}/_{\sim}$.
Consequently,
\[%
\operatorname{span}%
T_{2n}^{\alpha}\subseteq%
\operatorname{span}%
V_{n}^{\alpha}\text{.}%
\]

Finally, arguing by contradiction, we assume that there exists a constrained
trivariate function $\sigma:\Omega^{\alpha}\rightarrow%
\mathbb{R}
$ such that $\sigma\in%
\operatorname{span}%
V_{n}^{\alpha}\setminus%
\operatorname{span}%
T_{2n}^{\alpha}$. Then, using Theorem \ref{joint_systems_linear_independence}
and Corollary \ref{dimension}, we obtain the contradiction
\[
1+\delta_{n}=\dim%
\operatorname{span}%
\left\{  \sigma,%
\operatorname{span}%
T_{2n}^{\alpha}\right\}  \leq\dim%
\operatorname{span}%
V_{n}^{\alpha}=\delta_{n},
\]
therefore $\mathcal{T}_{2n}^{\alpha}=%
\operatorname{span}%
T_{2n}^{\alpha}=%
\operatorname{span}%
V_{n}^{\alpha}=\mathcal{V}_{n}^{\alpha}$.\qedhere
\end{proof}

The next section provides a method to normalize the constrained trivariate
basis $T_{2n}^{\alpha}$.

\section{\label{sec:partition_of_unity}Partition of unity}

Since the function system (\ref{united_system}) is a basis of the constrained trivariate
extension of the vector space (\ref{truncated_Fourier_vector_space}) that also
contains the constant function $1$, it follows that there exist unique
coefficients%
\begin{equation}
r_{2n,n,n}^{\alpha},\left\{  r_{2n,2n-i,j}^{\alpha}\right\}  _{j=0,i=j}%
^{n-1,2n-1-j},\left\{  g_{2n,2n-i,j}^{\alpha}\right\}  _{j=0,i=j}%
^{n-1,2n-1-j},\left\{  b_{2n,2n-i,j}^{\alpha}\right\}  _{j=0,i=j}^{n-1,2n-1-j}
\label{normalizing_weights}%
\end{equation}
such that%
\begin{align*}
1  &  =L\left(  u,v,w\right) \\
&  =r_{2n,n,n}^{\alpha}R_{2n,n,n}^{\alpha}\left(  u,v,w\right)  +\sum
_{j=0}^{n-1}\sum_{i=j}^{2n-1-j}r_{2n,2n-i,j}^{\alpha}R_{2n,2n-i,j}^{\alpha
}\left(  u,v,w\right) \\
&  +\sum_{j=0}^{n-1}\sum_{i=j}^{2n-1-j}g_{2n,2n-i,j}^{\alpha}G_{2n,2n-i,j}%
^{\alpha}\left(  u,v,w\right)  +\sum_{j=0}^{n-1}\sum_{i=j}^{2n-1-j}%
b_{2n,2n-i,j}^{\alpha}B_{2n,2n-i,j}^{\alpha}\left(  u,v,w\right)
\end{align*}
for all $\left(  u,v,w\right)  \in\Omega^{\alpha}$. Due to the symmetry
properties of the joint function system (\ref{united_system}), coefficients
(\ref{normalizing_weights}) have to fulfill the symmetry conditions%
\begin{align}
r_{2n,2n-i,j}^{\alpha}  &  =r_{2n,i,j}^{\alpha},~j=0,1,\ldots
,n-1,~i=j+1,j+2,\ldots,n-1,\label{symmetry_condition_1}\\
g_{2n,2n-i,j}^{\alpha}  &  =r_{2n,2n-i,j}^{\alpha},~j=0,1,\ldots
,n-1,~i=j,j+1,\ldots,2n-1-j,\label{symmetry_condition_2}\\
b_{2n,2n-i,j}^{\alpha}  &  =r_{2n,2n-i,j}^{\alpha},~j=0,1,\ldots
,n-1,~i=j,j+1,,\ldots,2n-1-j, \label{symmetry_condition_3}%
\end{align}
where constants $\left\{  r_{2n,2n-i,j}^{\alpha}\right\}  _{j=0,i=j}^{n,n}$
are unknown parameters at the moment.

For the present we are not able to give a closed form of normalizing
coefficients (\ref{normalizing_weights}) for arbitrary order $n\geq1$,
however we propose an efficient technique using which one can reduce their
determination to the solution of several lower triangular linear systems.
Normalizing constants (\ref{normalizing_weights}) can be determined by solving
the system of equations%
\begin{equation}
\left.  \frac{\partial^{r}}{\partial y^{r}}L\left(  u\left(  x,y\right)
,v\left(  x,y\right)  ,w\left(  x,y\right)  \right)  \right\vert _{y=0}%
=\delta_{r,0},~\forall x\in\left[  0,\alpha\right]  ,~\forall y\in\left[
0,\alpha-x\right]  ,~\forall r=0,1,\ldots,n\label{system_partition_of_unity}%
\end{equation}
under parametrization (\ref{parametrization_for_partition_of_unity}), where
$\delta_{r,0}$ denotes the Kronecker delta. (Naturally, one may solve the
system (\ref{system_partition_of_unity}) by using a different parametrization,
however symbolic calculations proved to be much shorter and easier under
parametrization (\ref{parametrization_for_partition_of_unity}).)

\begin{remark}
[\textbf{Normalizing coefficients of level $0$}]\label{normalizing_coeff_of_level_0}Due
to the boundary property (\ref{boundary_property_1}) it follows that%
\[
r_{2n,2n-i,0}^{\alpha}=c_{2n,2n-i}^{\alpha}=c_{2n,i}^{\alpha},~\forall
i=0,1,\ldots,n.
\]

\end{remark}

For each order $r=1,2,\ldots,n$, after evaluating the $r$th order mixed
partial derivatives with respect to $v\left(  x,y\right)  =y$ and $w\left(
x,y\right)  =x-y$ at $y=0$ and dividing both sides of
(\ref{system_partition_of_unity}) by $\cos^{2n}\frac{x}{2}\neq0$ ($x\in\left[
0,\alpha\right]  $) , equality%
\[
\sum_{k=0}^{r}\dbinom{r}{k}\left(  -1\right)  ^{k}\left.  \frac{\partial^{r}%
}{\partial v^{r-k}\partial w^{k}}L\left(  u\left(  x,y\right)  ,v\left(
x,y\right)  ,w\left(  x,y\right)  \right)  \right\vert _{y=0}=0,~\forall
x\in\left[  0,\alpha\right]  ,~y\in\left[  0,\alpha-x\right]
\]
can be rewritten into a polynomial expression of $\tan\frac{x}{2}$, since the
sum of powers of trigonometric functions $\sin\frac{\alpha-x}{2}$, $\cos
\frac{\alpha-x}{2}$, $\sin\frac{x}{2}$ and $\cos\frac{x}{2}$ appearing as
potential factors in the terms of the left hand side always equals $2n$.

As an example consider the first order partial derivative of the function
system (\ref{united_system}) with respect to $y$, at $y=0$ under
parametrization (\ref{parametrization_for_partition_of_unity}). Then, we have
the next proposition the rather lengthy and technical proof of which can be found in \cite{tech_report}.

\begin{proposition}
[\textbf{Normalizing coefficients of level $1$}]\label{prop:level_1_coefficients}For
arbitrary order $n\geq2$, normalizing constants $\left\{  r_{2n,2n-i,1}%
^{\alpha}\right\}  _{i=1}^{n}$ are given by%
\begin{equation}
r_{2n,2n-i,1}^{\alpha}=\left\{
\begin{array}
[c]{ll}%
\dfrac{2}{\sin\frac{\alpha}{2}}c_{2n,2n-2}^{\alpha}, & i=1,\\
& \\
\dfrac{i}{\sin\frac{\alpha}{2}}c_{2n,2n-i}^{\alpha}\cos\frac{\alpha}{2}%
+\dfrac{i+1}{\sin\frac{\alpha}{2}}c_{2n,2n-\left(  i+1\right)  }^{\alpha}, &
i=2,3,\ldots,n.
\end{array}
\right.  \label{level_1_coefficients}%
\end{equation}

\end{proposition}

Continuing the technique presented in the proof of Proposition
\ref{prop:level_1_coefficients} (cf. \cite{tech_report}) with the evaluation of the $r$th order ($2\leq
r\leq n$) partial derivatives appearing in the system
(\ref{system_partition_of_unity}), one is able to calculate the closed form of
normalizing constants $\left\{  r_{2n,2n-i,j}^{\alpha}\right\}  _{j=0,i=j}%
^{n,n}$. Examples \ref{normalizing_coefficients_n_1}--\ref{normalizing_coefficients_n_3}
provide closed formulas of these normalizing coefficients for $n=1,~2$
and $3$, respectively.

\begin{example}
[\textbf{Normalizing constants for $n=1$}]\label{normalizing_coefficients_n_1}In case
of $n=1$, the unique solution of the system (\ref{system_partition_of_unity})
is%
\begin{align*}
r_{2,2,0}^{\alpha}  &  =c_{2,2}^{\alpha}=\frac{1}{\sin^{2}\frac{\alpha}{2}},\\
r_{2,1,0}^{\alpha}  &  =c_{2,1}^{\alpha}=\frac{2\cos\frac{\alpha}{2}}{\sin
^{2}\frac{\alpha}{2}},\\
r_{2,1,1}^{\alpha}  &  =\frac{2}{\sin\frac{\alpha}{2}}c_{2,2}^{\alpha}%
-\frac{1}{\sin\frac{\alpha}{2}}c_{2,1}^{\alpha}\cos\frac{\alpha}{2}%
=\frac{2\sin\frac{\alpha}{2}}{\sin^{2}\frac{\alpha}{2}}.
\end{align*}
\end{example}

\begin{example}
[\textbf{Normalizing constants for $n=2$}]\label{normalizing_coefficients_n_2}For
$n=2$, the system (\ref{system_partition_of_unity}) admits the unique solution%
\begin{align*}
r_{4,4,0}^{\alpha}  &  =c_{4,4}^{\alpha}=\frac{1}{\sin^{4}\frac{\alpha}{2}},\\
r_{4,3,0}^{\alpha}  &  =c_{4,3}^{\alpha}=\frac{4\cos\frac{\alpha}{2}}{\sin
^{4}\frac{\alpha}{2}},\\
r_{4,2,0}^{\alpha}  &  =c_{4,2}^{\alpha}=\frac{2+4\cos^{2}\frac{\alpha}{2}%
}{\sin^{4}\frac{\alpha}{2}},\\
r_{4,3,1}^{\alpha}  &  =\dfrac{2}{\sin\frac{\alpha}{2}}c_{4,2}^{\alpha}%
=\frac{4+8\cos^{2}\frac{\alpha}{2}}{\sin^{5}\frac{\alpha}{2}},\\
r_{4,2,1}^{\alpha}  &  =\dfrac{2}{\sin\frac{\alpha}{2}}c_{4,2}^{\alpha}%
\cos\frac{\alpha}{2}+\dfrac{3}{\sin\frac{\alpha}{2}}c_{4,3}^{\alpha}%
=\frac{16\cos\frac{\alpha}{2}+8\cos^{3}\frac{\alpha}{2}}{\sin^{5}\frac{\alpha
}{2}},\\
r_{4,2,2}^{\alpha}  &  =-\frac{6}{\sin^{2}\frac{\alpha}{2}}c_{4,4}^{\alpha
}-\frac{3\cos\frac{\alpha}{2}}{\sin^{2}\frac{\alpha}{2}}c_{4,3}^{\alpha
}+\left(  \frac{6}{\sin^{2}\frac{\alpha}{2}}-\frac{3\cos^{2}\frac{\alpha}{2}%
}{\sin^{2}\frac{\alpha}{2}}+2\right)  c_{4,2}^{\alpha}=\frac{10+20\cos
^{2}\frac{\alpha}{2}}{\sin^{4}\frac{\alpha}{2}}.
\end{align*}

\end{example}

\begin{remark}
Normalizing coefficients provided in Examples
\ref{normalizing_coefficients_n_1} and \ref{normalizing_coefficients_n_2} also
appeared (without any explanations) in articles \cite{Wang2010a} and
\cite{Wang2010b}, respectively.
\end{remark}

\begin{example}
[\textbf{Normalizing constants for $n=3$}]\label{normalizing_coefficients_n_3}If $n=3$,
the system (\ref{system_partition_of_unity}) has the unique solution%
\begin{align*}
r_{6,6,0}^{\alpha}  &  =c_{6,6}^{\alpha}=\frac{1}{\sin^{6}\frac{\alpha}{2}},\\
r_{6,5,0}^{\alpha}  &  =c_{6,5}^{\alpha}=\frac{6\cos\frac{\alpha}{2}}{\sin
^{6}\frac{\alpha}{2}},\\
r_{6,4,0}^{\alpha}  &  =c_{6,4}^{\alpha}=\frac{12\cos^{2}\frac{\alpha}{2}%
+3}{\sin^{6}\frac{\alpha}{2}},\\
r_{6,3,0}^{\alpha}  &  =c_{6,3}^{\alpha}=\frac{8\cos^{3}\frac{\alpha}%
{2}+12\cos\frac{\alpha}{2}}{\sin^{6}\frac{\alpha}{2}},\\
r_{6,5,1}^{\alpha}  &  =\frac{2}{\sin\frac{\alpha}{2}}c_{6,4}^{\alpha}%
=\frac{24\cos^{2}\frac{\alpha}{2}+6}{\sin^{7}\frac{\alpha}{2}},\\
r_{6,4,1}^{\alpha}  &  =\frac{3}{\sin\frac{\alpha}{2}}c_{6,3}^{\alpha}%
+\frac{2\cos\frac{\alpha}{2}}{\sin\frac{\alpha}{2}}c_{6,4}^{\alpha}%
=\frac{48\cos^{3}\frac{\alpha}{2}+42\cos\frac{\alpha}{2}}{\sin^{7}\frac
{\alpha}{2}},\\
r_{6,3,1}^{\alpha}  &  =\frac{4}{\sin\frac{\alpha}{2}}c_{6,4}^{\alpha}%
+\frac{3\cos\frac{\alpha}{2}}{\sin\frac{\alpha}{2}}c_{6,3}^{\alpha}%
=\frac{24\cos^{4}\frac{\alpha}{2}+84\cos^{2}\frac{\alpha}{2}+12}{\sin^{7}%
\frac{\alpha}{2}},\\
r_{6,4,2}^{\alpha}  &  =c_{6,4}\left(  2+\dfrac{6}{\sin^{2}\frac{\alpha}{2}%
}\right)  +c_{6,3}^{\alpha}\frac{6}{\sin^{2}\frac{\alpha}{2}}\cos\frac{\alpha
}{2}=\frac{24\cos^{4}\frac{\alpha}{2}+162\cos^{2}\frac{\alpha}{2}+24}{\sin
^{8}\frac{\alpha}{2}},\\
r_{6,3,2}^{\alpha}  &  =-10\frac{1}{\sin^{2}\frac{\alpha}{2}}c_{6,5}^{\alpha
}+\left(  \frac{8\cos^{3}\frac{\alpha}{2}}{\sin^{2}\frac{\alpha}{2}}%
+8\cos\frac{\alpha}{2}\right)  c_{6,4}^{\alpha}+\left(  \frac{6\cos^{2}%
\frac{\alpha}{2}}{\sin^{2}\frac{\alpha}{2}}+\frac{12}{\sin^{2}\frac{\alpha}%
{2}}+3\right)  c_{6,3}^{\alpha}\\
&  =\frac{24\cos^{5}\frac{\alpha}{2}+252\cos^{3}\frac{\alpha}{2}+144\cos
\frac{\alpha}{2}}{\sin^{8}\frac{\alpha}{2}},\\
r_{6,3,3}^{\alpha}  &  =\frac{20}{\sin^{3}\frac{\alpha}{2}}c_{6,6}^{\alpha
}+\frac{50}{\sin^{3}\frac{\alpha}{2}}\cos\frac{\alpha}{2}c_{6,5}^{\alpha
}-\left(  \frac{56}{\sin^{3}\frac{\alpha}{2}}\cos^{4}\frac{\alpha}{2}%
-\frac{60}{\sin^{3}\frac{\alpha}{2}}\cos^{2}\frac{\alpha}{2}-\frac{8}%
{\sin\frac{\alpha}{2}}+\frac{56}{\sin\frac{\alpha}{2}}\cos^{2}\frac{\alpha}%
{2}-\frac{4}{\sin^{3}\frac{\alpha}{2}}\right)  c_{6,4}^{\alpha}\\
&  -\left(  \frac{10}{\sin^{3}\frac{\alpha}{2}}\cos^{3}\frac{\alpha}{2}%
-\frac{3}{\sin\frac{\alpha}{2}}\cos\frac{\alpha}{2}+\frac{12}{\sin^{3}%
\frac{\alpha}{2}}\cos\frac{\alpha}{2}\right)  c_{6,3}^{\alpha}\\
&  =\frac{-104\cos^{6}\frac{\alpha}{2}-276\cos^{4}\frac{\alpha}{2}+324\cos
^{2}\frac{\alpha}{2}+56}{\sin^{9}\frac{\alpha}{2}}.
\end{align*}

\end{example}

\begin{remark}
If one intends to numerically calculate the normalizing coefficients, we suggest to exploit symmetry conditions (\ref{symmetry_condition_1}%
)--(\ref{symmetry_condition_3}). In this way, the size of square collocation
matrices -- that potentially appear in linear systems of equations to be
solved -- can be reduced from $\delta_{n}=3n\left(  n+1\right)  +1$ to
$\frac{1}{2}\left(  n+1\right)  \left(  n+2\right)  $ that can be further
reduced to $\frac{1}{2}n\left(  n+1\right)  $ by means of Remark
\ref{normalizing_coeff_of_level_0}.
\end{remark}

\begin{definition}
[\textbf{Constrained trivariate trigonometric blending system}]The
normalized basis functions of the system%
\begin{equation}%
\begin{array}
[c]{ccl}%
\overline{T}_{2n}^{\alpha} & = & \left\{  \overline{R}_{2n,2n-i,j}^{\alpha
}\left(  u,v,w\right)  ,\overline{G}_{2n,2n-i,j}^{\alpha}\left(  u,v,w\right)
,\overline{B}_{2n,2n-i,j}^{\alpha}\left(  u,v,w\right)  :\left(  u,v,w\right)
\in\Omega^{\alpha}\right\}  _{j=0,i=j}^{n-1,2n-1-j}\\
&  & \\
&  & \cup\left\{  \overline{R}_{2n,n,n}^{\alpha}\left(  u,v,w\right)
=\overline{G}_{2n,n,n}^{\alpha}\left(  u,v,w\right)  =\overline{B}%
_{2n,n,n}^{\alpha}\left(  u,v,w\right)  :\left(  u,v,w\right)  \in
\Omega^{\alpha}\right\}
\end{array}
\label{normalized_united_system}%
\end{equation}
are called constrained trivariate trigonometric blending functions of order
$n\geq1$, where%
\begin{align*}
\overline{R}_{2n,2n-i,j}^{\alpha}\left(  u,v,w\right)   &  =r_{2n,2n-i,j}%
^{\alpha}R_{2n,2n-i,j}^{\alpha}\left(  u,v,w\right)  ,\\
\overline{G}_{2n,2n-i,j}^{\alpha}\left(  u,v,w\right)   &  =g_{2n,2n-i,j}%
^{\alpha}G_{2n,2n-i,j}^{\alpha}\left(  u,v,w\right)  ,\\
\overline{B}_{2n,2n-i,j}^{\alpha}\left(  u,v,w\right)   &  =b_{2n,2n-i,j}%
^{\alpha}B_{2n,2n-i,j}^{\alpha}\left(  u,v,w\right)
\end{align*}
and
\[
\overline{R}_{2n,n,n}^{\alpha}\left(  u,v,w\right)  =\overline{G}%
_{2n,n,n}^{\alpha}\left(  u,v,w\right)  =\overline{B}_{2n,n,n}^{\alpha}\left(
u,v,w\right)  =r_{2n,n,n}^{\alpha}R_{2n,n,n}^{\alpha}\left(  u,v,w\right)  .
\]

\end{definition}

\begin{remark}
Based on Examples \ref{normalizing_coefficients_n_1}--\ref{normalizing_coefficients_n_3}, in addition to the property of partition of unity, the first, the second and the third order cases of the function system (\ref{normalized_united_system}) are also non-negative. The non-negativity, the symmetry and closed formulas of normalizing coefficients of the constrained trivariate function system (\ref{normalized_united_system}) of arbitrary order at the moment constitute one of our open problems that will be detailed in Section \ref{sec:open_problems}.
\end{remark}

Subsections \ref{sec:triangular_trigonometric_patches} and
\ref{sec:triangular_rational_trigonometric_patches} introduce the notions of
general triangular trigonometric and rational trigonometric patches of order
$n$, respectively.

\section{Applications}\label{sec:applications}

One can easily associate control nets with the oriented graph (that induced
the constrained trivariate trigonometric function systems $\overline{R}%
_{2n}^{\alpha}$, $\overline{G}_{2n}^{\alpha}$ and $\overline{B}_{2n}^{\alpha}%
$) as shown in Fig.\ \ref{fig:associated_control_net}.

\begin{figure}
[!htb]
\begin{center}
\includegraphics[scale = 1
]%
{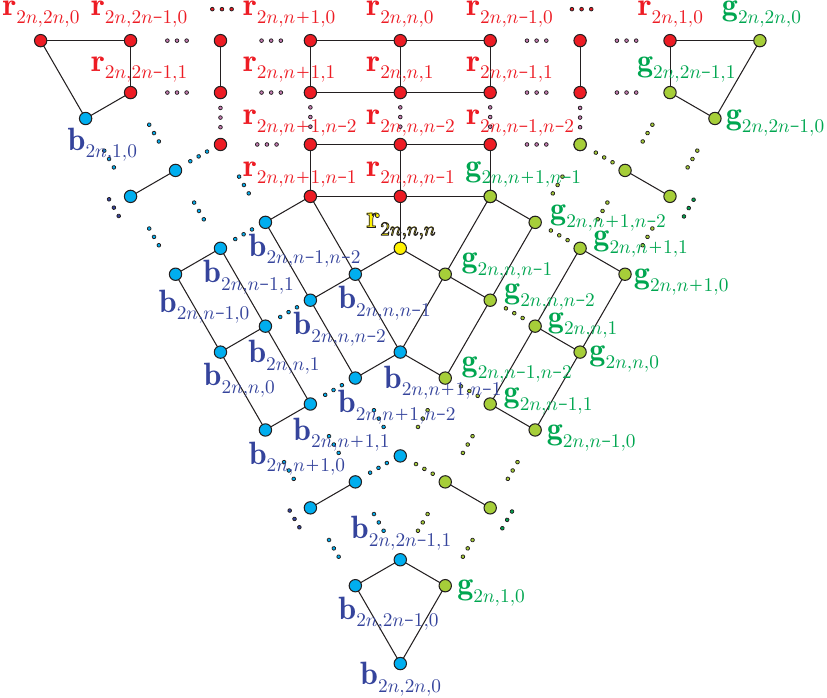}%
\caption{Layout of control points.}%
\label{fig:associated_control_net}%
\end{center}
\end{figure}

\begin{remark}
Naturally, one may triangulate the quadrangular "faces" of the control net shown in Fig.\ \ref{fig:associated_control_net}. However, the reason for not doing so is that our intention was to preserve the relations between constrained trivariate trigonometric functions induced by the oriented and multiplicatively weighted graph shown in Fig.\ \ref{fig:oriented_graph}, i.e., in this way the edges of the control net have a well-defined meaning. Probably, a deeper understanding of the graph may lead to additional edges needed for a possible triangulation.
\end{remark}

\subsection{\label{sec:triangular_trigonometric_patches}Triangular trigonometric patches}

By means of linear combinations of control points and blending functions
(\ref{normalized_united_system}) one can define a new surface modeling tool.

\begin{definition}
[\textbf{Triangular trigonometric patches}]\label{def:surface_patch}The constrained
trivariate vector function $\mathbf{s}_{n}^{\alpha}:\Omega^{\alpha}\rightarrow%
\mathbb{R}
^{3}$ of the form%
\begin{equation}%
\begin{array}
[c]{rcl}%
\mathbf{s}_{n}^{\alpha}\left(  u,v,w\right)  & = & \mathbf{r}_{2n,n,n}%
\overline{R}_{2n,n,n}^{\alpha}\left(  u,v,w\right)  +%
{\displaystyle\sum\limits_{j=0}^{n-1}}
{\displaystyle\sum\limits_{i=j}^{2n-1-j}}
\mathbf{r}_{2n,2n-i,j}\overline{R}_{2n,2n-i,j}^{\alpha}\left(  u,v,w\right) \\
&  & +%
{\displaystyle\sum\limits_{j=0}^{n-1}}
{\displaystyle\sum\limits_{i=j}^{2n-1-j}}
\mathbf{g}_{2n,2n-i,j}\overline{G}_{2n,2n-i,j}^{\alpha}\left(  u,v,w\right)  +%
{\displaystyle\sum\limits_{j=0}^{n-1}}
{\displaystyle\sum\limits_{i=j}^{2n-1-j}}
\mathbf{b}_{2n,2n-i,j}\overline{B}_{2n,2n-i,j}^{\alpha}\left(  u,v,w\right)
\end{array}
\label{triangular_trig_surface}%
\end{equation}
is called triangular trigonometric patch of order $n\geq1$, where vectors
\[
\left\{  \mathbf{r}_{2n,2n-i,j}\right\}  _{j=0,i=j}^{n,2n-1-j}\cup\left\{
\mathbf{g}_{2n,2n-i,j}\right\}  _{j=0,i=j}^{n-1,2n-1-j}\cup\left\{
\mathbf{b}_{2n,2n-i,j}\right\}  _{j=0,i=j}^{n-1,2n-1-j}\subset%
\mathbb{R}
^{3}%
\]
define its control net.
\end{definition}

Fig.\ \ref{fig:third_order_patch} illustrates a third order triangular trigonometric patch defined by $\delta_{3}=37$ control points.

\begin{figure}
[!htb]
\begin{center}
\includegraphics[
natheight=2.745800in,
natwidth=5.176800in,
height=2.7735in,
width=5.2036in
]%
{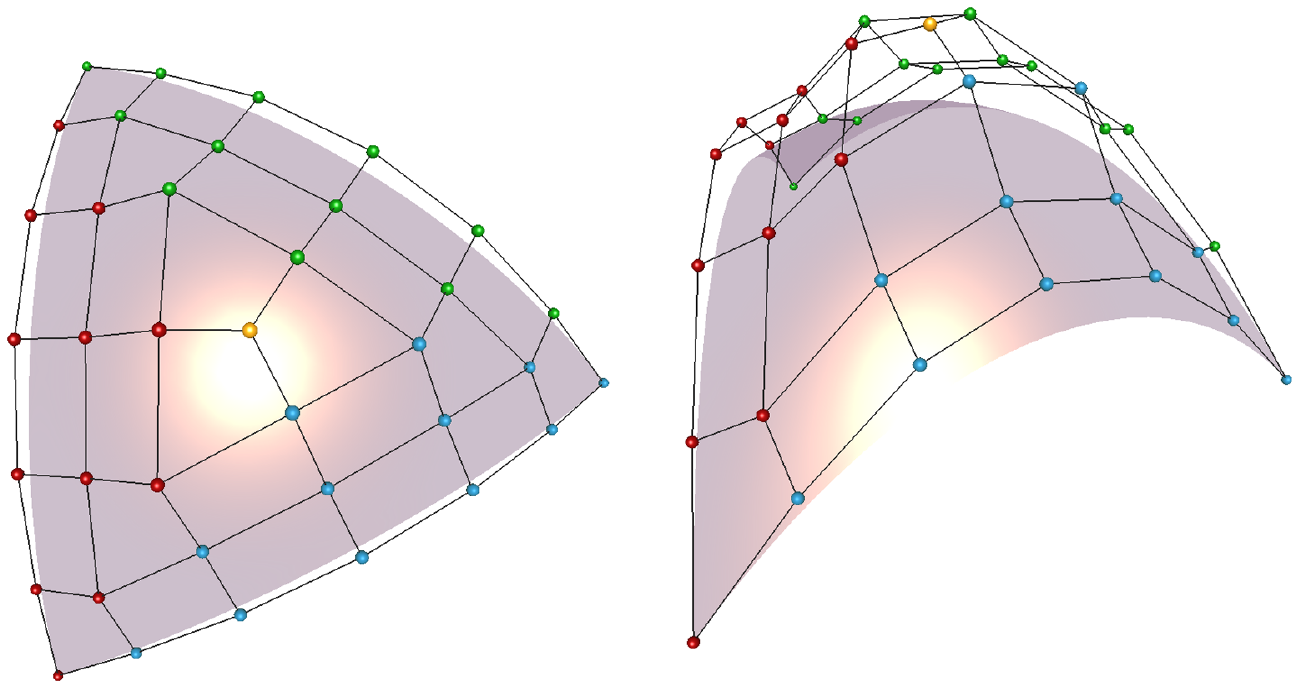}%
\caption{Different views of the same third order triangular trigonometric patch
along with its control net ($\alpha=\frac{\pi}{2}$).}%
\label{fig:third_order_patch}%
\end{center}
\end{figure}

Definition \ref{def:surface_patch} implies the following obvious geometric properties.

\begin{itemize}
\item
The triangular trigonometric patch (\ref{triangular_trig_surface}) has a global free shape parameter $\alpha \in \left(0,\pi\right).$

\item Since the basis (\ref{normalized_united_system}) consists of
functions that form a partition of the unity, any patch of the type
(\ref{triangular_trig_surface}) is invariant under the affine transformations
of its control points. Moreover, in the special cases of first, second and third order triangular trigonometric patches we have already proved that their shape lies in the convex hull of their control points, since in these cases the normalized function system (\ref{normalized_united_system}) is also non-negative. At the moment, the convex hull property of the patch (\ref{triangular_trig_surface}) of arbitrary order -- that is equivalent to the non-negativity of the normalizing coefficients of the basis (\ref{united_system}) -- is an open problem. However, considering the quadratic growth of the control points, it is very unlikely that in CAGD patches of order higher than three would be used for modeling. 

\item Due to boundary properties (\ref{boundary_property_1}%
)--(\ref{boundary_property_3}), the boundary curves of the patch
(\ref{triangular_trig_surface}) can be expressed in the univariate B-basis
(\ref{univariate_basis}), i.e.,%
\begin{align*}
\mathbf{s}_{n}^{\alpha}\left(  0,\alpha-w,w\right)   &  =\sum_{i=0}%
^{2n}\mathbf{g}_{2n,2n-i,0}A_{2n,2n-i}^{\alpha}\left(  w\right)  ,~w\in\left[
0,\alpha\right]  ,\\
\mathbf{s}_{n}^{\alpha}\left(  u,0,\alpha-u\right)   &  =\sum_{i=0}%
^{2n}\mathbf{r}_{2n,2n-i,0}A_{2n,2n-i}^{\alpha}\left(  u\right)  ,~u\in\left[
0,\alpha\right]  ,\\
\mathbf{s}_{n}^{\alpha}\left(  \alpha-v,v,0\right)   &  =\sum_{i=0}%
^{2n}\mathbf{b}_{2n,2n-i,0}A_{2n,2n-i}^{\alpha}\left(  v\right)  ,~v\in\left[
0,\alpha\right]  ,
\end{align*}
as a result of which (\ref{triangular_trig_surface}) interpolates the three
corners $\mathbf{r}_{2n,2n,0}$, $\mathbf{g}_{2n,2n,0}$ and $\mathbf{b}%
_{2n,2n,0}$ of its control net and its boundary curves can
exactly describe arcs of any trigonometric parametric curve the coordinate
functions of which are in the vector space (\ref{truncated_Fourier_vector_space}). The tangent planes at these corners are spanned by the terminal edges of the control polygons that generate the boundary curves of the given patch. Moreover, on account of Proposition \ref{Bernstein}, when $\alpha \to 0$ the boundary curves degenerate to B\'ezier curves of degree $2n$.
\item
It provides $\alpha$-dependent control point configurations for the exact description of
triangular patches of arbitrary trigonometric surfaces the coordinate
functions of which are given in the vector space
(\ref{constrained_trivariate_extension_of_truncated_Fourier_series}) (consider Example \ref{exmp:torus}).
\end{itemize}

\begin{example}
[\textbf{Control point based exact description of a toroidal triangle}]\label{exmp:torus}
Parametric equations%
\begin{equation}
\mathbf{t}_{\varrho,\mu}\left(  u,v\right)  =\left[
\begin{array}
[c]{ccc}%
\left(  \varrho+\mu\sin u\right)  \cos v & \left(  \varrho+\mu\sin u\right)
\sin v & \mu\cos u
\end{array}
\right]  ^{T},~u\in\left[  0,\alpha\right]  ,~v\in\left[  0,\alpha-u\right]
\label{toroidal_triangle}%
\end{equation}
define a triangular patch of a ring torus, where $\varrho>0$ is the distance
from the origin to the center of the meridian circle of radius $\mu\in\left(
0,\rho\right]$. The toroidal triangle
(\ref{toroidal_triangle}) can exactly be described by a second order triangular
trigonometric patch defined by control points%
\begin{align*}%
\color[rgb]{1.0, 0.0, 0.0}%
\mathbf{r}_{4,4,0}%
\normalcolor
&  =\left[
\begin{array}
[c]{ccc}%
\varrho+\mu\sin\alpha & 0 & \mu\cos\alpha
\end{array}
\right]  ^{T},\\%
\color[rgb]{1.0, 0.0, 0.0}%
\mathbf{r}_{4,3,0}%
\normalcolor
&  =\left[
\begin{array}
[c]{ccc}%
\varrho+\frac{\mu}{2}\left(  \tan\frac{\alpha}{2}+\sin\alpha\right)  & 0 &
\frac{\mu}{2}\left(  1+\cos\alpha\right)
\end{array}
\right]  ^{T},\\%
\color[rgb]{1.0, 0.0, 0.0}%
\mathbf{r}_{4,2,0}%
\normalcolor
&  =\left[
\begin{array}
[c]{ccc}%
\varrho+\frac{3\mu\sin\alpha}{4+2\cos\alpha} & 0 & \frac{3\mu\left(
1+\cos\alpha\right)  }{4+2\cos\alpha}%
\end{array}
\right]  ^{T},\\%
\color[rgb]{1.0, 0.0, 0.0}%
\mathbf{r}_{4,1,0}%
\normalcolor
&  =\left[
\begin{array}
[c]{ccc}%
\varrho+\frac{\mu}{2}\tan\frac{\alpha}{2} & 0 & \mu
\end{array}
\right]  ^{T},\\
& \\%
\color[rgb]{1.0, 0.0, 0.0}%
\mathbf{r}_{4,3,1}%
\normalcolor
&  =\left[
\begin{array}
[c]{ccc}%
\varrho+\frac{3\mu\sin\alpha}{4+2\cos\alpha} & \frac{3\varrho\sin\alpha
+\mu\left(  3+\cos\alpha\right)  \left(  1-\cos\alpha\right)  }{8+4\cos\alpha}
& \frac{3\mu\left(  1+\cos\alpha\right)  }{4+2\cos\alpha}%
\end{array}
\right]  ^{T},\\%
\color[rgb]{1.0, 0.0, 0.0}%
\mathbf{r}_{4,2,1}%
\normalcolor
&  =\left[
\begin{array}
[c]{ccc}%
\varrho+\frac{3\mu\left(  2+\cos\alpha\right)  \sin\alpha}{\left(
1+\cos\alpha\right)  \left(  10+2\cos\alpha\right)  } & \frac{2\varrho\left(
\sin\alpha+\tan\frac{\alpha}{2}\right)  +3\mu\left(  1-\cos\alpha\right)
}{10+2\cos\alpha} & \frac{3\mu\left(  3+\cos\alpha\right)  }{10+2\cos\alpha}%
\end{array}
\right]  ^{T},\\
& \\%
\color[rgb]{1.0, 0.6, 0.0}%
\mathbf{r}_{4,2,2}%
\normalcolor
&  =\left[
\begin{array}
[c]{ccc}%
\frac{6\varrho\left(  3+2\cos\alpha\right)  +\mu\left(  9+\cos\alpha\right)
\sin\alpha}{20+10\cos\alpha} & \frac{10\varrho\sin\alpha+\mu\left(
7+\cos\alpha\right)  \left(  1-\cos\alpha\right)  }{20+10\cos\alpha} &
\frac{3\mu\left(  3+2\cos\alpha\right)  }{10+5\cos\alpha}%
\end{array}
\right]  ^{T},\\
& \\%
\color[rgb]{0.0, 0.4, 0.0}%
\mathbf{g}_{4,4,0}%
\normalcolor
&  =\left[
\begin{array}
[c]{ccc}%
\varrho & 0 & \mu
\end{array}
\right]  ^{T},\\%
\color[rgb]{0.0, 0.4, 0.0}%
\mathbf{g}_{4,3,0}%
\normalcolor
&  =\left[
\begin{array}
[c]{ccc}%
\varrho & \frac{\varrho}{2}\tan\frac{\alpha}{2} & \mu
\end{array}
\right]  ^{T},\\%
\color[rgb]{0.0, 0.4, 0.0}%
\mathbf{g}_{4,2,0}%
\normalcolor
&  =\left[
\begin{array}
[c]{ccc}%
\frac{3\varrho\left(  1+\cos\alpha\right)  }{4+2\cos\alpha} & \frac
{3\varrho\sin\alpha}{4+2\cos\alpha} & \mu
\end{array}
\right]  ^{T},\\%
\color[rgb]{0.0, 0.4, 0.0}%
\mathbf{g}_{4,1,0}%
\normalcolor
&  =\left[
\begin{array}
[c]{ccc}%
\frac{\varrho}{2}\left(  1+\cos\alpha\right)  & \frac{\varrho}{2}\left(
\tan\frac{\alpha}{2}+\sin\alpha\right)  & \mu
\end{array}
\right]  ^{T},\\
& \\%
\color[rgb]{0.0, 0.4, 0.0}%
\mathbf{g}_{4,3,1}%
\normalcolor
&  =\left[
\begin{array}
[c]{ccc}%
\varrho+\frac{3\mu\sin\alpha}{8+4\cos\alpha} & \frac{3\varrho\sin\alpha
+2\mu\left(  1-\cos\alpha\right)  }{8+4\cos\alpha} & \mu
\end{array}
\right]  ^{T},\\%
\color[rgb]{0.0, 0.4, 0.0}%
\mathbf{g}_{4,2,1}%
\normalcolor
&  =\left[
\begin{array}
[c]{ccc}%
\frac{3\varrho\left(  3+\cos\alpha\right)  +3\mu\sin\alpha}{10+2\cos\alpha} &
\frac{3\varrho\left(  3\sin\frac{\alpha}{2}+\sin\frac{3\alpha}{2}\right)
+3\mu\left(  \cos\frac{\alpha}{2}-\cos\frac{3\alpha}{2}\right)  }{\left(
20+4\cos\alpha\right)  \cos\frac{\alpha}{2}} & \mu
\end{array}
\right]  ^{T},\\
& \\%
\color[rgb]{0.0, 0.0, 0.7}%
\mathbf{b}_{4,4,0}%
\normalcolor
&  =\left[
\begin{array}
[c]{ccc}%
\varrho\cos\alpha & \varrho\sin\alpha & \mu
\end{array}
\right]  ^{T},\\%
\color[rgb]{0.0, 0.0, 0.7}%
\mathbf{b}_{4,3,0}%
\normalcolor
&  =\left[
\begin{array}
[c]{ccc}%
\frac{\varrho}{2}\left(  1+\cos\alpha\right)  +\frac{\mu}{2}\tan\frac{\alpha
}{2}\cos\alpha & \frac{\varrho}{2}\left(  \tan\frac{\alpha}{2}+\sin
\alpha\right)  +\frac{\mu}{2}\left(  1-\cos\alpha\right)  & \mu
\end{array}
\right]  ^{T},\\%
\color[rgb]{0.0, 0.0, 0.7}%
\mathbf{b}_{4,2,0}%
\normalcolor
&  =\left[
\begin{array}
[c]{ccc}%
\frac{3\varrho\left(  1+\cos\alpha\right)  }{4+2\cos\alpha}+\frac{\mu}{2}%
\sin\alpha & \frac{3\varrho\sin\alpha+\mu\left(  3+\cos\alpha\right)  \left(
1-\cos\alpha\right)  }{4+2\cos\alpha} & \frac{3\mu\left(  1+\cos\alpha\right)
}{4+2\cos\alpha}%
\end{array}
\right]  ^{T},\\%
\color[rgb]{0.0, 0.0, 0.7}%
\mathbf{b}_{4,1,0}%
\normalcolor
&  =\left[
\begin{array}
[c]{ccc}%
\varrho+\frac{\mu}{2}\left(  \tan\frac{\alpha}{2}+\sin\alpha\right)  &
\frac{\varrho}{2}\tan\frac{\alpha}{2}+\frac{\mu}{2}\left(  1-\cos\alpha\right)
& \frac{\mu}{2}\left(  1+\cos\alpha\right)
\end{array}
\right]  ^{T},\\
& \\%
\color[rgb]{0.0, 0.0, 0.7}%
\mathbf{b}_{4,3,1}%
\normalcolor
&  =\left[
\begin{array}
[c]{ccc}%
\frac{\varrho\left(  7-2\cos^{2}\frac{\alpha}{2}+\cos\frac{\alpha}{2}\right)
\cos^{2}\frac{\alpha}{2}}{4+2\cos\alpha}+\frac{\mu}{4}\sin\alpha &
\frac{6\varrho\sin\alpha+\mu\left(  1-\cos\alpha\right)  \left(  3+\cos
\alpha\right)  }{8+4\cos\alpha} & \mu
\end{array}
\right]  ^{T},\\%
\color[rgb]{0.0, 0.0, 0.7}%
\mathbf{b}_{4,2,1}%
\normalcolor
&  =\left[
\begin{array}
[c]{ccc}%
\frac{3\varrho\left(  3+\cos\alpha\right)  \cos\frac{\alpha}{2}+\mu\left(
3\sin\frac{\alpha}{2}+2\sin\frac{3\alpha}{2}\right)  }{\left(  10+2\cos
\alpha\right)  \cos\frac{\alpha}{2}} & \frac{3\varrho\left(  3\sin\frac
{\alpha}{2}+\sin\frac{3\alpha}{2}\right)  +4\mu\left(  \cos\frac{\alpha}%
{2}-\cos\frac{3\alpha}{2}\right)  }{\left(  20+4\cos\alpha\right)  \cos
\frac{\alpha}{2}} & \frac{3\mu\left(  3+\cos\alpha\right)  }{10+2\cos\alpha}%
\end{array}
\right]  ^{T}.
\end{align*}

Fig.\ \ref{fig:torus} illustrates several triangular trigonometric
patches of order $2$ that are smoothly joined in order to form a part of a ring torus.%
\end{example}

\begin{figure}
[!htb]
\begin{center}
\includegraphics[
natheight=2.534800in,
natwidth=3.578600in,
height=2.5763in,
width=3.6262in
]%
{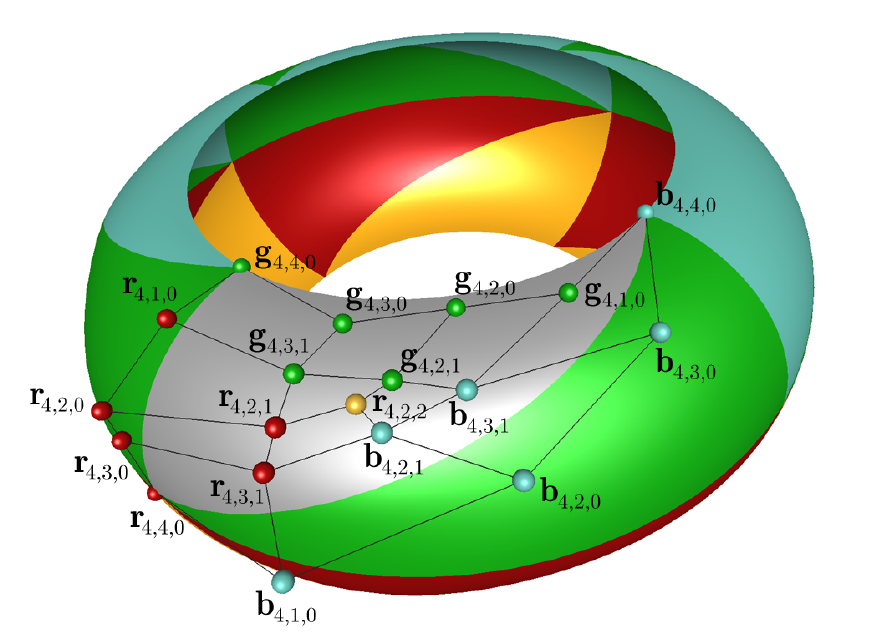}%
\caption{
Control point based exact description of a ring torus by means of smooth
second order triangular trigonometric patches. The
displayed control net generates the toroidal triangle (\ref{toroidal_triangle}) with settings $\mu=3$, $\rho = \frac{1+\sqrt{5}}{2}\mu$ and $\alpha=\frac{\pi}{2}$. A triangular part of a ring torus can also be described by polynomial rational B\'ezier or B-spline triangles. In contrast to existing polynomial triangular patches, our trigonometric modeling tool does not require additional weights that have to be associated with control points. Moreover, our approach can be used to exactly describe triangular parts of other non-rational trigonometric surfaces that can only be approximated by means of classical (rational) polynomial patches.}%
\label{fig:torus}%
\end{center}
\end{figure}

The next subsection introduces the rational counterpart both of the basis
(\ref{normalized_united_system}) and triangular trigonometric patch
(\ref{triangular_trig_surface}).

\subsection{\label{sec:triangular_rational_trigonometric_patches}Rational triangular
trigonometric patches}

By means of non-negative scalar values (weights)%
\[
\rho_{2n,n,n},\left\{  \rho_{2n,2n-i,j}\right\}  _{j=0,i=j}^{n-1,2n-1-j}%
,\left\{  \gamma_{2n,2n-i,j}\right\}  _{j=0,i=j}^{n-1,2n-1-j},\left\{
\beta_{2n,2n-i,j}\right\}  _{j=0,i=j}^{n-1,2n-1-j}%
\]
of rank $1$, i.e.,%
\[
\rho_{2n,n,n}+\sum_{j=0}^{n-1}\sum_{i=j}^{2n-1-j}\rho_{2n,2n-i,j}+\sum
_{j=0}^{n-1}\sum_{i=j}^{2n-1-j}\gamma_{2n,2n-i,j}+\sum_{j=0}^{n-1}\sum
_{i=j}^{2n-1-j}\beta_{2n,2n-i,j}\neq0,
\]
one can also define rational triangular trigonometric patches.

\begin{definition}
[\textbf{Rational triangular trigonometric patches}]\label{def:rational_patch}We refer
to the constrained trivariate vector function $\mathbf{q}_{n}^{\alpha}%
:\Omega^{\alpha}\rightarrow%
\mathbb{R}
^{3}$ of the form%
\begin{equation}%
\begin{array}
[c]{rcl}%
\mathbf{q}_{n}^{\alpha}\left(  u,v,w\right)  & = & \rho_{2n,n,n}%
\mathbf{r}_{2n,n,n}\dfrac{\overline{R}_{2n,n,n}^{\alpha}\left(  u,v,w\right)
}{\overline{\tau}_{n}^{\alpha}\left(  u,v,w\right)  }+%
{\displaystyle\sum\limits_{j=0}^{n-1}}
{\displaystyle\sum\limits_{i=j}^{2n-1-j}}
\rho_{2n,2n-i,j}\mathbf{r}_{2n,2n-i,j}\dfrac{\overline{R}_{2n,2n-i,j}^{\alpha
}\left(  u,v,w\right)  }{\overline{\tau}_{n}^{\alpha}\left(  u,v,w\right)  }\\
&  & +%
{\displaystyle\sum\limits_{j=0}^{n-1}}
{\displaystyle\sum\limits_{i=j}^{2n-1-j}}
\gamma_{2n,2n-i,j}\mathbf{g}_{2n,2n-i,j}\dfrac{\overline{G}_{2n,2n-i,j}%
^{\alpha}\left(  u,v,w\right)  }{\overline{\tau}_{n}^{\alpha}\left(
u,v,w\right)  }\\
&  & +%
{\displaystyle\sum\limits_{j=0}^{n-1}}
{\displaystyle\sum\limits_{i=j}^{2n-1-j}}
\beta_{2n,2n-i,j}\mathbf{b}_{2n,2n-i,j}\dfrac{\overline{B}_{2n,2n-i,j}%
^{\alpha}\left(  u,v,w\right)  }{\overline{\tau}_{n}^{\alpha}\left(
u,v,w\right)  }%
\end{array}
\label{rational_patch}%
\end{equation}
as rational triangular trigonometric patch of order $n\geq1$, where vectors
\[
\left\{  \mathbf{r}_{2n,2n-i,j}\right\}  _{j=0,i=j}^{n,2n-1-j}\cup\left\{
\mathbf{g}_{2n,2n-i,j}\right\}  _{j=0,i=j}^{n-1,2n-1-j}\cup\left\{
\mathbf{b}_{2n,2n-i,j}\right\}  _{j=0,i=j}^{n-1,2n-1-j}\subset%
\mathbb{R}
^{3}%
\]
define its control net and%
\[%
\begin{array}
[c]{rcl}%
\overline{\tau}_{n}^{\alpha}\left(  u,v,w\right)  & = & \rho_{2n,n,n}%
\overline{R}_{2n,n,n}^{\alpha}\left(  u,v,w\right)  +%
{\displaystyle\sum\limits_{\ell=0}^{n-1}}
{\displaystyle\sum\limits_{k=\ell}^{2n-1-\ell}}
\rho_{2n,2n-k,\ell}\overline{R}_{2n,2n-k,\ell}^{\alpha}\left(  u,v,w\right) \\
&  & +%
{\displaystyle\sum\limits_{\ell=0}^{n-1}}
{\displaystyle\sum\limits_{k=\ell}^{2n-1-\ell}}
\gamma_{2n,2n-k,\ell}\overline{G}_{2n,2n-k,\ell}^{\alpha}+%
{\displaystyle\sum\limits_{\ell=0}^{n-1}}
{\displaystyle\sum\limits_{k=\ell}^{2n-1-\ell}}
\beta_{2n,2n-k,\ell}\overline{B}_{2n,2n-k,\ell}^{\alpha}\left(  u,v,w\right)
.
\end{array}
\]
Quotient functions in (\ref{rational_patch}) determine the constrained
trivariate rational trigonometric function system $\overline{Q}_{2n}^{\alpha}$
of order $n$ that inherits advantageous properties of $\overline{T}%
_{2n}^{\alpha}$, i.e., $\overline{Q}_{2n}^{\alpha}$ is also
normalized and linearly independent.
\end{definition}

Patch (\ref{rational_patch}) is closed for the projective
transformation of its control points, i.e., the patch determined by the
projectively transformed control points coincides with the pointwisely
transformed patch. 

As an example, Fig.\ \ref{fig:dupin_03} shows the control point based exact description of the ring Dupin cyclide
\begin{equation}
\frac{1}{a-c\cos\left(  u+\varphi\right)  \cos\left(  v+\psi\right)  }\left[
\begin{array}
[c]{l}%
\mu\left(  c-a\cos\left(  u+\varphi\right)  \cos\left(  v+\psi\right)
\right)  +b^{2}\cos\left(  u+\varphi\right) \\
\left(  a-\mu\cos\left(  v+\psi\right)  \right)  b\sin\left(  u+\varphi\right)
\\
\left(  c\cos\left(  u+\varphi\right)  -\mu\right)  b\sin\left(
v+\psi\right)
\end{array}
\right]  ,~u\in\left[  0,\alpha\right]  ,~v\in\left[  0,\alpha-u\right]  ,
\label{Dupin_in_3d}%
\end{equation}
by means of sixteen smoothly joined second order rational triangular trigonometric patches, where parameters $a$, $b$, $c$ and $\mu$ have to fulfill the conditions $a^{2}=b^{2}+c^{2}$ and $c<\mu\leq a$, while phase changes $\left(  \varphi,\psi\right)  \in\left[  0,2\pi\right]\times\left[  0,2\pi\right] $ are free design parameters that can be used to
slide the patch on the cyclide.

\begin{figure}
[!htb]
\begin{center}
\includegraphics[
natheight=4.430300in,
natwidth=6.004300in,
height=4.4581in,
width=6.0321in
]%
{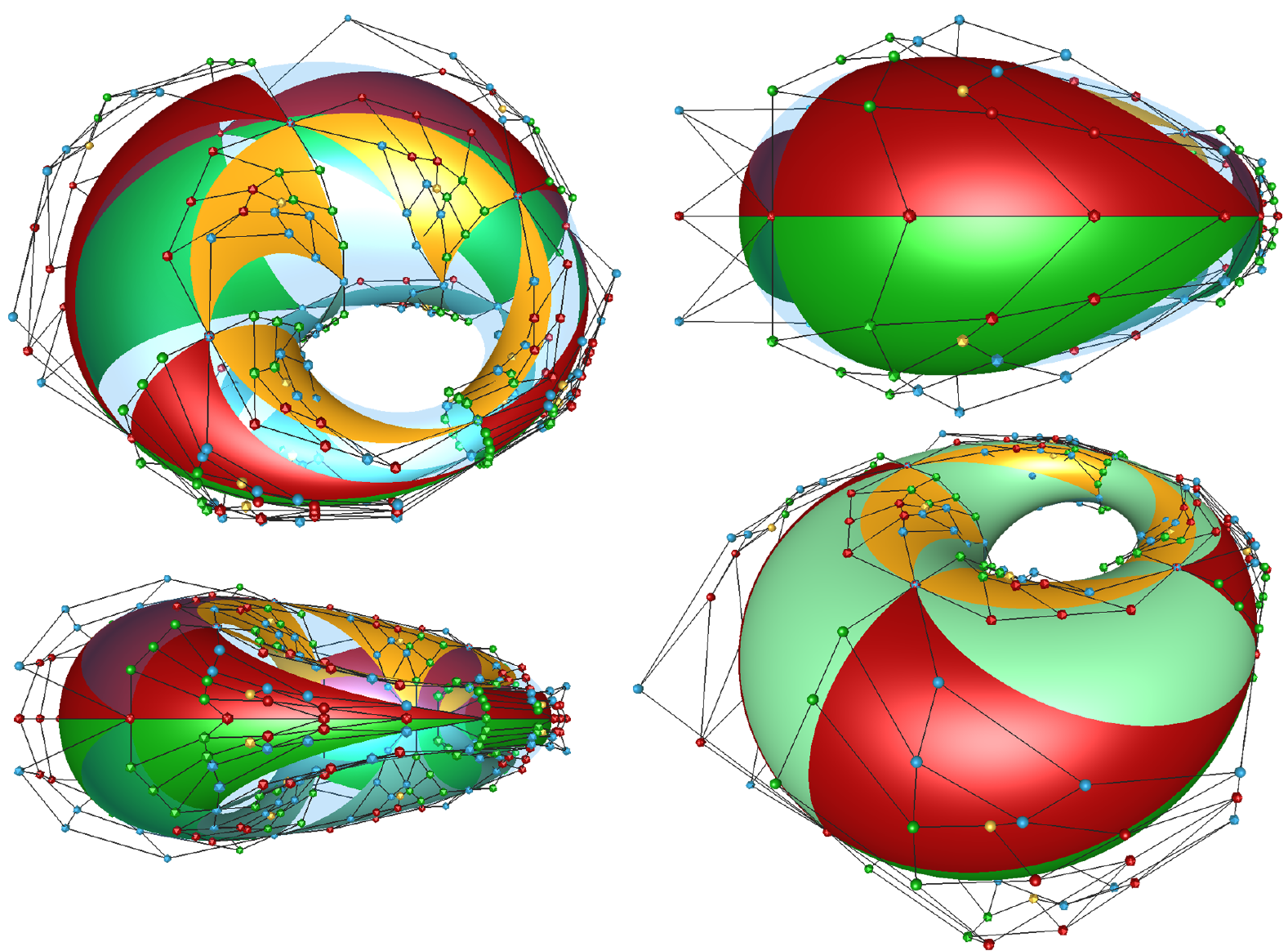}%
\caption{A possible control point based exact triangulization of a ring Dupin
cyclide (with settings $a = 6$, $b=4\sqrt{2}$, $c=2$ and $\mu=3$) which is rendered as a
transparent surface except the lower right figure. The shape parameter of each
second order rational triangular trigonometric patch is $\alpha=\frac{\pi}{2}$. (For interpretation
of the references to color in this figure legend, the reader is referred to
the web version of this paper.)}%
\label{fig:dupin_03}%
\end{center}
\end{figure}

In the next section we list several open problems that we could not solve for the time being.

\section{\label{sec:open_problems}Open problems}
Compared to the classical constrained trivariate Bernstein polynomials and
corresponding triangular patches, in this non-polynomial case, theoretical
questions are significantly harder to answer even for special values of the
order $n$. Assuming arbitrary values of the order $n$, this section formulates several open questions as follows.

\begin{question}
Are the normalizing coefficients of the constrained trivariate function system $T_{2n}^{\alpha}$ non-negative and symmetric? Can we give closed or at least recursive formulas for these coefficients?
\end{question}

Section \ref{sec:partition_of_unity} already described a technique to
determine the unique normalizing coefficients of the
constrained trivariate function system $T_{2n}^{\alpha}$. Due to the
complexity of this problem, closed formulas of these coefficients were given
only for levels $0$ and $1$ for arbitrary order, and for all levels
just for orders $n=1$, $2$ and $3$.

However, if one is able to transform the product
\[%
\begin{array}
[c]{rcl}%
1^{n+1} & = & \left(  \overline{R}_{2n,n,n}^{\alpha}+%
{\displaystyle\sum\limits_{j=0}^{n-1}}
{\displaystyle\sum\limits_{i=j}^{2n-1-j}}
\overline{R}_{2n,2n-i,j}^{\alpha}+%
{\displaystyle\sum\limits_{j=0}^{n-1}}
{\displaystyle\sum\limits_{i=j}^{2n-1-j}}
\overline{G}_{2n,2n-i,j}^{\alpha}+%
{\displaystyle\sum\limits_{j=0}^{n-1}}
{\displaystyle\sum\limits_{i=j}^{2n-1-j}}
\overline{B}_{2n,2n-i,j}^{\alpha}\right)  \\
&  & \\
&  & \cdot\left(  \overline{R}_{2,1,1}^{\alpha}+\overline{R}_{2,2,0}^{\alpha
}+\overline{R}_{2,1,0}^{\alpha}+\overline{G}_{2,2,0}^{\alpha}+\overline
{G}_{2,1,0}^{\alpha}+\overline{B}_{2,2,0}^{\alpha}+\overline{B}_{2,1,0}%
^{\alpha}\right)
\end{array}
\]
into the basis $\overline{T}_{2\left(  n+1\right)}^{\alpha}$, then this transformation would also provide:
\begin{itemize}
\item a recurrence relation between normalizing coefficients of consecutive orders;
\item the symmetry and non-negativity of normalizing coefficients;
\item the order elevation of the constrained trivariate trigonometric
blending system $\overline{T}_{2n}^{\alpha}$ (and consequently, the order elevation of triangular (rational) trigonometric patches
from arbitrary order $n$ to $n+1$).
\end{itemize}

In contrast to the technique presented in the proof of
Proposition \ref{prop:level_1_coefficients} (cf. \cite{tech_report}), the immediate inheritance of
non-negativity and symmetry of normalizing coefficients from lower to higher order
forms the advantage of this alternative method (consider Example \ref{exmp:from_1_to_2}), while its disadvantage lies in
its complexity and the tedious use of quite involved trigonometric identities.

For instance, it is relatively easy to transform the mixed products%
\[
\overline{R}_{2n,2n-i,j}^{\alpha}\cdot\overline{R}_{2,1,1}^{\alpha}%
,~\overline{R}_{2n,2n-i,j}^{\alpha}\cdot\overline{R}_{2,2,0}^{\alpha
},~\overline{R}_{2n,2n-i,j}^{\alpha}\cdot\overline{R}_{2,1,0}^{\alpha
},~\overline{R}_{2n,2n-i,j}^{\alpha}\cdot\overline{G}_{2,2,0}^{\alpha
}%
\]
into the basis $\overline{T}_{2\left(  n+1\right)  }^{\alpha}$ for all indices
$j=0,1,\ldots,n-1$ and $i=j,j+1,\ldots,2n-1-j$. However, in case of pairwise products%
\[
\overline{R}_{2n,2n-i,j}^{\alpha}\cdot\overline{B}_{2,2,0}^{\alpha},~
\overline{R}_{2n,2n-i,j}^{\alpha}\cdot\overline{G}_{2,1,0}^{\alpha}%
,~\overline{R}_{2n,2n-i,j}^{\alpha}\cdot\overline{B}_{2,1,0}^{\alpha}%
\]
one obtains functions that we could not write as the linear combination
of basis functions of $\overline{T}_{2\left(  n+1\right)  }^{\alpha}$, for the present. For the moment, we are only able (cf. \cite[Lemma 5.1]{tech_report}) to transform the pairwise products of first order blending functions into linear combinations of second order blending functions. Some immediate corollaries of this special transformation are presented in Examples \ref{exmp:from_1_to_2} and \ref{exmp:degree_elevation_from_1_to_2}.

\begin{example}[\textbf{Relation between first and second order normalizing coefficients}]
\label{exmp:from_1_to_2}
If one rewrites the square of unity%
\[
1^{2}=\left(  \overline{R}_{2,1,1}^{\alpha}+\overline{R}_{2,2,0}^{\alpha
}+\overline{R}_{2,1,0}^{\alpha}+\overline{G}_{2,2,0}^{\alpha}+\overline
{G}_{2,1,0}^{\alpha}+\overline{B}_{2,2,0}^{\alpha}+\overline{B}_{2,1,0}%
^{\alpha}\right)  ^{2}%
\]
into the second order basis $\overline{T}_{4}^{\alpha}$, we obtain the relations%
\begin{align*}
\left[
\begin{array}
[c]{c}%
\Red{r_{4,4,0}^{\alpha}}\\
\Green{g_{4,4,0}^{\alpha}}\\
\Blue{b_{4,4,0}^{\alpha}}%
\end{array}
\right]   &  =\left[
\begin{array}
[c]{c}%
\left(  \Red{r_{2,2,0}^{\alpha}}\right)  ^{2}\\
\left(  \Green{g_{2,2,0}^{\alpha}}\right)  ^{2}\\
\left(  \Blue{b_{2,2,0}^{\alpha}}\right)  ^{2}%
\end{array}
\right]  ,~\left[
\begin{array}
[c]{c}%
\Red{r_{4,3,0}^{\alpha}}\\
\Green{g_{4,3,0}^{\alpha}}\\
\Blue{b_{4,3,0}^{\alpha}}%
\end{array}
\right]  =\left[
\begin{array}
[c]{c}%
2\Red{r_{2,2,0}^{\alpha}r_{2,1,0}^{\alpha}}\\
2\Green{g_{2,2,0}^{\alpha}g_{2,1,0}^{\alpha}}\\
2\Blue{b_{2,2,0}^{\alpha}b_{2,1,0}^{\alpha}}%
\end{array}
\right]  ,~\left[
\begin{array}
[c]{c}%
\Red{r_{4,2,0}^{\alpha}}\\
\Green{g_{4,2,0}^{\alpha}}\\
\Blue{b_{4,2,0}^{\alpha}}%
\end{array}
\right]  =\left[
\begin{array}
[c]{c}%
2\Red{r_{2,2,0}^{\alpha}}\Green{g_{2,2,0}^{\alpha}}+\left(  \Red{r_{2,1,0}^{\alpha}}\right)
^{2}\\
2\Green{g_{2,2,0}^{\alpha}}\Blue{b_{2,2,0}^{\alpha}}+\left(  \Green{g_{2,1,0}^{\alpha}}\right)
^{2}\\
2\Red{r_{2,2,0}^{\alpha}}\Blue{b_{2,2,0}^{\alpha}}+\left(  \Blue{b_{2,1,0}^{\alpha}}\right)  ^{2}%
\end{array}
\right]  ,\\
\left[
\begin{array}
[c]{c}%
\Red{r_{4,1,0}^{\alpha}}\\
\Green{g_{4,1,0}^{\alpha}}\\
\Blue{b_{4,1,0}^{\alpha}}%
\end{array}
\right]   &  =\left[
\begin{array}
[c]{c}%
2\Red{r_{2,1,0}^{\alpha}}\Green{g_{2,2,0}^{\alpha}}\\
2\Green{g_{2,1,0}^{\alpha}}\Blue{b_{2,2,0}^{\alpha}}\\
2\Red{r_{2,2,0}^{\alpha}}\Blue{b_{2,1,0}^{\alpha}}%
\end{array}
\right]  ,\\
\left[
\begin{array}
[c]{c}%
\Red{r_{4,3,1}^{\alpha}}\\
\Green{g_{4,3,1}^{\alpha}}\\
\Blue{b_{4,3,1}^{\alpha}}%
\end{array}
\right]   &  =\left[
\begin{array}
[c]{c}%
2\Red{r_{2,2,0}^{\alpha}}\Orange{r_{2,1,1}^{\alpha}}+\left(  \Red{r_{2,2,0}^{\alpha}}\right)
^{2}\Green{g_{2,1,0}^{\alpha}}\sin\alpha+2\left(  1+\Red{r_{2,2,0}^{\alpha}}\cos^{2}%
\frac{\alpha}{2}\right)  \Red{r_{2,1,0}^{\alpha}}\Blue{b_{2,1,0}^{\alpha}}\sin\frac{\alpha
}{2}\\
2\Orange{r_{2,1,1}^{\alpha}}\Green{g_{2,2,0}^{\alpha}}+2\Red{r_{2,1,0}^{\alpha}}\left(
1+\Green{g_{2,2,0}^{\alpha}}\cos^{2}\frac{\alpha}{2}\right)  \Green{g_{2,1,0}^{\alpha}}%
\sin\frac{\alpha}{2}+\left(  \Green{g_{2,2,0}^{\alpha}}\right)  ^{2}\Blue{b_{2,1,0}^{\alpha}}\sin\alpha\\
2\Orange{r_{2,1,1}^{\alpha}}\Blue{b_{2,2,0}^{\alpha}}+\Red{r_{2,1,0}^{\alpha}}\left(  \Blue{b_{2,2,0}%
^{\alpha}}\right)  ^{2}\sin\alpha+2\Green{g_{2,1,0}^{\alpha}}\left(  1+\Blue{b_{2,2,0}%
^{\alpha}}\cos^{2}\frac{\alpha}{2}\right)  \Blue{b_{2,1,0}^{\alpha}}\sin\frac{\alpha
}{2}%
\end{array}
\right]  ,\\
\left[
\begin{array}
[c]{c}%
\Red{r_{4,2,1}^{\alpha}}\\
\\
\\
\\
\Green{g_{4,2,1}^{\alpha}}\\
\\
\\
\Blue{b_{4,2,1}^{\alpha}}\\
\end{array}
\right]   &  =\left[
\begin{array}
[c]{l}%
2\Red{r_{2,1,0}^{\alpha}}\Orange{r_{2,1,1}^{\alpha}}+\Red{r_{2,2,0}^{\alpha}}\Red{r_{2,1,0}^{\alpha
}}\Green{g_{2,1,0}^{\alpha}}\tan\frac{\alpha}{2}+\left(  \Red{r_{2,1,0}^{\alpha}}\right)
^{2}\Green{g_{2,1,0}^{\alpha}}\sin\frac{\alpha}{2}\\
+\Red{r_{2,1,0}^{\alpha}}\left(  \Blue{b_{2,1,0}^{\alpha}}\right)  ^{2}\sin\frac{\alpha}%
{2}+\Red{r_{2,1,0}^{\alpha}}\Green{g_{2,2,0}^{\alpha}}\Blue{b_{2,1,0}^{\alpha}}\tan\frac{\alpha}%
{2}\\
\\
2\Orange{r_{2,1,1}^{\alpha}}\Green{g_{2,1,0}^{\alpha}}+\Red{r_{2,1,0}^{\alpha}}\left(  \Green{g_{2,1,0}%
^{\alpha}}\right)  ^{2}\sin\frac{\alpha}{2}+\Red{r_{2,1,0}^{\alpha}}\Green{g_{2,1,0}%
^{\alpha}}\Blue{b_{2,2,0}^{\alpha}}\tan\frac{\alpha}{2}\\
+\Green{g_{2,2,0}^{\alpha}g_{2,1,0}^{\alpha}}\Blue{b_{2,1,0}^{\alpha}}\tan\frac{\alpha}%
{2}+\left(  \Green{g_{2,1,0}^{\alpha}}\right)  ^{2}\Blue{b_{2,1,0}^{\alpha}}\sin\frac{\alpha
}{2}\\
\\
2\Orange{r_{2,1,1}^{\alpha}}\Blue{b_{2,1,0}^{\alpha}}+\Red{r_{2,2,0}^{\alpha}}\Green{g_{2,1,0}^{\alpha
}}\Blue{b_{2,1,0}^{\alpha}}\tan\frac{\alpha}{2}+\Red{r_{2,1,0}^{\alpha}}\Blue{b_{2,2,0}^{\alpha
}b_{2,1,0}^{\alpha}}\tan\frac{\alpha}{2}\\
+\left(  \Red{r_{2,1,0}^{\alpha}}\right)  ^{2}\Blue{b_{2,1,0}^{\alpha}}\sin\frac{\alpha}%
{2}+\Green{g_{2,1,0}^{\alpha}}\left(  \Blue{b_{2,1,0}^{\alpha}}\right)  ^{2}\sin\frac{\alpha
}{2}%
\end{array}
\right]  ,\\
\Orange{r_{4,2,2}^{\alpha}} &  =\left(  \Orange{r_{2,1,1}^{\alpha}}\right)  ^{2}+2\Red{r_{2,2,0}%
^{\alpha}}\Green{g_{2,2,0}^{\alpha}}+2\Red{r_{2,2,0}^{\alpha}}\Blue{b_{2,2,0}^{\alpha}}%
+2\Green{g_{2,2,0}^{\alpha}}\Blue{b_{2,2,0}^{\alpha}}+2\Red{r_{2,1,0}^{\alpha}}\Green{g_{2,1,0}^{\alpha
}}+2\Red{r_{2,1,0}^{\alpha}}\Blue{b_{2,1,0}^{\alpha}}+2\Green{g_{2,1,0}^{\alpha}}\Blue{b_{2,1,0}^{\alpha}}
\end{align*}
between first and second order normalizing coefficients. It is easy to verify that these relations lead to the same second order
normalizing coefficients that were described in Example
\ref{normalizing_coefficients_n_2}. Observe that the symmetry and non-negativity
of normalizing coefficients of order $1$ are inherited by those of order
$2$.
\end{example}

\begin{example}[\textbf{Order elevation from $1$ to $2$}]
\label{exmp:degree_elevation_from_1_to_2}Consider a first order triangular trigonometric patch $\mathbf{s}_{1}^{\alpha}$ of type (\ref{triangular_trig_surface}). Rewriting the product%
\begin{align*}
\mathbf{s}_{1}^{\alpha}  \cdot1=  &  \left(
\Orange{\mathbf{r}_{2,1,1}}\overline{R}_{2,1,1}^{\alpha}+\Red{\mathbf{r}_{2,2,0}}\overline
{R}_{2,2,0}^{\alpha}+\Red{\mathbf{r}_{2,1,0}}\overline{R}_{2,1,0}^{\alpha
}+\Green{\mathbf{g}_{2,2,0}}\overline{G}_{2,2,0}^{\alpha}+\Green{\mathbf{g}_{2,1,0}}%
\overline{G}_{2,1,0}^{\alpha}+\Blue{\mathbf{b}_{2,2,0}}\overline{B}_{2,2,0}^{\alpha
}+\Blue{\mathbf{b}_{2,1,0}}\overline{B}_{2,1,0}^{\alpha}\right) \\
&  \cdot\left(  \overline{R}_{2,1,1}^{\alpha}+\overline{R}_{2,2,0}^{\alpha
}+\overline{R}_{2,1,0}^{\alpha}+\overline{G}_{2,2,0}^{\alpha}+\overline
{G}_{2,1,0}^{\alpha}+\overline{B}_{2,2,0}^{\alpha}+\overline{B}_{2,1,0}%
^{\alpha}\right)
\end{align*}
into the basis $\overline{T}_{4}^{\alpha}$ and then collecting the vector
coefficients of second order blending functions, one obtains the order elevated (i.e., second order) representation $\mathbf{s}_{2}^{\alpha}$ of the original patch $\mathbf{s}_{1}^{\alpha}$. Moreover, $\mathbf{s}_{2}^{\alpha}$ is generated by control points
\begin{align*}
\left[
\begin{array}
[c]{c}%
\Red{\mathbf{r}_{4,4,0}}
\\%
\Green{\mathbf{g}_{4,4,0}}
\\%
\Blue{\mathbf{b}_{4,4,0}}
\end{array}
\right]  =  &  \left[
\begin{array}
[c]{c}%
\Red{\mathbf{r}_{2,2,0}}
\\%
\Green{\mathbf{g}_{2,2,0}}
\\%
\Blue{\mathbf{b}_{2,2,0}}
\end{array}
\right]  ,\\
\left[
\begin{array}
[c]{c}%
\Red{\mathbf{r}_{4,3,0}}
\\%
\Green{\mathbf{g}_{4,3,0}}
\\%
\Blue{\mathbf{b}_{4,3,0}}
\end{array}
\right]  =  &  \frac{1}{2}\left[
\begin{array}
[c]{c}%
\Red{\mathbf{r}_{2,2,0}}
\\%
\Green{\mathbf{g}_{2,2,0}}
\\%
\Blue{\mathbf{b}_{2,2,0}}
\end{array}
\right]  +\frac{1}{2}\left[
\begin{array}
[c]{c}%
\Red{\mathbf{r}_{2,1,0}}
\\%
\Green{\mathbf{g}_{2,1,0}}
\\%
\Blue{\mathbf{b}_{2,1,0}}
\end{array}
\right]  ,\\
\left[
\begin{array}
[c]{c}%
\Red{\mathbf{r}_{4,2,0}}
\\%
\Green{\mathbf{g}_{4,2,0}}
\\%
\Blue{\mathbf{b}_{4,2,0}}
\end{array}
\right]  =  &  \frac{1}{2+4\cos^{2}\frac{\alpha}{2}}\left[
\begin{array}
[c]{c}%
\Red{\mathbf{r}_{2,2,0}}
\\%
\Green{\mathbf{g}_{2,2,0}}
\\%
\Blue{\mathbf{b}_{2,2,0}}
\end{array}
\right]  +\frac{2\cos^{2}\frac{\alpha}{2}}{1+2\cos^{2}\frac{\alpha}{2}}\left[
\begin{array}
[c]{c}%
\Red{\mathbf{r}_{2,1,0}}
\\%
\Green{\mathbf{g}_{2,1,0}}
\\%
\Blue{\mathbf{b}_{2,1,0}}
\end{array}
\right]  +\frac{1}{2+4\cos^{2}\frac{\alpha}{2}}\left[
\begin{array}
[c]{c}%
\Green{\mathbf{g}_{2,2,0}}
\\%
\Blue{\mathbf{b}_{2,2,0}}
\\%
\Red{\mathbf{r}_{2,2,0}}
\end{array}
\right]  ,\\
\left[
\begin{array}
[c]{c}%
\Red{\mathbf{r}_{4,1,0}}
\\%
\Green{\mathbf{g}_{4,1,0}}
\\%
\Blue{\mathbf{b}_{4,1,0}}
\end{array}
\right]  =  &  \frac{1}{2}\left[
\begin{array}
[c]{c}%
\Red{\mathbf{r}_{2,1,0}}
\\%
\Green{\mathbf{g}_{2,1,0}}
\\%
\Blue{\mathbf{b}_{2,1,0}}
\end{array}
\right]  +\frac{1}{2}\left[
\begin{array}
[c]{c}%
\Green{\mathbf{g}_{2,2,0}}
\\%
\Blue{\mathbf{b}_{2,2,0}}
\\%
\Red{\mathbf{r}_{2,2,0}}
\end{array}
\right]  ,\\
\left[
\begin{array}
[c]{c}%
\Red{\mathbf{r}_{4,3,1}}
\\%
\Green{\mathbf{g}_{4,3,1}}
\\%
\Blue{\mathbf{b}_{4,3,1}}
\end{array}
\right]  =  &  \frac{1}{2+4\cos^{2}\frac{\alpha}{2}}\left[
\begin{array}
[c]{c}%
\Red{\mathbf{r}_{2,2,0}}
\\%
\Green{\mathbf{g}_{2,2,0}}
\\%
\Blue{\mathbf{b}_{2,2,0}}
\end{array}
\right]  +\frac{\cos^{2}\frac{\alpha}{2}}{1+2\cos^{2}\frac{\alpha}{2}}\left[
\begin{array}
[c]{c}%
\Red{\mathbf{r}_{2,1,0}}
\\%
\Green{\mathbf{g}_{2,1,0}}
\\%
\Blue{\mathbf{b}_{2,1,0}}
\end{array}
\right]  +\frac{\sin^{2}\frac{\alpha}{2}}{2+4\cos^{2}\frac{\alpha}{2}}\left[
\begin{array}
[c]{c}%
\Orange{\mathbf{r}_{2,1,1}}
\\%
\Orange{\mathbf{r}_{2,1,1}}
\\%
\Orange{\mathbf{r}_{2,1,1}}
\end{array}
\right] \\
&  +\frac{\cos^{2}\frac{\alpha}{2}}{2+4\cos^{2}\frac{\alpha}{2}}\left[
\begin{array}
[c]{c}%
\Green{\mathbf{g}_{2,1,0}}
\\%
\Blue{\mathbf{b}_{2,1,0}}
\\%
\Red{\mathbf{r}_{2,1,0}}
\end{array}
\right]  +\frac{\cos^{2}\frac{\alpha}{2}}{1+2\cos^{2}\frac{\alpha}{2}}\left[
\begin{array}
[c]{c}%
\Blue{\mathbf{b}_{2,1,0}}
\\%
\Red{\mathbf{r}_{2,1,0}}
\\%
\Green{\mathbf{g}_{2,1,0}}
\end{array}
\right]  ,\\
\left[
\begin{array}
[c]{c}%
\Red{\mathbf{r}_{4,2,1}}
\\%
\Green{\mathbf{g}_{4,2,1}}
\\%
\Blue{\mathbf{b}_{4,2,1}}
\end{array}
\right]  =  &  \frac{1}{8+4\cos^{2}\frac{\alpha}{2}}\left[
\begin{array}
[c]{c}%
\Red{\mathbf{r}_{2,2,0}}
\\%
\Green{\mathbf{g}_{2,2,0}}
\\%
\Blue{\mathbf{b}_{2,2,0}}
\end{array}
\right]  +\frac{1+\cos^{2}\frac{\alpha}{2}}{4+2\cos^{2}\frac{\alpha}{2}%
}\left[
\begin{array}
[c]{c}%
\Red{\mathbf{r}_{2,1,0}}
\\%
\Green{\mathbf{g}_{2,1,0}}
\\%
\Blue{\mathbf{b}_{2,1,0}}
\end{array}
\right]  +\frac{\sin^{2}\frac{\alpha}{2}}{4+2\cos^{2}\frac{\alpha}{2}}\left[
\begin{array}
[c]{c}%
\Orange{\mathbf{r}_{2,1,1}}
\\%
\Orange{\mathbf{r}_{2,1,1}}
\\%
\Orange{\mathbf{r}_{2,1,1}}
\end{array}
\right] \\
&  +\frac{1}{8+4\cos^{2}\frac{\alpha}{2}}\left[
\begin{array}
[c]{c}%
\Green{\mathbf{g}_{2,2,0}}
\\%
\Blue{\mathbf{b}_{2,2,0}}
\\%
\Red{\mathbf{r}_{2,2,0}}
\end{array}
\right]  +\frac{1+2\cos^{2}\frac{\alpha}{2}}{8+4\cos^{2}\frac{\alpha}{2}%
}\left[
\begin{array}
[c]{c}%
\Green{\mathbf{g}_{2,1,0}}
\\%
\Blue{\mathbf{b}_{2,1,0}}
\\%
\Red{\mathbf{r}_{2,1,0}}
\end{array}
\right]  +\frac{1+2\cos^{2}\frac{\alpha}{2}}{8+4\cos^{2}\frac{\alpha}{2}%
}\left[
\begin{array}
[c]{c}%
\Blue{\mathbf{b}_{2,1,0}}
\\%
\Red{\mathbf{r}_{2,1,0}}
\\%
\Green{\mathbf{g}_{2,1,0}}
\end{array}
\right]  ,\\%
\Orange{\mathbf{r}_{4,2,2}}
=  &  \frac{1}{5+10\cos^{2}\frac{\alpha}{2}}%
\Red{\mathbf{r}_{2,2,0}}
+\frac{4\cos^{2}\frac{\alpha}{2}}{5+10\cos^{2}\frac{\alpha}{2}}%
\Red{\mathbf{r}_{2,1,0}}
+\frac{2\sin^{2}\frac{\alpha}{2}}{5+10\cos^{2}\frac{\alpha}{2}}%
\Orange{\mathbf{r}_{2,1,1}}
\\
&  +\frac{1}{5+10\cos^{2}\frac{\alpha}{2}}%
\Green{\mathbf{g}_{2,2,0}}
+\frac{4\cos^{2}\frac{\alpha}{2}}{5+10\cos^{2}\frac{\alpha}{2}}%
\Green{\mathbf{g}_{2,1,0}}
\\
&  +\frac{1}{5+10\cos^{2}\frac{\alpha}{2}}%
\Blue{\mathbf{b}_{2,2,0}}
+\frac{4\cos^{2}\frac{\alpha}{2}}{5+10\cos^{2}\frac{\alpha}{2}}%
\Blue{\mathbf{b}_{2,1,0}}
.
\end{align*}
Observe that control points%
\[%
\Orange{\mathbf{r}_{4,2,2}}
,~\left\{
\Red{\mathbf{r}_{4,4-i,j}}
\right\}  _{j=0,i=j}^{1,3-j},~\left\{
\Green{\mathbf{g}_{4,4-i,j}}
\right\}  _{j=0,i=j}^{1,3-j},~\left\{
\Blue{\mathbf{b}_{4,4-i,j}}
\right\}  _{j=0,i=j}^{1,3-j}%
\]
are in fact convex combinations of different subsets of control points%
\[%
\Orange{\mathbf{r}_{2,1,1}}
,~%
\Red{\mathbf{r}_{2,2,0}}
,~%
\Red{\mathbf{r}_{2,1,0}}
,~%
\Green{\mathbf{g}_{2,2,0}}
,~%
\Green{\mathbf{g}_{2,1,0}}
,~%
\Blue{\mathbf{b}_{2,2,0}}
,~%
\Blue{\mathbf{b}_{2,1,0}}
.
\]
Thus, the degree elevated (second order) control net is closer to the given patch
than its original (first order) one for all values of $\alpha$. Fig.\ \ref{fig:degree_elevation_from_1_to_2} illustrates this phenomenon for
different values of the shape parameter $\alpha$.
\end{example}

\begin{figure}
[!htb]
\begin{center}
\includegraphics[
height=5.6455in,
width=6.4714in
]%
{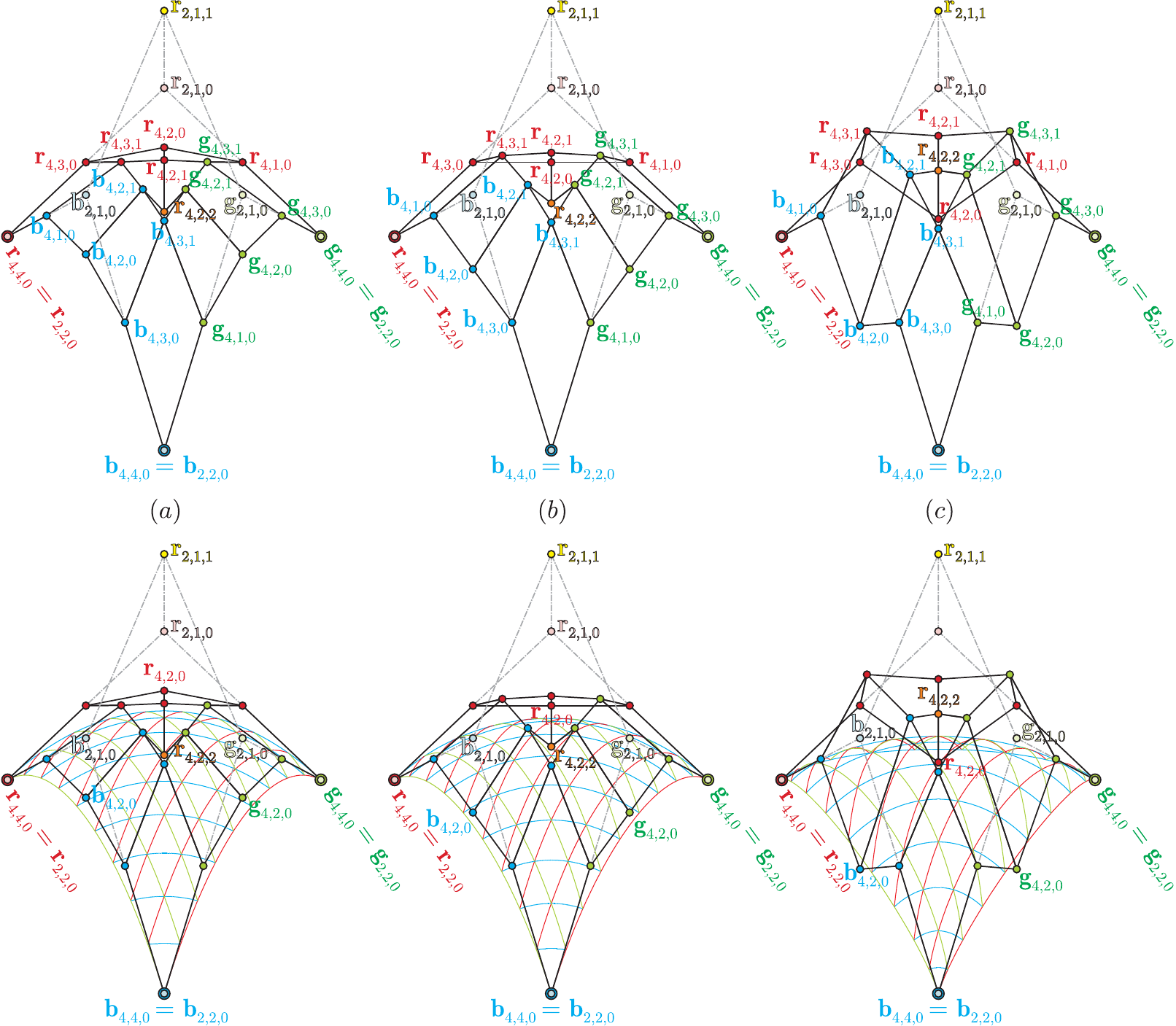}%
\caption{Order elevation of a first order triangular trigonometric patch for different values of the
shape parameter $\alpha$; in cases (\emph{a}), (\emph{b}) and (\emph{c}) the parameter $\alpha$ fulfills the conditions
$0<\alpha<\frac{\pi}{2}$, $\alpha=\frac{\pi}{2}$ and $\frac{\pi}{2}<\alpha
<\pi$, respectively.}%
\label{fig:degree_elevation_from_1_to_2}%
\end{center}
\end{figure}

\begin{question}
\label{que:open_problem_basis_transformation}
What is the general (inverse) transformation between the constrained trivariate bases
$\overline{T}_{2n}^{\alpha}$ and $V_{n}^{\alpha}$?
\end{question}

Answering Question \ref{que:open_problem_basis_transformation} would be important in the control point based exact description of triangular patches of (rational) trigonometric surfaces the coordinate functions of which are given
in (the rational counterpart of) $\mathcal{V}_{n}^{\alpha}$.

Question \ref{que:limiting_case} formulated below is related to the shape of the triangular trigonometric patch (\ref{triangular_trig_surface}) in the limiting case $\alpha \to 0$. The question is motivated by the following observations. Using the parametrization
\[
\left\{
\begin{array}{rcll}
u\left(x,y\right) & = & \alpha x, & x \in \left[0, 1\right],\\
v\left(x,y\right) & = & \alpha y, & y \in \left[0, x\right],\\
w\left(x,y\right) & = & \alpha z, & z = 1-x-y
\end{array}
\right.
\]
of the domain $\Omega^{\alpha}$ and the well-known identity $\lim_{t\to 0}\frac{\sin\left(t\right)}{t}=1$, it is easy to observe that the first, second and third order normalized constrained trivariate trigonometric function systems of type (\ref{normalized_united_system}) degenerate to constrained trivariate Bernstein polynomials defined on the unit simplex $x+y+z=1$ as shown in Fig.\ \ref{fig:limit_cases_n_1_2_3}.

\begin{figure}[!htb]
\center
\includegraphics[scale = 1]{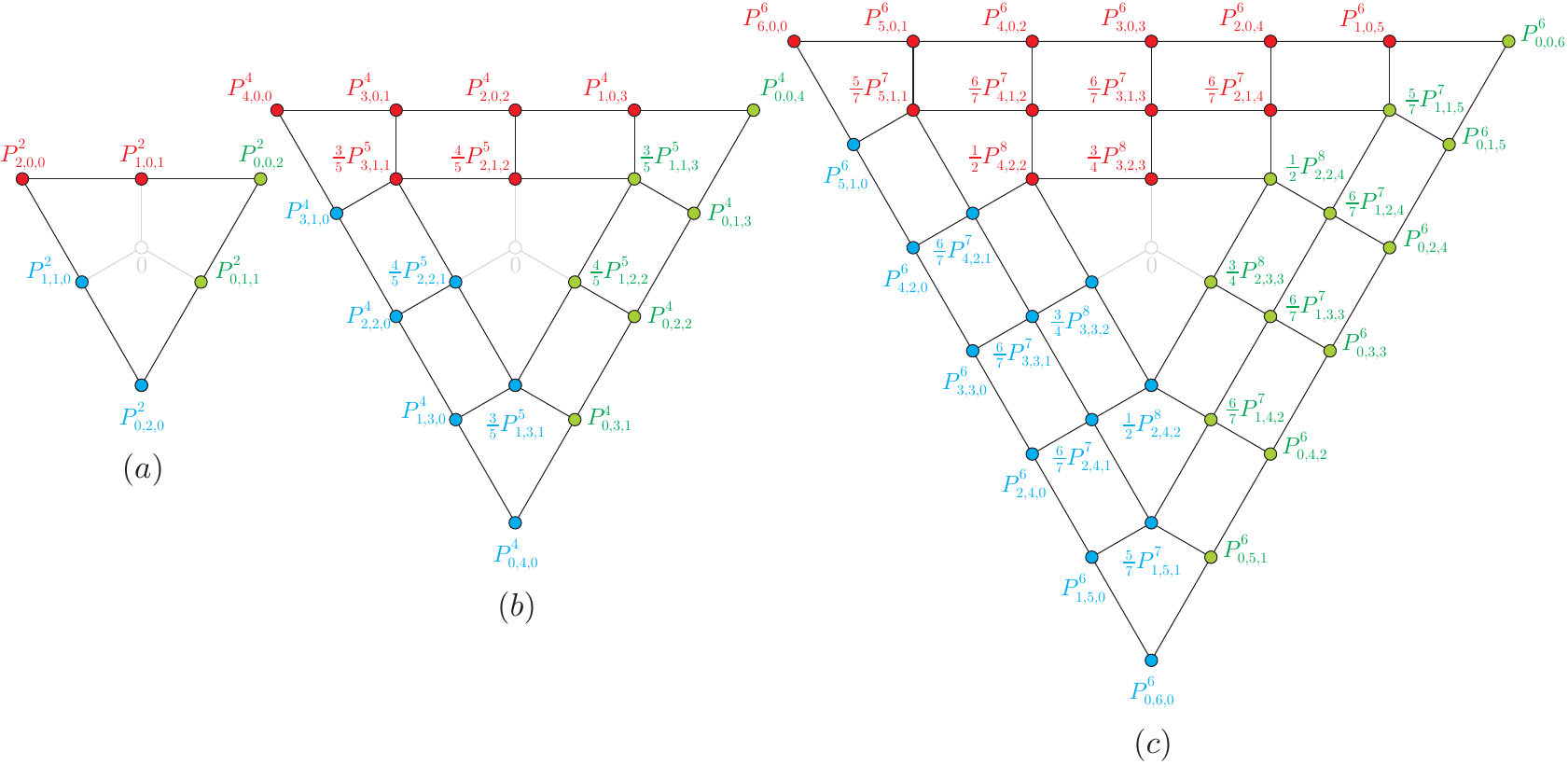}
\caption{Cases (\emph{a}), (\emph{b}) and (\emph{c}) illustrate the limiting case $\alpha\to 0$ of the first, second and third order constrained trivariate normalized trigonometric bases of type (\ref{normalized_united_system}), respectively. ($P^{m}_{d,e,f}$ denotes the constrained trivariate Bernstein polynomial $\frac{m!}{d!e!f!}x^d y^e \left(1-x-y\right)^f$ of degree $m$, where $x\in\left[0,1\right]$, $y\in\left[0,x\right]$ and $d,e,f\in\left\{0,1,\ldots,m\right\}$ such that $d+e+f=m$. Constrained triavariate normalized trigonometric basis functions associated with the innermost node of the graphs above vanish when $\alpha\to 0$.)}
\label{fig:limit_cases_n_1_2_3}
\end{figure}

Moreover, the control nets of the original first, second and third order triangular trigonometric patches can be converted to control nets that describe classical polynomial quadratic, quintic and octic triangular B\'ezier patches, respectively, when $\alpha \to 0$. Due to symmetry, we only list for $n=2$ and $3$ the position of B\'ezier points
\begin{align*}
\mathbf{p}_{5,0,0}  & =\Red{\mathbf{r}_{4,4,0}},\\
\mathbf{p}_{4,0,1}  & =\frac{1}{5}\Red{\mathbf{r}_{4,4,0}}+\frac{4}{5}\Red{\mathbf{r}_{4,3,0}},\,
\mathbf{p}_{3,0,2}    =\frac{2}{5}\Red{\mathbf{r}_{4,3,0}}+\frac{3}{5}\Red{\mathbf{r}_{4,2,0}},\,
\mathbf{p}_{2,0,3}    =\frac{3}{5}\Red{\mathbf{r}_{4,2,0}}+\frac{2}{5}\Red{\mathbf{r}_{4,1,0}},\,
\mathbf{p}_{1,0,4}    =\frac{4}{5}\Red{\mathbf{r}_{4,1,0}}+\frac{1}{5}\Green{\mathbf{g}_{4,4,0}},\\
& \\
\mathbf{p}_{3,1,1}  & =\frac{1}{5}\Red{\mathbf{r}_{4,3,0}}+\frac{3}{5}\Red{\mathbf{r}_{4,3,1}}+\frac{1}{5}\Blue{\mathbf{b}_{4,1,0}},\,
\mathbf{p}_{2,1,2}    =\frac{1}{5}\Red{\mathbf{r}_{4,2,0}}+\frac{4}{5}\Red{\mathbf{r}_{4,2,1}}
\end{align*}
and
\begin{align*}
\mathbf{p}_{8,8,0}  & =\Red{\mathbf{r}_{6,6,0}},\\
\mathbf{p}_{7,0,1}  & =\frac{1}{4}\Red{\mathbf{r}_{6,6,0}}+\frac{3}{4}\Red{\mathbf{r}_{6,5,0}},\,
\mathbf{p}_{6,0,2}    =\frac{1}{28}\Red{\mathbf{r}_{6,6,0}}+\frac{3}{7}\Red{\mathbf{r}_{6,5,0}}+\frac{15}{28}\Red{\mathbf{r}_{6,4,0}},\,
\mathbf{p}_{5,0,3}    =\frac{3}{28}\Red{\mathbf{r}_{6,5,0}}+\frac{15}{28}\Red{\mathbf{r}_{6,4,0}}+\frac{5}{14}\Red{\mathbf{r}_{6,3,0}},\\
\mathbf{p}_{4,0,4}  & =\frac{3}{14}\Red{\mathbf{r}_{6,4,0}}+\frac{4}{7}\Red{\mathbf{r}_{6,3,0}}+\frac{3}{14}\Red{\mathbf{r}_{6,2,0}},\,
\mathbf{p}_{3,0,5}    =\frac{5}{14}\Red{\mathbf{r}_{6,3,0}}+\frac{15}{28}\Red{\mathbf{r}_{6,2,0}}+\frac{3}{28}\Red{\mathbf{r}_{6,1,0}},\,
\mathbf{p}_{2,0,6}    =\frac{15}{28}\Red{\mathbf{r}_{6,2,0}}+\frac{3}{7}\Red{\mathbf{r}_{6,1,0}}+\frac{1}{28}\Green{\mathbf{g}_{6,6,0}},\\
\mathbf{p}_{1,0,7}  & =\frac{3}{4}\Red{\mathbf{r}_{6,1,0}}+\frac{1}{4}\Green{\mathbf{g}_{6,6,0}},\\
& \\
\mathbf{p}_{6,1,1}  & =\frac{1}{28}\Red{\mathbf{r}_{6,6,0}}+\frac{3}{14}\Red{\mathbf{r}_{6,5,0}}+\frac{15}{28}\Red{\mathbf{r}_{6,5,1}}+\frac{3}{14}\Blue{\mathbf{b}_{6,1,0}},\\
\mathbf{p}_{5,1,2}  & =\frac{1}{14}\Red{\mathbf{r}_{6,5,0}}+\frac{5}{28}%
\Red{\mathbf{r}_{6,4,0}}+\frac{5}{28}\Red{\mathbf{r}_{6,5,1}}+\frac{15}{28}\Red{\mathbf{r}%
_{6,4,1}}+\frac{1}{28}\Blue{\mathbf{b}_{6,1,0}},\\
\mathbf{p}_{4,1,3}  & =\frac{3}{28}\Red{\mathbf{r}_{6,4,0}}+\frac{1}{7}%
\Red{\mathbf{r}_{6,3,0}}+\frac{9}{28}\Red{\mathbf{r}_{6,4,1}}+\frac{3}{7}\Red{\mathbf{r}%
_{6,3,1}},\\
\mathbf{p}_{3,1,4}  & =\frac{1}{7}\Red{\mathbf{r}_{6,3,0}}+\frac{3}{28}%
\Red{\mathbf{r}_{6,2,0}}+\frac{3}{7}\Red{\mathbf{r}_{6,3,1}}+\frac{9}{28}\Red{\mathbf{r}%
_{6,2,1}},\\
\mathbf{p}_{2,1,5}  & =\frac{5}{28}\Red{\mathbf{r}_{6,2,0}}+\frac{1}{14}%
\Red{\mathbf{r}_{6,1,0}}+\frac{15}{28}\Red{\mathbf{r}_{6,1,1}}+\frac{1}{28}\Green{\mathbf{g}%
_{6,5,0}}+\frac{5}{28}\Green{\mathbf{g}_{6,5,1}},\\
\mathbf{p}_{1,1,6}  & =\frac{3}{14}\Red{\mathbf{r}_{6,1,0}}+\frac{1}{28}%
\Green{\mathbf{g}_{6,6,0}}+\frac{3}{14}\Green{\mathbf{g}_{6,5,0}}+\frac{15}{28}\Green{\mathbf{g}%
_{6,5,1}},\\
& \\
\mathbf{p}_{4,2,2}  & =\frac{1}{28}\Red{\mathbf{r}_{6,4,0}}+\frac{3}{14}%
\Red{\mathbf{r}_{6,4,1}}+\frac{1}{28}\Blue{\mathbf{b}_{6,2,0}}+\frac{3}{14}\Blue{\mathbf{b}%
_{6,2,1}}+\frac{1}{2}\Red{\mathbf{r}_{6,4,2}},\\
\mathbf{p}_{3,2,3}  & =\frac{1}{28}\Red{\mathbf{r}_{6,2,0}}+\frac{3}{14}%
\Red{\mathbf{r}_{6,3,1}}+\frac{3}{4}\Red{\mathbf{r}_{6,3,2}},%
\end{align*}
respectively. Fig.\ \ref{fig:conversion_of_limit_cases_n_1_2_3} shows all B\'ezier points obtained by the evaluation of convex combinations describing these conversion processes.  

\begin{figure}[!htb]
\center
\includegraphics[scale = 1]{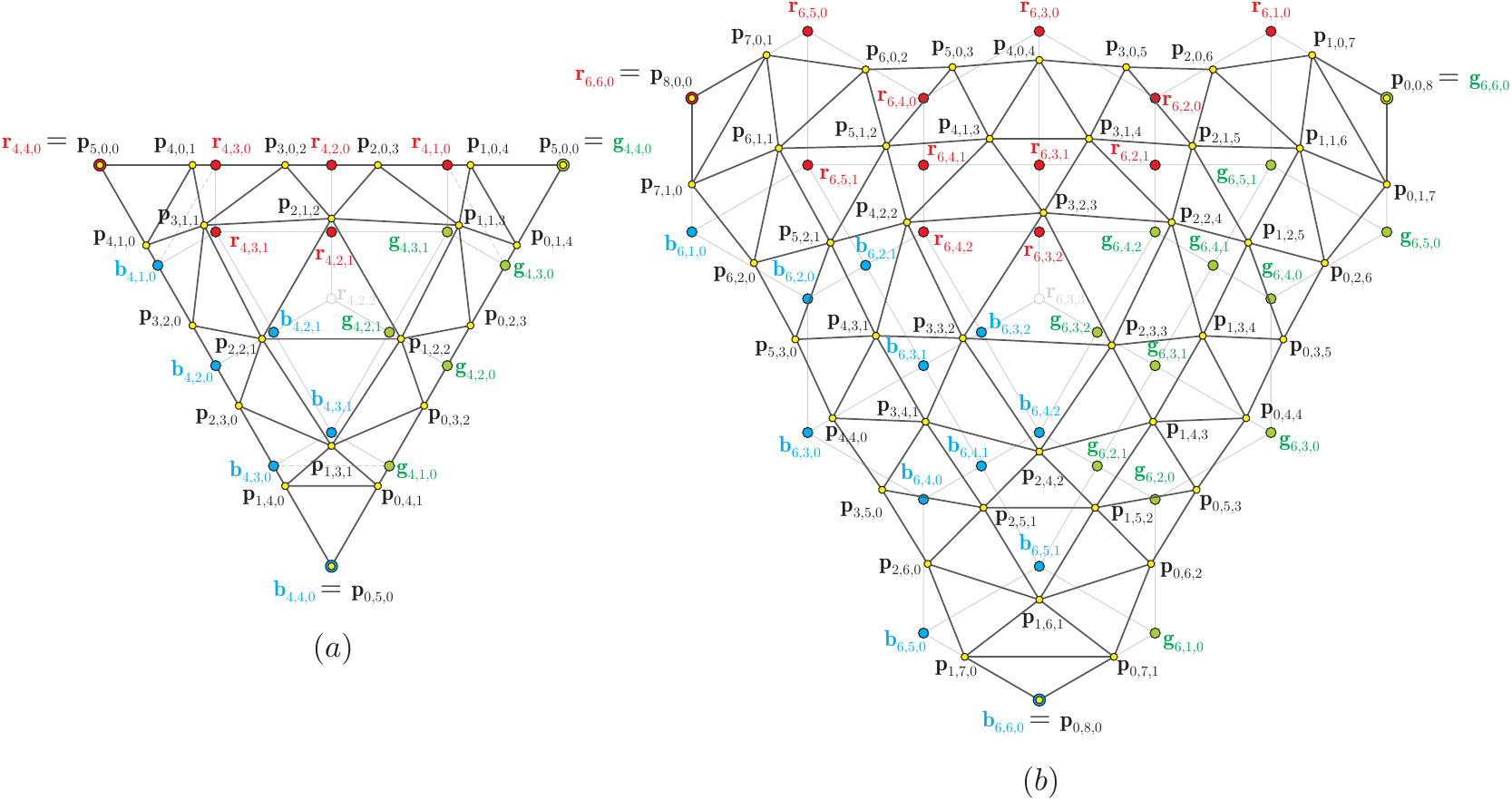}
\caption{Conversion of control nets of original second and third order triangular trigonometric patches to that of (\emph{a}) quintic and (\emph{b}) octic triangular B\'ezier patches, respectively, when $\alpha \to 0$.}
\label{fig:conversion_of_limit_cases_n_1_2_3}
\end{figure}

\begin{question}
\label{que:limiting_case}
Is the limiting case $\alpha \to 0$ of the (rational) triangular trigonometric patch of order $n$ a (rational) triangular B\'ezier patch of degree $3n-1$ defined on the unit simplex? If so, then how can we convert the control net of the original (rational) triangular trigonometric patch to that of the (rational) triangular B\'ezier patch obtained in this limiting case?
\end{question}

\section{Final remarks and future work}

The constrained trivariate counterpart of the univariate normalized
B-basis\ (\ref{univariate_basis}) of the vector space
(\ref{truncated_Fourier_vector_space}) of first and second order
trigonometric polynomials were introduced in recent articles \cite{Wang2010a}
and \cite{Wang2010b}, respectively. By means of a multiplicatively weighted oriented graph and an equivalence
relation we were able to provide a natural description of the normalized
linearly independent constrained trivariate function system
(\ref{normalized_united_system}) of dimension $\delta_{n}=3n\left(n+1\right)+1$ that spans the same vector space of functions as the
constrained trivariate extension of the canonical basis of truncated Fourier
series of order $n\in\mathbb{N}$. The proposed extension was applied to define (rational) triangular trigonometric patches of order $n$.

In Section \ref{sec:open_problems} we have outlined some theoretical problems that will form our forthcoming research directions. We also intend to illustrate the applicability of
the proposed (rational) triangular trigonometric patches by providing $\alpha$-dependent control point based formulas for order elevation and the exact description of triangular patches that lie on trigonometric
(rational) surfaces.

\section*{Acknowledgements}
\'{A}. R\'{o}th partially realized his research in the frames of the highly important National Excellence Program -- working out and operating an inland student and researcher support, identification number T\'AMOP 4.2.4.A/2-11-1-2012-0001. The project is realized with the help of European Union and Hungary subsidy and co-financing by the European Social Fund. \'{A}. R\'{o}th and A. Krist\'{a}ly have also been supported by the Romanian national grant CNCS-UEFISCDI/PN-II-RU-TE-2011-3-0047. I. Juh\'{a}sz carried out his research in the framework of the Center of Excellence of Mechatronics and Logistics at the University of Miskolc. The authors would like to thank the kind help of W.-Q. Shen who translated for them the
article \cite{Wang2010b} from Chinese to English.

\end{document}